\documentclass[10pt,reqno]{amsart}
\usepackage{latexsym,amsmath,amssymb,amscd}
\usepackage[all]{xy}
\usepackage{enumerate}
\usepackage
{hyperref}

\makeatletter
\newcommand\@dotsep{4.5}
\def\@tocline#1#2#3#4#5#6#7{\relax
	\ifnum #1>\c@tocdepth 
	\else
	\par \addpenalty\@secpenalty\addvspace{#2}%
	\begingroup \hyphenpenalty\@M
	\@ifempty{#4}{%
		\@tempdima\csname r@tocindent\number#1\endcsname\relax
	}{%
		\@tempdima#4\relax
	}%
	\parindent\z@ \leftskip#3\relax \advance\leftskip\@tempdima\relax
	\rightskip\@pnumwidth plus1em \parfillskip-\@pnumwidth
	#5\leavevmode\hskip-\@tempdima #6\relax
	\leaders\hbox{$\m@th
		\mkern \@dotsep mu\hbox{.}\mkern \@dotsep mu$}\hfill
	\hbox to\@pnumwidth{\@tocpagenum{#7}}\par
	\nobreak
	\endgroup
	\fi}
\makeatother 

\begin{document}


\makeatletter
\@addtoreset{figure}{section}
\def\thefigure{\thesection.\@arabic\c@figure}
\def\fps@figure{h,t}
\@addtoreset{table}{bsection}

\def\thetable{\thesection.\@arabic\c@table}
\def\fps@table{h, t}
\@addtoreset{equation}{section}
\def\theequation{
\arabic{equation}}
\makeatother

\newcommand{\bfi}{\bfseries\itshape}

\newtheorem{theorem}{Theorem}[section]
\newtheorem{corollary}[theorem]{Corollary}
\newtheorem{definition}[theorem]{Definition}
\newtheorem{example}[theorem]{Example}
\newtheorem{lemma}[theorem]{Lemma}
\newtheorem{notation}[theorem]{Notation}
\newtheorem{convention}[theorem]{Convention}
\newtheorem{proposition}[theorem]{Proposition}
\newtheorem{remark}[theorem]{Remark}
\numberwithin{equation}{section}

\renewcommand{\1}{{\bf 1}}
\newcommand{\Ad}{{\rm Ad}}
\newcommand{\Aut}{{\rm Aut}\,}
\newcommand{\ad}{{\rm ad}}
\newcommand\BHLG{\mathbb{BHLG}}
\newcommand\BLA{\mathbb{BLA}}
\newcommand{\bm}{{\bf m}}
\newcommand{\card}{{\rm card}}
\newcommand{\Ci}{{\mathcal C}^\infty}
\newcommand{\Der}{{\rm Der}\,}
\newcommand{\de}{{\rm d}}
\newcommand{\ee}{{\rm e}}
\newcommand{\End}{{\rm End}\,}
\newcommand{\Fr}{{\mathsf{F}}}
\newcommand{\ev}{{\rm ev}}
\newcommand{\GL}{{\rm GL}}
\newcommand{\Gr}{{\rm Gr}}
\newcommand{\graf}{{\mathsf{G}}}
\newcommand{\Hom}{{\rm Hom}}
\newcommand{\hotimes}{\widehat{\otimes}}
\newcommand{\id}{{\rm id}}
\newcommand{\ie}{{\rm i}}
\newcommand{\gl}{{{\mathfrak g}{\mathfrak l}}}
\newcommand{\Ker}{{\rm Ker}\,}
\newcommand{\Lie}{\text{\bf L}}
\newcommand{\pr}{{\rm pr}}
\newcommand{\Ran}{{\rm Ran}\,}
\renewcommand{\Re}{{\rm Re}\,}
\newcommand{\spa}{{\rm span}\,}
\newcommand{\Tr}{{\rm Tr}\,}
\newcommand{\U}{{\rm U}}

\newcommand{\Ac}{{\mathcal A}}
\newcommand{\Bc}{{\mathcal B}}
\newcommand{\Cc}{{\mathcal C}}
\newcommand{\Dc}{{\mathcal D}}
\newcommand{\Ec}{{\mathcal E}}
\newcommand{\Fc}{{\mathcal F}}
\newcommand{\Gc}{{\mathcal G}}
\newcommand{\Hc}{{\mathcal H}}
\newcommand{\Jc}{{\mathcal J}}
\newcommand{\Kc}{{\mathcal K}}
\newcommand{\Lc}{{\mathcal L}}
\renewcommand{\Mc}{{\mathcal M}}
\newcommand{\Nc}{{\mathcal N}}
\newcommand{\Oc}{{\mathcal O}}
\newcommand{\Pc}{{\mathcal P}}
\newcommand{\Rc}{{\mathcal R}}
\newcommand{\Sc}{{\mathcal S}}
\newcommand{\Tc}{{\mathcal T}}
\newcommand{\Vc}{{\mathcal V}}
\newcommand{\Uc}{{\mathcal U}}
\newcommand{\Xc}{{\mathcal X}}
\newcommand{\Yc}{{\mathcal Y}}
\newcommand{\Zc}{{\mathcal Z}}
\newcommand{\Wc}{{\mathcal W}}

\newcommand{\Ag}{{\mathfrak A}}
\newcommand{\Bg}{{\mathfrak B}}
\newcommand{\Cg}{{\mathfrak C}}
\newcommand{\Fg}{{\mathfrak F}}
\newcommand{\Gg}{{\mathfrak G}}
\newcommand{\Ig}{{\mathfrak I}}
\newcommand{\Jg}{{\mathfrak J}}
\newcommand{\Lg}{{\mathfrak L}}
\newcommand{\Mg}{{\mathfrak M}}
\newcommand{\Pg}{{\mathfrak P}}
\newcommand{\Sg}{{\mathfrak S}}
\newcommand{\Xg}{{\mathfrak X}}
\newcommand{\Yg}{{\mathfrak Y}}
\newcommand{\Zg}{{\mathfrak Z}}

\newcommand{\ag}{{\mathfrak a}}
\newcommand{\bg}{{\mathfrak b}}
\newcommand{\dg}{{\mathfrak d}}
\renewcommand{\gg}{{\mathfrak g}}
\newcommand{\hg}{{\mathfrak h}}
\newcommand{\kg}{{\mathfrak k}}
\newcommand{\mg}{{\mathfrak m}}
\newcommand{\n}{{\mathfrak n}}
\newcommand{\og}{{\mathfrak o}}
\newcommand{\pg}{{\mathfrak p}}
\newcommand{\sg}{{\mathfrak s}}
\newcommand{\tg}{{\mathfrak t}}
\newcommand{\ug}{{\mathfrak u}}
\newcommand{\zg}{{\mathfrak z}}

\newcommand{\bs}{\mathbf{s}}
\newcommand{\bt}{\mathbf{t}}

\newcommand{\BB}{\mathbb B}
\newcommand{\CC}{{\mathbb C}}
\newcommand{\EE}{\mathbb E}
\newcommand{\FF}{\mathbb F}
\newcommand{\HH}{{\mathbb H}}
\newcommand{\MM}{{\mathbb M}}
\newcommand{\NN}{\mathbb N}
\newcommand{\QQ}{\mathbb Q}
\newcommand{\RR}{{\mathbb R}}
\newcommand{\TT}{{\mathbb T}}
\newcommand{\ZZ}{\mathbb Z}

\newcommand{\tto}{\rightrightarrows}

\def\H{\textbf{\rm H}}
\def\Q{\textbf{\rm Q}}

\title[Smooth Banach structure on orbit spaces]{Smooth Banach structure on orbit spaces and leaf spaces} 

\author[D. Belti\c t\u a, F. Pelletier]{Daniel Belti\c t\u a, Fernand Pelletier} 

\address{Institute of Mathematics ``Simion Stoilow'' 
	of the Romanian Academy, 
	P.O. Box 1-764, Bucharest, Romania}
\email{Daniel.Beltita@imar.ro, beltita@gmail.com}

\address{Unit\'e Mixte de Recherche 5127 CNRS, Universit\'e  de Savoie Mont Blanc, Laboratoire de Math\'ematiques (LAMA),Campus Scientifique,  73370 Le Bourget-du-Lac, France}
\email{fernand.pelletier@univ-smb.fr, fer.pelletier@gmail.com}

\date{
	\today
}

\begin{abstract} 
We investigate the quotients of Banach manifolds with respect to free actions of pseudogroups of local diffeomorphisms. These quotient spaces are called \H-manifolds 
since the corresponding simply transitive action of the pseudogroup on its orbits is regarded as a homogeneity condition.
The importance of these structures stems from the fact that for every  regular foliation without holonomy of a Banach manifold, the corresponding leaf space has the natural structure of an \H-manifold. 
This is our main technical result, and one of its remarkable consequences is an infinite-dimensional version of Sophus Lie's third fundamental theorem, to the effect that every real Banach-Lie algebra can be integrated to an \H-group, that is, a group object in the category of \H-manifolds. 
In addition to these general results we discuss a wealth of examples of \H-groups which are not Banach-Lie groups.
\end{abstract}

\maketitle

\tableofcontents

\section{Introduction}

The problem of adapting Sophus Lie's third fundamental theorem to  infinite-dimensional Lie algebras received much attention, 
and the research in this area is streamlined by two main questions: 
Firstly, due to well-known counterexamples (e.g., \cite{DL66}) there arises the question of finding sufficient conditions 
ensuring that a given Banach-Lie algebra is enlargible, in the sense that it is isomorphic to the Lie algebra of some Banach-Lie group. 
See for instance \cite{Swi}, \cite{Pes93a}, \cite{Pes93b}, \cite{Pes95}, \cite{Be06}, 
and particularly \cite{Ne02} for a detailed discussion of this question in terms of central extensions. 
Secondly, there is the question of finding categories of group-like objects endowed with Lie functors into categories of topological Lie algebras, which should allow one to recover (``integrate'') in this way various classes of Lie algebras. 
An interesting result of this type was obtained in \cite[Th. 6.5]{Wo11}, 
to the effect that the \'etale Lie 2-groups recover all the Mackey-complete locally exponential Lie algebras in the sense of \cite{Ne06}. 

The present paper belongs to the second of the research directions mentioned above, and we focus on obtaining a category of groups, rather than group-like objects, a direction that was also explored for instance in \cite{Woj03} and \cite{Woj06}. 
Here we continue our earlier study of non-smooth quotients of Banach-Lie groups from \cite{BPZ19} in the framework of \H-manifolds 
---a certain category of 
quotients of smooth Banach manifolds for which the corresponding quotient maps satisfy a suitable homogeneity condition. 
One of our main results is that the group objects 
in that category 
are sufficient for solving the integrability problem for arbitrary real Banach-Lie algebras (Theorem~\ref{Lie_III}). 

Roughly speaking, an \H-manifold is the quotient of a Banach manifold with respect to the free action of a pseudogroup of local diffeomorphisms (Proposition~\ref{actionG}). 
In the special of case of finite-dimensional manifolds, this notion goes back to \cite{Pra}, and the first examples of \H-manifolds were the leaf spaces of regular foliations without holonomy. 

A special class of \H-manifolds, the so-called \Q-manifolds, was introduced in \cite{Ba73}, partially motivated by the need to provide a framework for the study of quotient groups of finite-dimensional Lie groups with respect to non-closed normal subgroups, as for instance the quotient group~$\RR/\QQ$. 
The natural extension of \Q-manifolds to the setting of manifolds modeled on Banach spaces was later explored in \cite{Pla} and \cite{Pl80b}, 
with a view to addressing the integrability problem for separable Banach-Lie algebras. 
However, lacking a sufficiently well developed Lie theory for \Q-groups, the integrability of separable Banach-Lie algebras was only partially established. 
(See Remark~\ref{critique} below for a more detailed discussion in this connection.)

It is worth mentioning that quotient spaces such as the leaf spaces of regular foliations were often studied using  diffeological structures (cf. \cite{IZ13} and also \cite{Mag}) or Fr\"olicher structures (cf. \cite{Fr80} and \cite{BIZKW}). 
However, when the corresponding quotient map has a suitable homogeneous transverse structure, the quotient space carries a richer structure from a differential geometrical point of view. 
The relation between these approaches and the \H-structures is discussed in some more detail in Appendix~\ref{AppA}, and we refer to \cite{BIZKW} for the relation between diffeological structures and Fr\"olicher structures.

In order to point out one of the peculiarities of our study, we recall that the category of finite-dimensional Lie groups, and even of some infinite-dimensional Lie groups, is a \emph{full} subcategory of the category of topological groups, that is, every continuous homomorphism is automatically smooth, cf. for instance \cite[Th. IV.1.18]{Ne06} for the case of locally exponential Lie groups (which includes  the Banach-Lie groups). 
In particular, the smooth structure of a Lie group is uniquely determined by its underlying topology. 
In stark contrast to that situation, the topology of the \Q-manifolds, and even more of the \H-manifolds, plays a rather limited role in their study 
(cf.  \cite{Ba73} and \cite{BPZ19}). 
In fact, the quotient of every Lie group by any countable subgroup is a \Q-group, as noted in \cite{Ba73}. 
Hence the natural topology of a \Q-group could be trivial in simple examples such as the quotient group $\RR/\QQ$ that we already mentioned above. 
It is for this reason that our techniques rely a lot on differential geometry, particularly the transverse structure of certain foliations, 
and we have rather little use of the topological methods that are often used in Lie theory. 

The structure of this paper is as follows: 
In Section~\ref{Sect2}, we develop the basic aspects of the \H-manifolds modeled on Banach spaces. 
In Section~\ref{Sect3} we obtain the main result to the effect that, loosely speaking, if a regular foliation of a Banach manifold is without holonomy, then its leaf space has the natural structure of an \H-manifold (Theorem~\ref{QuotientBanachManifold}). 
In Section~\ref{Sect4}, we develop some basic Lie theory for the \H-groups, 
in particular the appropriate versions of Sophus Lie's fundamental theorems I and III. 
Specifically, we construct the Lie functor from the category of \H-groups to the category of real Banach-Lie algebras (Theorem~\ref{Lie_I}) 
and we prove that every real Banach-Lie algebra arises from an \H-group (Theorem~\ref{Lie_III}). 
Then, in Section~\ref{Sect5} we develop methods to construct quotients of Banach-Lie groups which are \Q-groups but not Banach-Lie groups (Theorem~\ref{pseudo_discr_cor2}). 
Finally, in Appendix~\ref{AppA} we proceed to a brief comparison between the theory of \H-manifolds and earlier theories of differentiability beyond the realm of smooth manifolds, particularly the Fr\"olicher spaces and the diffeological spaces. 

\subsection*{General notation} 
The graph of any mapping $h\colon U\to U'$ is denoted by 
$$\graf_h:=\{(x,h(x))\in U\times U'\mid x\in U\}\subseteq U\times U'.$$

\section{\H-structures
}
\label{Sect2}
In this section we introduce the pre-\H-structures and \H-structures modeled on Banach spaces and we explore some of their basic features such as a characterization in terms of the graph of the corresponding equivalence relation (Proposition~\ref{characterization}) and the nontrivial fact that the direct product of two \H-atlases is again an \H-atlas (Corollary~\ref{characterization_cor1}). 
We then introduce the tangent space of a pre-\H-manifold and the tangent vector fields, which are needed in the construction of the Lie functor on the category of group \H-manifolds in Section~\ref{Sect4}.
\subsection{n.n.H. Banach manifolds in the Bourbaki sense}

We briefly recall here some classical notions on not-necessarily-Hausdorff (for short n.n.H.) manifolds modeled on Banach spaces that may be different for different connected components of the manifold under consideration, as in \cite{Bou3}.

 A \emph{$\Ci$-atlas}   on a set $M$   is a family $\{(U_\alpha, u_\alpha) \}_{\alpha\in A}$ of subsets $U_\alpha$ of $M$ 
and maps $u_\alpha$ from $U_\alpha$ to a Banach space $\mathbb{M}_\alpha$ such that:
\begin{itemize}
\item  $u_\alpha$ is a bijection of $U_\alpha$ onto an open  subset of   $\mathbb{M}_\alpha$ for all $\alpha\in A$;

\item $M=\bigcup\limits_{\alpha\in A}U_\alpha$;

\item if  $\alpha,\beta\in A$ and  $U_{\alpha\beta}:=U_\alpha\cap U_\beta\not=\emptyset$, 
then  $u_{\alpha\beta}:=u_\alpha\circ u_\beta^{-1}\vert_{u_\beta(U_{\alpha\beta})}\colon u_\beta(U_{\alpha\beta})\to
u_\alpha(U_{\alpha\beta})$ is a $\Ci$-map.
\end{itemize}
As usual, one has the notion of equivalent  $\Ci$-atlases on  $M$, 
as well as refinement of a given $\Ci$-atlas, 
and every $\Ci$-atlas is compatible with a maximal $\Ci$-atlas.
Such a maximal atlas 
defines a topology on $M$ which in general fails to have the Hausdorff property.

\begin{definition}\label{Banach} 
\normalfont
A maximal  $\Ci$-atlas on $M$  is called a \emph{not-necessarily-Hausdorff  Banach manifold structure} on $M$, 
for short a \emph{n.n.H. Banach manifold}. 
This structure is called a \emph{Hausdorff Banach manifold  structure} on $M$ 
(for short a \emph{Banach manifold}) 
if the topology defined by this atlas is a Hausdorff topology.
\end{definition}

\begin{remark}\label{maSHausdorffSetInnnH}  
\normalfont 
If $M$ is a n.n.H.  Banach manifold, 
then its corresponding topology is~$T_1$ hence  each finite subset of $M$ is closed. 
Moreover, since $M$ is locally Hausdorff, by Zorn's Lemma,  there exists a maximal open dense subset $M_0$ of $M$ which is an open Banach manifold in $M$.  
(See for instance \cite[Lemma 4.2]{BaGa}.)
\end{remark}

It is clear that the classical construction of the tangent bundle $TM$ of a  Banach manifold $M$ 
can be applied to a n.n.H. Banach manifold and so we get again a n.n.H.
 Banach manifold $TM$. 
 
 A  \emph{weakly immersed n.n.H. Banach submanifold} 
 (resp., an \emph{immersed n.n.H. Banach submanifold}) $N$ of a n.n.H. Banach manifold $M$ is an injective smooth map $\iota\colon N\to M$  such that $T_x\iota\colon T_xN\to T_{\iota(x)}M$ is injective (resp., is injective and its range is closed) for every $x\in N$. 
 A weakly immersed n.n.H. Banach submanifold $N$ of $M$ is called \emph{closed Banach submanifold} (resp. \emph{split Banach submanifold}) if $N$ is a closed subset of $M$ (resp.\  $T\iota(T_xN)$ is closed and complemented in  $T_{\iota(x)}M$ for all $x\in N$). 
 A weakly immersed Banach submanifold $N$ of $M$ 
 is called an \emph{initial submanifold} if 
 for every Banach manifold $P$ and every smooth mapping $f\colon P\to M$ with $f(P)\subseteq \iota(N)$ there exists a smooth mapping $f_0\colon P\to N$ with $\iota\circ f_0=f$. 
 
 More generally, a (not necessarily injective) smooth map $\psi\colon P\to M$ is called a \emph{split immersion} if the bounded linear operator $T_x\psi\colon T_xP\to T_{\psi(x)}M$ is injective and its range is a split closed linear subspace of $T_{\psi(x)}M$ for every $x\in P$.

A  \emph{submersion} $p\colon N\to M$ between two n.n.H. Banach manifolds is a surjective smooth map such that $Tp(T_xN)=T_{p(x)}M$ and $\Ker T_xp$ is complemented in $T_xN$ for each $x\in N$. 

A \emph{locally trivial fibration}  $p\colon N\to M$ is a submersion such that for each $x\in M$ there exist an open neighbourhood $U$ of $x$ and a diffeomorphism $\Phi\colon p^{-1}(U)\to U\times p^{-1}(x)$ such that $p_1\circ\Phi=p$ where $p_1=U\times p^{-1}(S)\to U$ is the canonical projection. 
If the basis is connected then all the fibers are diffeomorphic. 

A \emph{n.n.H Banach  bundle} is a locally trivial fibration  $\pi\colon  \mathcal{A}\to M$ whose  fiber is a Banach space. Of course $\mathcal{A}$ is Hausdorff if and only if $M$ is so. Again of course on each connected components  $\mathcal{A}_\alpha$ of $\mathcal{A}$ the fibers are isomorphic to a common  Banach space  (called {\it the typical fiber}) but this model can change from a connected  component to an another one. 
All classical  notions of n.n.H. Banach bundle morphisms, n.n.H. Banach bundle isomorphisms, n.n.H.Banach subbundles, regular foliations etc.\  are  clear.

\subsection{Definition and first properties of a (pre-)\H-structure on a set}
For the rest of this section,  unless otherwise specified,  $M$ is a    Banach manifold   (not necessarily Hausdorff) modeled on a Banach space $\mathbb{M}$. 
In the same way as Pradines in \cite{Pra} 
we introduce the following definitions.

\begin{definition}\label{pre_H_def}
	\normalfont
	Let $\pi\colon M\to S$ be a map from a n.n.H. Banach manifold to an arbitrary set. 
	To this mapping we associate three objects: 
	\begin{enumerate}[{\rm(a)}]
		\item The pseudogroup $\Gamma(\pi)$ consisting of the 
		\emph{transverse diffeomorphisms} associated to the triple $(M, \pi, S)$, that is, diffeomorphisms $h\colon U\to U'$ satisfying
		$\pi\circ h=\pi\vert_U$, where $U,U'\subseteq M$ are open subsets.
		\item The groupoid $\widetilde{\Gamma(\pi)}\mathop{\tto}M$ of germs of $\Gamma(\pi)$, whose source/target maps are denoted by $\bs,\bt\colon \widetilde{\Gamma(\pi)}\to M$. 
		\item The equivalence relation 
		$\Rc_\pi:=\{(x,y)\in M\times M\mid \pi(x)=\pi(y)\}$.  
	\end{enumerate}	
It is easily seen that for every $\gamma\in\widetilde{\Gamma(\pi)}$ one has $(\bt(\gamma),\bs(\gamma))\in\Rc_\pi$, hence one has the mapping 
\begin{equation*}
(\bt,\bs)\colon  \widetilde{\Gamma(\pi)}\to \Rc_\pi, \quad \gamma\mapsto (\bt(\gamma),\bs(\gamma)).
\end{equation*}
For every 
subpseudogroup $\Gamma\subseteq\Gamma(\pi)$ (cf. \cite[Def. 1.1.2]{Wal}), 
its set of germs $\widetilde{\Gamma}$ is a subgroupoid of the above groupoid $\widetilde{\Gamma(\pi)}\mathop{\tto}M$ 
and we have the restricted mapping 
\begin{equation}\label{pre_H_def_eq}
	(\bt,\bs)\vert_{\widetilde{\Gamma}}\colon\widetilde{\Gamma}\to \Rc_\pi, \quad \gamma\mapsto (\bt(\gamma),\bs(\gamma)).
\end{equation}
The above 
quadruple
$(M, \pi, \Gamma, S)$ 
is called a \emph{pre-\H-chart}  on $S$ 
if the mapping \eqref{pre_H_def_eq} is surjective. 

Moreover, we say that a pre-\H-chart 
$(M, \pi, \Gamma, S)$
is an \emph{\H-chart} 
if the mapping~\eqref{pre_H_def_eq} is bijective. 

A (pre-)\H-chart 
$(M, \pi, \Gamma, S)$ 
is called a (\emph{pre-})\emph{\H-atlas} if the mapping $\pi$ is surjective. 
\end{definition} 

In Definition~\ref{pre_H_def}, it is straightforward to check that $\Gamma(\pi)$ satisfies the axioms of a pseudogroup of transformations from \cite[Ch. I, \S 1]{KN63} or \cite[Def. 1.1.1]{Wal}. 

\begin{remark}\label{H_def}
\normalfont 
One can reformulate the notion of an \H-atlas (or \H-chart) in a groupoid-free language as follows. 
Let $S$ be a set.
A \emph{pre-\H-atlas}  on $S$ is a 
quadruple
	$(M, \pi, \Gamma, S)$ 
where $\pi$ is a surjective map from a n.n.H. Banach manifold $M$  on $S$ 
and $\Gamma\subseteq\Gamma(\pi)$ is a subpseudogroup 
with the following homogeneity property:
\begin{quotation}
	for all $x,y\in M$ with $\pi(x)=\pi(y)$, there exists 
	$h\in\Gamma$
	with $h(x)=y$. 
\end{quotation}
If in that homogeneity property the germ of 
$h\in\Gamma$ 
at $x$ is moreover uniquely determined by the conditions $h(x)=y$ and $\pi\circ h=\pi\vert_U$, then 
the 
quadruple	$(M, \pi, \Gamma, S)$ 
is an \emph{\H-atlas}. 
\end{remark}

For the following definition we need a specialization of the notion of equivalence of pseudogroups in the sense of \cite[\S 1.1]{Hae80}, \cite[Def. 1.1.2]{Hae85}, and \cite[\S 1.4]{Hae88}. 
Specifically, if $\Gamma$ and $\Gamma'$ are pseudogroups of diffeomorphisms on the n.n.H. manifolds $M$ and $M'$,  
then an \emph{\'etale morphism from $\Gamma$ to $\Gamma'$} is a maximal set $\Phi$ of diffeomorphisms from  open subsets of $M$ onto open subsets of $M'$, satisfying the following conditions: 
\begin{enumerate}[{\rm(i)}]
	\item If $\varphi\in\Phi$, $h\in\Gamma$, and $h'\in\Gamma'$, then $h'\circ\varphi\circ h\in \Phi$. 
	\item The union of the domains of the elements of $\Phi$ is equal to $M$. 
	\item If $\varphi,\psi\in\Phi$, then $\varphi\circ \psi^{-1}\in\Gamma'$. 
\end{enumerate}
If moreover the set $\Phi^{-1}:=\{\varphi^{-1}\mid\varphi\in\Phi\}$ is an \'etale morphism from $\Gamma'$ to $\Gamma$, then $\Phi$ is called an \emph{equivalence of pseudogroups} 
and the pseudogroups $\Gamma$ and $\Gamma'$ are called \emph{equivalent}. 
If this is the case, then the set $\Gamma\cup\Gamma'\cup\Phi\cup\Phi'$ is a pseudogroup on the disjoint union $M\sqcup M'$. 
See \cite[\S 2]{AM06} and \cite[\S 2]{AM08} for more details on this construction. 

\begin{definition}  
\label{compatib}
\normalfont
If $S$ is a set, then two pre-\H-charts 
	$(M, \pi, \Gamma, S)$ and $(M', \pi', \Gamma', S)$ 
are called \emph{compatible} if 
there exists an equivalence of pseudogroups $\Phi$ 
	from $\Gamma$ to~$\Gamma'$ satisfying $\pi'\circ\varphi=\pi\vert_U$ for every element $\varphi\colon U\to V$ from~$\Phi$.
	
A (\emph{pre}-)\emph{\H-manifold structure} on $S$ is an equivalence class of  (pre-)\H-atlases on $S$. 
In this case~$S$ is called a (\emph{pre}-)\emph{\H-manifold}.
\end{definition}

\begin{remark}\label{QuotientTopology} 
\normalfont 
Since an \H-manifold $S$  is a quotient space of a Banach manifold, for the quotient topology, 
$S$ is locally connected and locally path connected (cf. \cite{Bou1}). 
In particular, if $S$ is connected, it is also path-connected.
\end{remark}

\begin{example}
\normalfont
For instance, for any Banach manifold $M$ the 
let us consider the pseudogroup $\Gamma_0$ generated by $\id_M$ in the sense of \cite[Def. 1.1.3]{Wal}. 
	Specifically, $\Gamma_0=\{\id_U\mid U\text{ open }\subseteq M\}$.   
	Then the quadruple $(M,\id_M,\Gamma_0,M)$ 
is an \H-atlas. 
\end{example}

\begin{example}\label{Q_ex}
	\normalfont
Let $\pi\colon M\to S$ be mapping from a Banach manifold to a set. 
We recall from \cite{Ba73} and \cite{Pla} (see also \cite{Mo75}, \cite{Mo77}, 
\cite[Def. 2.1]{BPZ19}, \cite{BP24}) that the mapping $\pi$ is called a \emph{\Q-chart} 
if it satisfies the following conditions: 
\begin{enumerate}[{\rm(a)}]
	\item\label{Q_ex_a} For all $x,y\in M$ with $\pi(x)=\pi(y)$, there exist 
	open subsets $U,U'\subseteq M$ with $x\in U$,  $y\in U'$, and a diffeomorphism $h\colon U\to U'$ with $h(x)=y$ and 
	$\pi(h(z))=\pi(z)$ for all $z\in U$.
	\item\label{Q_ex_b} For every Banach manifold $T$ and any smooth maps $f,h\colon T\to M$ with $\pi\circ f=\pi\circ h$, the set $\{t\in T\mid f(t)=h(t)\}$ is open in $T$. 
\end{enumerate} 
The triple $(M,\pi,S)$ is also called a \Q-chart, and it is callen a \Q-atlas if $\pi(M)=S$. 
If the above conditions \eqref{Q_ex_a}--\eqref{Q_ex_b} are satisfied, 
then the quadruple $(M,\pi,\Gamma(\pi),S)$ is an \H-chart 
since 
the above condition~\eqref{Q_ex_b} implies the uniqueness assertion in the homogeneity condition from Remark~\ref{H_def}. 
	Loosely speaking, every \Q-chart is an \H-chart, 
	since the pseudogroup~$\Gamma(\pi)$ in the above quadruple $(M,\pi,\Gamma(\pi),S)$ is uniquely determined by~$\pi$.
However, as shown by a simple example with $M=\RR^2$ in \cite{Pra}, the converse implication does not hold true in general. 
More specifically, if $(M,\pi,\Gamma,S)$ is an \H-chart, then $\pi\colon M\to S$ need not be a \Q-chart even if $\Gamma$ is the maximal pseudogroup, that is, $\Gamma=\Gamma(\pi)$. 
For the sake of completeness we give in Example~\ref{H-nonQ}  below a generalization of Pradines' example from \cite{Pra} in a Banach context. 
\end{example}

\begin{example}\label{H-nonQ}
\normalfont  
Let   $M$ be a real Banach space and select any closed linear subspace 
$M_0\subseteq M$. 
We denote $M_1:=M\setminus M_0$, which is open in~$M$. 
We define an equivalence relation on~$M$ by 
\begin{equation}\label{Eq_sim}
	x \mathcal{R} y \Longleftrightarrow \begin{cases}
	&\text{either }x\in M_1\text{ and } y=-x, \\
	&\text{or }  x\in M_0\text{ and } y=x.
	\end{cases}
\end{equation}
	We denote $S:=M/\mathcal{R}$ and let $\pi\colon M\to S$ its corresponding quotient map. 
	 
\noindent	\textbf{Claim 1:} \emph{If $\dim M\ge 2$, then 
the quadruple $(M,\pi,\Gamma(\pi),S)$}
	is an $\H$-atlas. 

To prove this claim we check the homogeneity property in Remark~\ref{H_def}. 
Let $x,y\in M$ arbitrary with $\pi(x)=\pi(y)$, that is, $x \mathcal{R} y$. 
If $x\in M_0$ then $y=x\in M_0$ and for any open subset $U\subseteq M_0$ with $x\in U$, defining $U':=U$ and $h:=\id_U\colon U\to U'$, we have $h(x)=y$ and $\pi\circ h=\pi\vert_U$. 
If however $x\in M_1$, then $y=-x\in M_1$ and there exists an open subset $U\subseteq M_1$ with $U\cap(-U)=\emptyset$. 
Then defining $U':=-U$ and $h:=-\id_U\colon U\to U'$, we have $h(x)=y$ and $\pi\circ h=\pi\vert_U$. 

The uniqueness of the germ of $h$ required in Remark~\ref{H_def} will follow as soon as we will prove the following fact: 
\emph{for any connected open subsets $U,U'\subseteq M$ and any homeomorphism $h\colon U\to U'$ satisfying $\pi\circ h=\pi\vert_U$ we have either $U=U'$ and $h=\id_U$, or $U'=-U$ and $h=-\id_U$}. 
In order to prove that fact, we denote $U_j:=U\cap M_j=\{x\in U\mid h(x)=(-1)^jx\}$ for $j=0,1$. 
Since $M=M_0\sqcup M_1$, we then obtain $U=U_0\sqcup U_1$, hence also 
\begin{equation}
\label{H-nonQ_eq2}
U\setminus\{0\}=(U_0\setminus\{0\})\sqcup U_1.
\end{equation}
Since $U\subseteq M$ is a connected open subset and $\dim M\ge 2$, it is easily checked that the open set $U\setminus\{0\}$ is connected as well. 
On the other hand, for $j=0,1$, we have $U_j=\{x\in U\mid h(x)=(-1)^jx\}$, 
hence $U_j$ is a closed subset of $U$, since $h$ is a continuous map. 
Then \eqref{H-nonQ_eq2} is a partition of the connected set $U\setminus\{0\}$ into two closed subsets. 
Therefore we have either $U\setminus\{0\}=U_0\setminus\{0\}$, 
which this easily implies $U=U_0$ and $h=\id_U$, or $U\setminus\{0\}=U_1$, 
which implies $h=-\id_U$. 
This completes the proof of Claim~1 above.

\noindent	\textbf{Claim 2:} \emph{If $M_0\subsetneqq M$, then $\pi:M\to S$ is not a \Q-atlas}. 

To prove this claim, we use the hypothesis $M_0\subsetneqq M$ 
to select a point  $u\in M\setminus M_0=M_1$ 
and we consider the smooth curves 
$$c_\pm\colon\RR\to M,\quad c_{\pm}(t)=\pm tu.$$
For every $t\in\RR$ we have $tu\in M_1$, hence $(-tu) \mathcal{R} (+tu)$. 
That is, $\pi\circ c_-=\pi\circ c_+$. 
On the other hand, we have  $\{t\in\RR\mid c_-(t) = c_+(t)\}=\{0\}$, 
which is not an open subset of $\RR$, and thus 
the condition~\eqref{Q_ex_b} in the definition of a \Q-atlas in Example~\ref{Q_ex} is not satisfied. 
\end{example}

\begin{remark}  
\normalfont 
As  suggested  in \cite{Pra}, the result of Example~\ref{H-nonQ} 
holds true under much more general assumptions. 
Thus, Claim~1 holds true  if $M$ is a Banach space with $\dim M\ge 2$ while $M_0\subseteq M$ is any closed subset that is invariant to the symmetry $x\mapsto -x$ and satisfies $0\in M_0$.  
Claim~2 needs that there exists $u\in M$ with $tu\in M\setminus M_0$ for every $t\in\RR\setminus\{0\}$. 
\end{remark}

\begin{example}
	\normalfont
For any \'etale map $f\colon N\to M$,  if 
$(M,\pi,\Gamma,S)$ 
is an \H-atlas on~$S$, then 
$(N,\pi,f^{-1}(\Gamma),S)$ 
is an \H-atlas as well. 
Here $f^{-1}(\Gamma)$ is the pseudogroup on~$N$ generated in the sense of \cite[Def. 1.1.3]{Wal} by the diffeomorphisms of the form $\sigma\circ h\circ f\vert_U$ where $U\subseteq N$ is an open subset for which $f\vert_U$ is injective, $h\in \Gamma$ with its domain containing $f(U)$, 
and $\sigma\colon h(f(U))\to N$ is a smooth cross-section of $f$, 
i.e., $f\circ\sigma=\id_{h(f(U))}$. 
(Compare the construction of an \'etale morphism of pseudogroups defined by an \'etale map in \cite[\S 1.4]{Hae88}.)
We also note that the set of diffeomorphisms so the \H-atlases 
$$\Phi:=\{f\vert_U\colon U\to f(U)\mid U\text{ open }\subseteq M\text{ with }f\vert_U\text{ injective }\}$$
is an equivalence of pseudogroups from $f^{-1}(\Gamma)$ to $\Gamma$, 
so the atlases $(M,\pi,\Gamma,S)$ and $(N,\pi,f^{-1}(\Gamma),S)$ are equivalent.

It follows that if $\{(U_\alpha,\varphi_\alpha)\}_{\alpha\in A}$ is an atlas for $M$,  and we define 
$\widehat{M}:=\bigsqcup\limits_{\alpha\in A} U_\alpha$ and $\widehat{\pi}\colon \widehat{M}\to S$ is defined by $\widehat{\pi}\vert_{ U_\alpha}=\pi\vert_{U_\alpha}$, then $(\widehat{M},\widehat{\pi},\Gamma(\widehat{\pi}), S)$ 
is 
an \H-atlas. 
Note that even if $M$ is not Hausdorff, $\widehat{M}$ is Hausdorff. 
Thus for any \H-manifold, there always exists an \H-atlas 
$(M,\pi,\Gamma,S)$ such that $M$ is Hausdorff. 
 Moreover, we may assume that each connected component of ${M}$ is simply connected.
\end{example}

\begin{example}
	\label{pre-H_prod}
	\normalfont
If the 
quadruple $(M_i,\pi_i,\Gamma_i,S_i)$ 
is a pre-\H-atlas for $i=1,2$, then the  
quadruple $(M_1\times M_2, \pi_1\times\pi_2, \Gamma_1\times\Gamma_2,S_1\times S_2)$
is clearly a pre-\H-atlas, too. 
The similar assertion for \H-atlases is less trivial and is the content of Corollary~\ref{characterization_cor1}. 
The direct product of two \Q-atlases is again a \Q-atlas, as noted in \cite[Rem. 2.4]{BPZ19}, 
and the case finite-dimensional \Q-atlases goes back to \cite[Ch. 1, \S 1, no. 6]{Ba73}. 
\end{example}

\begin{example}\label{open_subm}
	\normalfont
	Let $(M,\pi,\Gamma,S)$ 
	be a (pre-)\H-atlas and $S_0\subseteq S$ be a subset for which we denote $M_0:=\pi^{-1}(S_0)\subseteq M$. 
	If the subset $M_0\subseteq M$ is open, then it is straightforward to check that $(M_0,\pi\vert_{M_0}, \Gamma_{M_0}, S_0)$ 
	is again a (pre-)\H-atlas. 
	Here $\Gamma_{M_0}$ is the pseudogroup on $M_0$ generated by restrictions to $M_0$ of the elements of~$\Gamma$, 
		discussed in \cite[\S 1.1]{Hae80}. 
	
	Additionally, if $(M,\pi, S)$ is a \Q-atlas, then $(M_0,\pi\vert_{M_0}, S_0)$ is a \Q-atlas as well. 
	In this context, we refer to \cite[Prop. 2.9]{BPZ19} for a useful description of the quotient topology of \Q-manifolds in terms of local charts. 
\end{example}

\begin{definition}\label{H-smooth}
\normalfont 
Let $S_1$ and $S_2$ be pre-\H-manifolds. 
A map $f\colon S_1\to S_2$ is called  
a \emph{pre-\H-morphism}, or is said to be \emph{pre-\H-smooth}, 
or simply \emph{smooth}, 
 if for every $s_i\in S_i$  for $i=1,2$ with $f(s_1)=s_2$ there exist pre-\H-charts $(M_i,\pi,\Gamma_i,S_i)$ 
 with $s_i\in\pi(M_i)$ for $i=1,2$ 
such that for every $x_i\in M_i$ with $\pi_i(x_i)=s_i$ for $i=1,2$  
there exists a neighbourhood $U_i$ of $x_i\in M_i$  and a smooth map $\hat{f}\colon U_1\to U_2$ 
such that $\hat{f}(x_1)=x_2$ and $\pi_2(\hat{f}(z))=f(\pi_1(z))$ for all $z\in U_1$. 
(The map $\hat{f}$ will be sometimes called a \emph{lift} of $f$ with respect to the given pre-\H-charts and points $x_i\in\pi^{-1}(s_i)$ for $i=1,2$.)
 
The above mapping $f$ is called \emph{submersion}/\emph{immersion}/\emph{\'etale} if it has a lift $\hat{f}\colon U\to U'$ which is a submersion/immersion/\'etale map 
for all $s_1\in S_1$ and $s_2\in S_2$ with $f(s_1)=s_2$. 
Moreover, the mapping $f\colon S_1\to S_2$ is called a \emph{diffeomorphism} if it is bijective and both $f$ and $f^{-1}$ are smooth. 

If $S_1$ and $S_2$ are pre-\H-manifolds, then the set of all smooth maps from $S_1$ to $S_2$ is denoted by $\Hom_{\H}(S_1,S_2)$. 
\end{definition}

\begin{remark}\label{H-categ}
\normalfont 
The notions introduced in Definition~\ref{H-smooth} are correctly defined, that is, 
they do not depend of the choice of the pre-\H-charts 
$(M_i,\pi,\Gamma_i,S_i)$  for $i=1,2$. 
In fact, let 
$(M'_i,\pi',\Gamma'_i,S_i)$ 
be a local chart that is compatible with 
$(M_i,\pi,\Gamma_i,S_i)$ 
and satisfies $s_i\in \pi'_i(M'_i)$ for $i=1,2$. 
Also let $x'_i\in M'_i$ with $\pi'_i(x'_i)=s_i$ for $i=1,2$. 
Since 
the pre-\H-charts $(M_i,\pi,\Gamma_i,S_i)$ and $(M'_i,\pi',\Gamma'_i,S_i)$ 
are compatible, 
it follows that there exists an equivalence of pseudogroups 
$\Phi_i$ from $\Gamma_i$ to $\Gamma'_i$ as in Definition~\ref{compatib}.  
Since the union of the domains of the elements of $\Phi_i$ is equal to $M_i$, 
there exists $\varphi_i\colon U_{x_i}\to \varphi_i(U_{x_i})$ in $\Phi_i$ with $x_i\in U_{x_i}$. 
We have $\pi'_i\circ\varphi_i=\pi_i\vert_{U_i}$, hence $\pi'_i(\varphi(x_i))=\pi_i(x_i)$. 
Since
$\pi(x_i)=s_i=\pi'_i(x'_i)$, 
we obtain 
$\pi'_i(\varphi_i(x_i))=\pi'_i(x'_i)$, 
hence there exists $h'_i\in\Gamma'_i$ with $\varphi_i(x_i)$ in the domain of $h'_i$ and $h'_i(\varphi_i(x_i))=x'_i$. 
By shrinking $U_{x_i}$, we may assume that $\varphi_i(U_{x_i})$ is contained in the domain of $h'_i$, and let us denote $U_{x'_i}:=h'_i(\varphi_i(U_{x_i}))$, 
which is an open neigbourhood of $x'_i$. 
Here we have $h'_i\circ\varphi_i=h'_i\circ\varphi_i\circ\id_{U_{x_i}}\in\Phi_i$ 
since $\id_{U_{x_i}}\in\Gamma_i$ and $\Phi_i$ is an \'etale morphism of pseudogroups from $\Gamma_i$ to $\Gamma'_i$. 
Let us denote $h_i:= h'_i\circ\varphi_i\in\Phi_i$. 
We thus obtain 
open subsets 
$U_{x_i}\subseteq M_i$ and $U_{x'_i}\subseteq M'_i$ with $x_i\in U_{x_i}$ and $x'_i\in U_{x'_i}$ and a diffeomorphism $h_i\colon U_{x_i}\to U_{x'_i}$ with $h_i(x_i)=x'_i$ and $\pi'_i\circ h_i=\pi_i\vert_{U_{x_i}}$. 
Since $U_i$ is an open neighbourhood of $x_i\in M_i$, 
we may also assume 
$U_{x_i}\subseteq U_i$ 
(after replacing $U_{x_i}$ by $U_{x_i}\cap U_i$ and $U_{x'_i}$ by $h_i(U_{x_i}\cap U_i)$). 
Since $\hat{f}\colon U_1\to U_2$ is continuous and $\hat{f}(x_1)=x_2$, 
one can further assume $\hat{f}(U_{x_1})\subseteq U_{x_2}$ 
(after replacing $U_{x_1}$ by $U_{x_1}\cap\hat{f}^{-1}(U_{x_2})$). 
Then it makes sense to define 
$\hat{f}':=h_2\circ \hat{f}\circ h_1^{-1}$ and 
one has the commutative diagram 
$$\xymatrix{
	U_{x'_1} \ar[r]^{\hat{f}'} & U_{x'_2} \\
	U_{x_1} \ar[d]_{\pi_1} \ar[r]^{\hat{f}} \ar[u]^{h_1} & U_{x_2} \ar[d]^{\pi_2}  \ar[u]_{h_2}\\
		S_1\ar[r]^{f} & S_2
	}$$
where $\hat{f}'(x'_1)=x'_2$ and for all $z\in U_{x'_1}$, 
$$\pi_2(\hat{f}'(z))
=\pi_2(h_2(\hat{f}(h_1^{-1}(z))))
=\pi_2(\hat{f}(h_1^{-1}(z)))
=f(\pi_1(h_1^{-1}(z)))
=f(\pi_1(z))$$ 
where the second equality follows since $\pi_2\circ h_2=\pi_2\vert_{U_{x_2}}$, the third equality is based on $\pi_2\circ\hat{f}=f\circ \pi_1\vert_{U_1}$, and the fourth equality is based on $\pi_1\circ h_1=\pi_1\vert_{U_{x_1}}$.  

It is also clear that the identity mapping of every pre-\H-manifold is pre-\H-smooth and the composition of two pre-\H-smooth mappings is again pre-\H-smooth, 
hence one obtains the category of pre-\H-manifolds and its full subcategory consisting of the \H-manifolds. 
Since the morphisms in these two categories are the same, it is not necessary to introduce a notion of \H-smooth mapping. 
\end{remark}

\begin{proposition}\label{H-submersion} 
	Let $S'$, $S$, and $T$ be pre-\H-manifolds. 
	If $f\colon S'\to S$ is a surjective submersion,  
	and $\phi\colon S\to T$  is a mapping such that 
	 $\phi\circ f\colon S'\to T$  is smooth,  then $\phi$ is smooth.
\end{proposition}

\begin{proof} 
	We select any pre-\H-atlases 
$(M,\pi,\Gamma, S)$, $(M', \pi',\Gamma', S')$, and $(N,\tau, \Xi,T)$ 
	and fix an arbitrary pair $(x,z)\in M\times N$ with $(\phi\circ \pi)(x)=\tau(z)$. 
	
	Since $f$ is surjective, there exist  $s'\in S'$ such that $f(s')=\pi(x)$.  
	We then select  $x'\in M'$ such that $\pi'(x')=s'$. 
	Since $f\colon S'\to S$ is a submersion, there exists an open subset  $U_{x'}\subseteq M'$ 
	and a submersion $\hat{f}\colon U_{x'}\to 
	M$ with 
	$\pi\circ \hat{f}=f\circ \pi'\vert_{U_{x'}}$, 
	$x'\in U_{x'}$, 
	 and 
	 $\hat{f}(x')=x$.  
Moreover, since $\hat{f}$ is a submersion between Banach manifolds, there exist an open subset $U_x\subseteq M$ and a smooth mapping $s:U_x\to U_{x'}$ with $ \hat{f}\circ s=\id_{U_x}$ and $s(x)=x'$. 	
	
On the other hand $\phi(f(s'))=\phi(\pi(x))=\tau(z)$, $\pi'(x')=s'$, and 
the mapping	$\phi\circ f\colon S'\to T$ is smooth, 
	hence there exist an open subset $ U'_{x'}\subseteq M'$ 
and a smooth mapping $\hat{g}: U'_{x'}\to N$ 
		with $\tau\circ \hat{g}=(\phi\circ f)\circ \pi'\vert_{U'_{x'}}$, 
		$x'\in U'_{x'}$  and  $\hat{g}(x')=z$. 
		
Finally, since $\hat{f}\colon U_{x'}\to M$ is a submersion, 
it is in particular an open mapping, 
hence we obtain the open subset $W_x:=\hat{f}(U_{x'}\cap U'_{x'})\subseteq M$. 
Then the composition $\hat{g}\circ s\vert_{W_x}: W_x\to N$ is a smooth mapping which is a lift of $\phi$ 
	since 
$$\tau\circ (\hat{g}\circ s)=(\phi\circ f)\circ \pi'\circ s
=\phi\circ \pi\circ \hat{f}\circ s=\phi\circ \pi\text{ on }W_x$$
and moreover $(\hat{g}\circ s)(x)=\hat{g}(s(x))=\hat{g}(x')=z$. 
This completes the proof. 
	\end{proof}

\subsection{Alternative definition of \H-manifolds}
In the present context, just as in the case of \Q-manifolds in \cite[Prop.~1.1.2]{Pla}, we have the following characterization of an \H-atlas:

\begin{proposition}\label{characterization}  
	Let $\pi\colon M\to S$ be a surjective map from a Banach manifold onto a set $S$ 
and $\Gamma$ be a diffeomorphism pseudogroup on $M$ with $\Gamma\subseteq\Gamma(\pi)$. 
	Then $(M,\pi, \Gamma, S)$ is an \H-atlas if and only if the following conditions are satisfied: 
\begin{enumerate}[{\rm(i)}]
\item
\label{characterization_item1}
The graph $\Rc_\pi$ of the equivalence relation associated to $\pi$ has the structure of a split Banach submanifold
of $ M\times M$. 
  \item
  \label{characterization_item2}
  If $p_1,p_2\colon M\times M\to M$ are the Cartesian projections, then 
   $p_i\vert_{\Rc_\pi}\colon\Rc_\pi\to M$ is an  \'etale map 
   for $i=1,2$. 
   \item
   \label{characterization_item3} 
   For every  $h\in\Gamma$, its graph $\graf_h$ is an open submanifold of $\Rc_\pi$. 
\end{enumerate}
If these conditions are satisfied, then $(M,\pi, S)$ is a \Q-atlas if and only if $\Rc_\pi$ is an initial submanifold of $M\times M$. 
\end{proposition}

The proof of this proposition needs the following lemma which will also be used in the proof of Lemma~\ref{ParticularOpen}. 
Here we use the notion of pseudogroup of transformations in the sense of \cite[Ch. I, \S 1]{KN63}.

\begin{lemma}\label{graphIntersection} 
Let $\Gamma$ be a pseudogroup of 
transformations of a topological space $M$ and assume that the action of $\Gamma$ on $M$ is free\footnote{This means that for any $x\in M$ and  $h\in \Gamma(\pi)$, if $h(x)=x$  then the germ of $h$ at $x$ is the germ of $\id_M$  at $x$.}.

For $i=1,2$  let 
$h_i\colon U_i\to U'_i$ 
with $h_i\in\Gamma$,
and assume 
$$\graf_{h_1}\cap \graf_{h_2}\not=\emptyset.$$
If we denote  
$U:=\{x\in U_1\cap U_2\mid h_1(x)=h_2(x)\} $  
and $U':=h_1(U)=h_2(U)$, then the following assertions hold: 
\begin{enumerate}[{\rm(i)}]
	\item\label{graphIntersection_item1} 
	The set $U$ is open in $U_1\cap U_2$, and is also closed if the relative topology of $U_1\cap U_2$ is Hausdorff. 
	\item\label{graphIntersection_item2}  The set $U'$ is open in $U_1'\cap U_2'$, and is also closed if the relative topology of $U_1'\cap U_2'$ is Hausdorff. 
	\item\label{graphIntersection_item3}  The mapping $h:= h_1\vert_U=h_2\vert_U\colon U\to U'$ belongs to $\Gamma$. 
	\item\label{graphIntersection_item4}  One has $\graf_{h_1}\cap \graf_{h_2}=\graf_h$. 
\end{enumerate} 
\end{lemma}

\begin{proof} 
	\eqref{graphIntersection_item1} 
	To prove that $U$ is open, let $x\in U$ arbitrary, 
	hence $h_1(x)=h_2(x)=:y$. 
	Since 
	$h_1,h_2\in\Gamma$ and the action of $\Gamma$ on $M$ is free, 
	it follows that the germs of $h_1$ and $h_2$ at $x$ coincide, that is, there exists an open set $W\subseteq M$ satisfying $x\in W\subseteq U_i$ for $i=1,2$ and $h_1\vert_W=h_2\vert_W$. 
	In particular $W\subseteq U$. 
	Since the point $x\in U$ is arbitrary, we see that $U$ is an open subset of $U_1\cap U_2$. 
	
	Since both $h_1$ and $h_2$ are continuous, we see by the definition of $U$ that $U$ is closed in $U_1\cap U_2$ if additionally the relative topology of $U_1\cap U_2$ is Hausdorff. 
	
	\eqref{graphIntersection_item2} This can be proved as above, replacing $h_i$ by $h_i^{-1}$. 
	
	\eqref{graphIntersection_item3} This follows since the restriction of every element of the pseudogroup $\Gamma$ to any open subset of its domain still belongs to $\Gamma$. 
	
	\eqref{graphIntersection_item4} This is a straightforward property of graphs of any mappings. 
\end{proof}

\begin{lemma}\label{pseudogroup} 
	If $(M,\pi ,\Gamma,S)$
	is a pre-\H-atlas, 
	we introduce the binary relation
	$$x\simeq y \iff (\exists h\in \Gamma)\quad  h(x)=y.$$
	Then the following assertions hold: 
	\begin{enumerate}[{\rm(i)}]
		\item\label{pseudogroup_item1}
		The  binary relation $\simeq$ coincides with 
		the equivalence relation $\Rc_\pi$ 
		associated to $\pi$, that is, 
		for arbitrary $x,y\in M$ one has 
		$x\simeq y$ if and only if $\pi(x)=\pi(y)$.  
		\item\label{pseudogroup_item2} 
		If $(M,\pi ,\Gamma,S)$ is an \H-atlas, then the action of $\Gamma$ 
		on $M$ is free.  
	\end{enumerate}
\end{lemma}

\begin{proof}
	\eqref{pseudogroup_item1} This is straightforward. 
	
	\eqref{pseudogroup_item2} 
	This follows at once from the definition of an \H-atlas. 
\end{proof}

\begin{lemma}\label{Qdot}
	Let $(M,\pi,S)$ be a \Q-atlas, $J\subseteq\RR$ be an open interval with $0\in J$, and $c_1,c_2\in\Ci(J,M)$. 
	If $\pi\circ c_1=\pi\circ c_2$, then there exists a transversal diffeomorphism $h\colon U_1\to U_2$ with $c_j(0)\in U_j$ for $j=1,2$, 
	and there exists an open interval $J_0\subseteq \RR$ with $0\in J_0\subseteq J$ and 
	$h(c_1(t))=c_2(t)$ for every $t\in J_0$, and in particular  $(T_{c_1(0)}h)(\dot{c}_1(0))=\dot{c}_2(0)$. 
\end{lemma}

\begin{proof}
	Since $\pi(c_1(0))=\pi(c_2(0))$, there is a transversal diffeomorphism $h\colon U_1\to U_2$ with $c_j(0)\in U_j$ for $j=1,2$, and $h(c_1(0))=c_2(0)$. 
	Replacing $J$ by the connected component of $0$ in the open subset $c_1^{-1}(U_1)\cap c_2^{-1}(U_2)\subseteq J$, we may assume $c_j(J)\subseteq U_j$ for $j=1,2$. 
	
	Then the smooth curves $h\circ c_1,c_2\colon J\to M$ satisfy $(h\circ c_1)(0)=c_2(0)$ hence, since $(M,\pi,S)$ is a \Q-atlas, the set $\{t\in J\mid (h\circ c_1)(t)=c_2(t)\}$ is an open neighbourhood of $0\in J$. 
	Denoting by $J_0$ the connected component of $0$ in that open neighbourhood, 
	we obtain the assertion. 
\end{proof}

\begin{proof}[Proof of Proposition \ref{characterization}] 
	First assume that $(M,\pi ,\Gamma,S)$ 
	is an \H-atlas. 
	For any 
	$h_1,h_2\in\Gamma$ 
	with $\graf_{h_1}\cap\graf_{h_2}\ne\emptyset$ there exists $h\in\Gamma$
	with $\graf_{h_1}\cap\graf_{h_2}=\graf_h$ 
	by Lemma~\ref{graphIntersection} 
	and the definition of a diffeomorphism pseudogroup, 
	cf. \cite[Def. 1.1.1(iii)]{Wal}. 
	Therefore the family of all graphs of elements of $\Gamma$
	gives rise to an atlas of Banach manifold on $\Rc$ 
	satisfying the condition~\eqref{characterization_item3} in the statement. 
	For arbitrary $(x,y)\in \Rc$, if $h\colon U\to U'$ 
	satisfies $h\in\Gamma$ and 
	$h(x)=y$, then 
	the smooth map $\sigma\colon U\to \Rc$, $\sigma(z):= (z,h(z))$ is a local cross-section of $p_1\vert_{\Rc}$ with $\sigma(x)=(x,y)$. 
	This implies that the inclusion mapping $\Rc\hookrightarrow M\times M$ is an immersion 
	and $\Ker T_{(z,h(z))}p_1$ is a topological complement of $T_{(z,h(z))}\mathcal{R}$ in $T_{(z,h(z))}(M\times M)$. But  the tangent space at $(z,h(z))\in \graf_{h_{U}}$ is the closed subspace 
	$\{(u, (T_z h)u)\in T_zM\times T_{h(z)}M\mid u\in T_zM\}$.
	 It follows that $p_2$ is also an \'etale map. 
	
	Conversely, assume that conditions \eqref{characterization_item1}--\eqref{characterization_item3}  
	are satisfied. 
	Since $p_1\vert_{\Rc}\colon \Rc\to M$ is an \'etale map, 
	for all $(x,y)\in \Rc$ there exists a smooth mapping $\sigma\colon U\to W$ with $\sigma(x)=(x,y)$, $p_1\circ \sigma=\id_U$, and $\sigma\circ p_1\vert_W=\id_W$ for some open subsets $U\subseteq M$ and $W\subseteq \Rc$ with $x\in U$ and $(x,y)\in W$. 
	The equality $p_1\circ \sigma=\id_U$ shows that there exists a mapping $h\colon U\to M$ with $\sigma(z)=(z,h(z))$ for all $z\in U$. 
	Then the equality $\id_W=\sigma\circ p_1\vert_W$ implies $(z,v)=\sigma(z)=(z,h(z))$ for every point $(z,v)\in W$, hence $W=\graf_h$. 
	Moreover, since $\sigma(x)=(x,y)$, one has $h(x)=y$. 
	We also note that, since $p_2\vert_{\Rc}\colon \Rc\to M$ is smooth and $h=p_2\circ\sigma$, it follows that $h$ is smooth. 
	
	Using also the hypothesis that $p_2\vert_{\Rc}\colon \Rc\to M$ is 
an \'etale map, 
we obtain a mapping $\sigma'\colon U'\to W'$ with $\sigma'(y)=(x,y)$, $p_2\circ \sigma'=\id_{U'}$, and $\sigma'\circ p_2\vert_{W'}=\id_{W'}$ for some open subsets $U'\subseteq M$ and $W'\subseteq \Rc$ with $y\in U'$ and $(x,y)\in W'$. 
	After replacing both $W$ and $W'$ by $W\cap W'$ and suitably shrinking $U$ and $U'$ we may and do assume $W=W'$. 
	As above, the equality $p_2\circ \sigma'=\id_{U'}$ shows that there exists a mapping $h'\colon U'\to M$ with $\sigma'(z')=(h'(z'), z')$ for all $z'\in U'$. 
	Then the equality $\id_{W'}=\sigma'\circ p_2\vert_{W'}$ implies $(w,z')=\sigma'(z')=(h'(z'),z')$ for every point $(w,z')\in W=\graf_h$, 
	and this implies that both mappings $h\colon U\to M$ and $h'\colon U'\to M$ are injective and are inverse to each other, 
	that is, $U'=h(U)$, $U=h'(U')$, $h\circ h'=\id_{U'}$, and $h'\circ h=\id_U$. 
	Moreover, since $\sigma'(y)=(x,y)$, one has $h'(y)=x$. 
	Also, since $p_1\vert_{\Rc}\colon \Rc\to M$ is smooth and $h=p_1\circ\sigma'$, we obtain that $h'$ is smooth.
	
    Consequently $h\colon U\to U'$ is a diffeomorphism. 
   Moreover, since we have assumed that the conditions \eqref{characterization_item2}--\eqref{characterization_item3}  
    	are satisfied, it follows that every $x\in U$ has an open neigbourhood $U_x\subseteq U$ for which there exists $h_x\in \Gamma$ with $h_x\colon U_x\to h(U_x)$ and $\graf_{h_x}\subseteq \graf_h$, that is, $h\vert_{U_x}=h_x$. 
    It then follows by the definition of a diffeomorphism pseudogroup, specifically \cite[Def. 1.1.1(iv)]{Wal}, that actually $h\in\Gamma$.  
	To prove the uniqueness of germs of transverse diffeomorphisms,  
	 let $\bar{h}\colon \bar{U}\to \bar{U}'$ be another 
diffeomorphism with $\bar{h}(x)=y$ and $\bar{h}\in\Gamma$. 
Then $z\mapsto (z,h(z))$ and $z\mapsto (z,\bar{h}(z))$ are smooth cross-sections section of $p_1$ over $U\cap \bar{U}$ taking the same value at $x$. 
Since $p_1$ is an \'etale map,   
this implies that $h=\bar{h}$ on some neighborhood of $x\in U\cap \bar{U}$. 
This completes the proof of the fact that \eqref{characterization_item1}--\eqref{characterization_item3} imply that 
$(M,\pi,\Gamma,S)$
is an \H-atlas.

It remains to show that if $(M,\pi,\Gamma,S)$ 
is an \H-atlas, then 
$(M,\pi, S)$ is a \Q-atlas if and only if $\Rc_\pi$ is an initial submanifold of $M\times M$. 
In fact, if $(M,\pi, S)$ is a \Q-atlas and $\varphi\colon  T\to M\times M$ is a smooth mapping defined on a Banach manifold with $\varphi(T)\subseteq \Rc_\pi$, then for every smooth curve $c\colon J\to T$ (where $J\subseteq \RR$ is an open interval) it directly follows by Lemma~\ref{Qdot} that 
the curve $\varphi\circ c\colon J\to \Rc_\pi$ is smooth. 
Then, by the technique of convenient analysis (see e.g., \cite{Fr81}), 
we obtain that $\varphi\colon T\to \Rc_\pi$ is smooth, and therefore $\Rc_\pi$ is an initial submanifold of $M\times M$. 
Conversely, if $\Rc_\pi$ is an initial submanifold of $M\times M$, 
then for every smooth mapping $\varphi=(f_1,f_2)\colon T\to M\times M$ defined on a Banach manifold with $\varphi(T)\subseteq \Rc_\pi$ 
it follows that the mapping $\varphi=(f_1,f_2)\colon T\to \Rc_\pi$ is smooth. 
Let $t_0\in T$ arbitrary. 
Then $(f_1(t_0),f_2(t_0))=\varphi(t_0)\in\Rc_\pi$ hence there exist an 
open subset $T_0\subseteq T$ with $t_0\in T_0$ and a transversal diffeomorphism $h\colon U\to V$ with $\varphi(T_0)\subseteq \graf_h\subseteq U\times V$ 
and $f_2(t)=h(f_1(t))$ for every $t\in T_0$. 
If moreover $f_1(t_0)=f_2(t_0)=:x_0\in M$, then the transversal diffeomorphism $h$ satisfies $h(x_0)=x_0$ hence, by the uniqueness condition in the definition of an \H-atlas, we obtain an open subset $U_0\subseteq U\cap V$ with $x_0\in U_0$ and $h\vert_{U_0}=\id_{U_0}$. 
Then $\graf_{h\vert_{U_0}}$ is an open neighbourhood of $(x_0,x_0)\in\Rc_\pi$ hence, since $\varphi=(f_1,f_2)\colon T\to \Rc_\pi$ is smooth and in particular continuous, it follows that $\varphi^{-1}(\graf_{h\vert_{U_0}})$ is an open neighbourhood of $t_0\in T$ 
which is contained in the set $\{t\in T\mid f_1(t)=f_2(t)\}$. 
This shows that condition~\eqref{Q_ex_b} is satisfied, and therefore $(M,\pi, S)$ is a \Q-atlas. 
This completes the proof. 
\end{proof}

Proposition~\ref{characterization} has the following corollary.

\begin{corollary}
	\label{characterization_cor1}
	The direct product of two \H-atlases is again an \H-atlas.
\end{corollary}

\begin{proof}
	Let 
	$M_j$ be a Banach manifold, $\pi_j\colon M_j\to S_j$ be a surjective map, 
	and $\Gamma_j\subseteq\Gamma(\pi_j)$ be a diffeomorphism pseudogroup 
	 for which the graph $\Rc_{\pi_j}$ has the structure of a Banach manifold satisfying the conditions \eqref{characterization_item1}--\eqref{characterization_item3} in  Proposition~\ref{characterization} for $j=1,2$. 
	We consider the direct product map 
	$$\pi:=\pi_1\times\pi_2\colon M_1\times M_2\to S_1\times S_2,
	\quad  \pi(x_1,x_2):=(\pi_1(x_1),\pi_2(x_2))$$ 
	and we show that its graph $\Rc_\pi$ can be endowed with 
	the structure of a Banach manifold satisfying the same conditions \eqref{characterization_item1}--\eqref{characterization_item3} in  Proposition~\ref{characterization} 
	with respect to the pseudogroup $\Gamma:=\Gamma_1\times\Gamma_2=\{h_1\times h_2\mid h_1\in\Gamma_1,\ h_2\in\Gamma_2\}$. 
	
	To this end we note that 
	\begin{align*}
	\Rc_\pi & =\{((x_1,x_2),(y_1,y_2))\in (M_1\times M_2)\times(M_1\times M_2)\mid 
	\pi_j(x_j)=\pi_j(y_j),\ j=1,2\}, \\
	\Rc_{\pi_j}& =\{(x_j,y_j)\in M_j\times M_j\mid \pi_j(x_j)=\pi_j(y_j)\}
	\text{ for }j=1,2.
	\end{align*}
	If we now consider the mapping that interchanges the inner variables 
	\begin{align*}
	\Psi & \colon (M_1\times M_2)\times(M_1\times M_2)\to (M_1\times M_1)\times(M_2\times M_2), \\ 
	\Psi & (((x_1,x_2),(y_1,y_2)))=((x_1,y_1),(x_2,y_2))
	\end{align*}
	then it is clear that $\Psi$ is a diffeomorphism of Banach manifolds and \begin{equation}
	\label{graph-prod}
	\Psi(\Rc_\pi) =\Rc_{\pi_1}\times\Rc_{\pi_2}.
	\end{equation}
	More precisely, we have the commutative diagram 
	\begin{equation}
	\label{graph-prod-diagr}
	\xymatrix{
		\Rc_\pi \ar[d]_{\Psi\vert_{\Rc_\pi}}\ar[r]^{\iota\qquad\quad} & (M_1\times M_2)\times (M_1\times M_2)  \\
		\Rc_{\pi_1}\times\Rc_{\pi_2} \ar[r]^{\iota_1\times \iota_2\qquad\quad} & (M_1\times M_1)\times(M_2\times M_2)  \ar[u]_{\Psi^{-1}}
	}
	\end{equation}
	where $\iota_j\colon\Rc_{\pi_j}\to M_j\times M_j$ for $j=1,2$ and $\iota\colon\Rc_\pi\to (M_1\times M_2)\times (M_1\times M_2) $ 
	are inclusion maps. 
	
	Moreover, if we denote by $p^\sigma_1$ the first Cartesian projection defined on the graph~$\Rc_\sigma$ of any of the equivalence relations $\sigma\in\{\pi,\pi_1,\pi_2\}$, then we have 
	\begin{align*}
	p^{\pi}_1((x_1,x_2),(y_1,y_2))
	& =(x_1,x_2) \\
	& =(p^{\pi_1}_1(x_1,y_1),p^{\pi_2}_1(x_2,y_2)) \\
	&=(p^{\pi_1}_1\times p^{\pi_2}_1)((x_1,y_1),(x_2,y_2))
	\end{align*}
	hence 
	\begin{equation}
	\label{proj-prod}
	p^{\pi}_1=(p^{\pi_1}_1\times p^{\pi_2}_1)\circ\Psi^{-1}.
	\end{equation}
	and we have the commutative diagram 
	\begin{equation}
	\label{proj-prod-diagr}
	\xymatrix{
		\Rc_\pi \ar[d]_{p^{\pi}_1} \ar[r]^{\Psi\vert_{\Rc_\pi}}& \Rc_{\pi_1}\times\Rc_{\pi_2}  \ar[dl]^{p^{\pi_1}_1\times p^{\pi_2}_1}\\
		M_1\times M_2 & 
	}
	\end{equation}
	We now endow the graph $\Rc_\pi$ with the unique structure of a Banach manifold for which the bijective map  $$\Psi\vert_{\Rc_\pi}\colon\Rc_\pi\to \Rc_{\pi_1}\times\Rc_{\pi_2}$$
	(cf. \eqref{graph-prod}) is a diffeomorphism. 
	Since $\Psi$ and $\Psi^{-1}$ are diffeomorphisms, it then follows by \eqref{graph-prod},\eqref{graph-prod-diagr},\eqref{proj-prod},\eqref{proj-prod-diagr},  
	and 
	the fact that 
	$\Rc_{\pi_j}$ has the structure of a split Banach submanifold (respectively, initial split Banach submanifold, respectively submanifold) of the Banach manifold $M_j$ for which the first Cartesian projection $p^{\pi_j}_1\colon\Rc_{\pi_j}\to M_j$ is a local diffeomorphism for $j=1,2$, that this property is shared by $\Rc_\pi$ with respect to $M_1\times M_2$. 
	Finally, if $h_j\in\Gamma_j$ for $j=1,2$, 
	then the definition of $\Psi$ implies $\Psi(\graf_{h_1\times h_2})=\graf_{h_1}\times\graf_{h_2}$.  
	Since $\graf_{h_j}\subseteq\Rc_{\pi_j}$ is an open submanifold for $j=1,2$ 
and $\Psi\vert_{\Rc_\pi}\colon\Rc_\pi\to\Rc_{\pi_1}\times\Rc_{\pi_2}$ is a diffeomorphism, it then follows that $\graf_{h_1}\times\graf_{h_2}\subseteq \Rc_\pi$ is an open submanifold. 
Thus, 
the graph $\Rc_\pi$ satisfies condition \eqref{characterization_item3} in  Proposition~\ref{characterization} 
with respect to the pseudogroup $\Gamma=\Gamma_1\times\Gamma_2$, and this completes the proof. 
\end{proof}

We conclude this subsection by a proposition which provides examples of  \H-manifolds and which is, in some sense, 
converse to Lemma~\ref{pseudogroup}.

\begin{proposition}\label{actionG} 
	Let $\Gamma$ be a diffeomorphism pseudogroup on a  Banach manifold~$\Tc$.  
 If 
  $ \pi_\Tc:\Tc \rightarrow \Tc/\Gamma$ is the canonical projection on the quotient space associated to  the equivalence relation $\simeq$ defined\footnote{See Lemma \ref{pseudogroup}.} by $\Gamma$,   then 
$(\Tc,\pi_\Tc,\Gamma,\Tc/\Gamma)$ 
  is a pre-\H-atlas on~$\Tc/\Gamma$. 
 
 If moreover the action of $\Gamma$ on $\Tc$ is free, then 
$(\Tc,\pi_\Tc,\Gamma,\Tc/\Gamma)$  is an \H-atlas. 
\end{proposition}

\begin{proof}
It is straightforward to check that $(\Tc,\pi_\Tc,\Gamma,\Tc/\Gamma)$ 
 is a pre-\H-atlas. 

We now prove that if the action of $\Gamma$ on $\Tc$ is free, then 
$(\Tc,\pi_\Tc,\Gamma,\Tc/\Gamma)$   is actually an \H-atlas. 
We denote by  $\Rc$ the graph   of the equivalence relation $\simeq$. 
For any $(x,y)\in \Rc$, there exists   $h\in \Gamma$ such that $h(x)=y$ and the germ of $h$ at $x\in\Tc$ is unique. 
The domain of $h$ is an open neighbourhood of $x$, and 
the graph $\graf_h$ of $h$ 
defines a neighbourhood of $(x,y)$ in $\Rc$. 
The intersection of two such graphs is again the graph of such a diffeomorphism by Lemma~\ref{graphIntersection}. 
We can provide $\mathcal{R}$ by the topology generated by such graphs. 
In this way, it is clear that the canonical projection $q_1$ of $ \Tc\times  \Tc$ on the first factor   in restriction to $\mathcal{R}$ is a local homeomorphism.  

Moreover, since $ \Tc$ is a Banach manifold, we can then provide $\mathcal{R}$ with a Banach manifold structure. 
More precisely consider the graph $\graf_h$ 
over $U$ 
as previously, such that $q_1(U)$ is a chart domain for $ \Tc$.  
If $\phi: U\to \mathbb{T}$ is the associated parametrization into the Banach model $\mathbb{T}$ of $ \Tc$, then $(\graf_h,\Phi)$ is a chart on $\mathcal{R}$ where $\Phi=\phi\circ q_1$. 
The smoothness of the transition functions is due to the fact  that $\Tc$ is a Banach manifold and that the intersection of such graphs is again a graph of a local diffeomorphism by Lemma~\ref{graphIntersection} again.

It remains to show the inclusion map  $J\colon\Rc\hookrightarrow\Tc\times\Tc$  is an immersion whose range is complemented. 
Choose some chart domain  $\graf_h=\{(x,h(x))\mid x\in U\}$ in $\mathcal{R}$ as previously built.    
Since $(g(U), \phi\circ g^{-1})$ is also a chart in $ \Tc$ so $U\times g(U)$ is a chart for $ \Tc\times  \Tc$ and $G_U=R\cap U\times g(U)$.
Then the inclusion of $\mathcal{R}$  in restriction to such a chart domain   in $ \Tc\times  \Tc$ is the map $(x,h(x))\mapsto (x,h(x))$ from $\graf_h$ to $U\times h(U)$  and so is smooth and injective.  
So the differential of $J$ is injective  and its range at $(x,g(x))$  is the closed space $T_x  \Tc\times T_{g(x)}g(T_x \Tc)\subset T_x \Tc\times T_{g(x)}  \Tc $ which is complemented.
\end{proof}

\begin{example}
\label{Ba73_page242}
\normalfont 
We now show that the counterexample of a \Q-manifold given in \cite[Ch. I, \S 2, no. 3, page 242]{Ba73}
can be explained in the framework of \H-manifolds. 
Let $M=\RR$ and for every $k\in\RR$ consider the translation map $\lambda_k\colon M\to M$, $\lambda_k(x):=x+k$. 
Then define $\Gamma$ as diffeomorphism pseudogroup on $M$ generated by the family of diffeomorphisms $\{\lambda_k\mid k\in\RR\}$ in the sense of \cite[Def. 1.1.3]{Wal}. 
It is clear that the action of $\Gamma$ is free and transitive, that is, the set $M/\Gamma=\{\ast\}$ consists of a single point and, for the quotient map $\pi\colon\RR= M\to M/\Gamma=\{\ast\}$, the quadruple $(\RR,\pi,\Gamma,\{\ast\})$ is an \H-atlas by Proposition~\ref{actionG}. 
Moreover, the graph of the equivalence relation associated quotient map $\pi\colon M\to M/\Gamma$ satisfies 
$$\Rc_\pi=M\times M=\RR^2.$$
By Proposition~\ref{characterization}, one has a unique structure of manifold on $\Rc_\pi$ 
for which the Cartesian projections $p_i\colon\Rc_\pi\to M$ are \'etale maps and for every $h\in\Gamma$ its graph~$\graf_h$ is an open submanifold of $\Rc_\pi$. 
With the above notation, the graphs of the translation diffeomorphisms $\lambda_k$ are  
$$\graf_{\lambda_k}=\{(x,x+k)\mid x\in\RR\}\subseteq\RR^2\text{ for every }k\in\RR.$$
The manifold structure on $\Rc_\pi$ corresponding to the \H-atlas  $(\RR,\pi,\Gamma,\{\ast\})$ 
is then the structure of 1-dimensional manifold of $\RR^2$ which coincides with the foliation by lines which are parallel to the main diagonal 
$$\RR^2=\bigsqcup_{k\in\RR}\{(x,x+k)\mid x\in\RR\}.$$
For this manifold structure on $\Rc_\pi=\RR^2$, the identity map $\Rc_\pi\to M\times M$ is an immersion which is \emph{not} an initial map, 
and this reflects the fact pointed out in \cite[page 242]{Ba73}, 
namely that 
the \H-atlas  $(\RR,\pi,\Gamma,\{\ast\})$ is \emph{not} a \Q-atlas. 
In particular, if we denote by $\Gamma(\pi)$ the diffeomorphism pseudogroup on $M$ 
consisting of all transverse diffeomorphisms of $\pi$, 
then in this case $\Gamma(\pi)$ consists of \emph{all} local diffeomprphisms of $\RR$, 
but the quadruple $(\RR,\pi,\Gamma(\pi),\{\ast\})$ is \emph{not} an \H-atlas. 
(Compare Example~\ref{Q_ex}.)
\end{example}

\subsection{Tangent space of a pre-\H-manifold}
\label{connexetangent} 
We begin by a lemma that  will be repeatedly used.

\begin{lemma}
\label{ParticularOpen}
If $(M,\pi,\Gamma,S)$ 
is a pre-\H-atlas, then for every $s\in S$ and every finite subset $F\subseteq \pi^{-1}(s)$ there exist a family $\{U^F_x\mid x\in F\}$ of open subsets of $M$ and a family of diffeomorphisms $\{h^F_{xy}\colon U^F_y\to U^F_x\mid x,y\in F\}$ satisfying $x\in U^F_x$, $h^F_{xy}\in\Gamma$, 
$h^F_{xx}=\id_{U^F_x}$, $h^F_{xy}=(h^F_{yx})^{-1}$, $h^F_{xy}\circ h^F_{yz}=h^F_{xz}$, and  $h^F_{xy}(y)=x$ for all $x,y,z\in F$. 
\end{lemma}

\begin{proof}
The assertion is clear if $F=\emptyset$. 
Let us assume that $F\ne\emptyset$ and select $x_0\in F$. 
Since $(M,\pi,\Gamma,S)$ 
is a pre-\H-atlas, for every $y\in F\setminus\{x_0\}$ we may select open subsets $V_{x_0y},V_{yx_0}\subseteq M$ and 
a 
diffeomorphism 
$h_{yx_0}\colon V_{x_0y}\to V_{yx_0}$ with 
$h_{yx_0}\in\Gamma$ and 
$h_{yx_0}(x_0)=y$. 

We define $U^F_{x_0}:=\bigcap\limits_{y\in F\setminus\{x_0\}}V_{x_0y}$, which is the intersection of finitely many open neighbourhoods of $x_0\in M$. 
Then, for every $y\in F\setminus\{x_0\}$ we define $U^F_y:=h_{x_0y}(U^F_{x_0})$, 
$h^F_{yx_0}:=h_{yx_0}\vert_{U^F_{x_0}}\colon U^F_{x_0}\to U^F_y$, 
and $h^F_{x_0y}:=(h^F_{yx_0})^{-1}$. 
Moreover, if $y_1,y_2\in F\setminus\{x_0\}$ then we define 
$h^F_{y_1y_2}:=h^F_{y_1x_0}\circ h^F_{x_0y_2}$. 
Finally, defining $h^F_{xx}:=\id_{U^F_x}$ for every $x\in F$, it is straightforward to check that the families of open sets $\{U^F_x\mid x\in F\}$ and of diffeomorphisms $\{h^F_{xy}\mid x,y\in F\}$ satisfy all the required conditions. 
\end{proof}

\begin{remark}
\label{ParticularOpen_rem}
	\normalfont 
	In the framework of  Lemma \ref{ParticularOpen}, 
	if $s\in S$ and $F=\{x,y\}\subseteq\pi^{-1}(s)$, then we will use the notation $h_{xy}:=h^F_{xy}$, 
	$O^x_y:=U^F_y$, and $O^y_x:=U^F_x$. 
	Thus  $h_{xy}\colon O^x_y\to O^y_x$. 
	
	For any fixed  $s\in S$ and $x\in \pi^{-1}(s)$, 
	we then denote 
	$$O_s(x):=\bigcup_{y\in \pi^{-1}(s)} O_y^x\subseteq M.$$ 
	The open subsets $O_y^x$ for $y\in\pi^{-1}(x)$ may \emph{never} be selected to be mutually disjoint. 
	See also \cite[Prop. 2.6]{BPZ19}. 
	In particular, the sets $O_y^x$ are \emph{not} the connected components of $O_s(x)$. 
	
	For instance, let $\pi\colon \RR\to\RR/\QQ$ be the natural quotient map, 
	which is a \Q-atlas, and in particular an \H-atlas. 
	For any $x\in\RR$, if we denote $s:=\pi(x)=x+\QQ\in\RR/\QQ$, 
	then one has $\pi^{-1}(s)=x+\QQ$, which is a dense subset of $\RR$. 
	Now for every $y\in\pi^{-1}(s)$ we select any open set  $O_y^x\subseteq \RR$ with $y\in O_y^x$. 
	Then for every $y_1\in \pi^{-1}(s)$ the set $\pi^{-1}(s)\setminus\{y_1\}$ is still dense in $\RR$, hence its intersection with the open set $O_{y_1}^x$ is nonempty. 
	Then there exists $y_2\in \pi^{-1}(s)\setminus\{y_1\}$ with $y_2\in O_{y_1}^x$, which implies 
	$y_2\in O_{y_1}^x\cap O_{y_2}^x$, and thus $O_{y_1}^x\cap O_{y_2}^x\ne\emptyset$. 
\end{remark}

In the following definition we adapt the definition of the tangent space of a \Q-manifold given in \cite[Ch. 1, \S 1, no. 4]{Ba73}. 
It will be convenient to use the following simplified notation: 

\begin{notation}
	\normalfont
	If $M$ is a n.n.H. Banach manifold and $\pi\colon M\to S$ is a surjective map then
	for any $x,y\in M$ 
	and any 
	diffeomorphism $h\in\Gamma(\pi)$ 
	(cf. Definition~\ref{preHtg_def})
	the notation $h(x)=y$ implicitly means that $\pi(x)=\pi(y)$ and the point~$x$ belongs to the open set where~$h$ is defined. 
\end{notation}

\begin{definition}
\label{tgH_def}
\normalfont 
Let $S$ be a pre-\H-manifold and $s\in S$. 
We consider all the quadruples $(M,\pi,x,v)$ 
where $(M,\pi,\Gamma,S)$ 
is a pre-\H-chart of $S$ with $x\in M$, $\pi(x)=s$, and $v\in T_xM$. 
If $(M_j,\pi_j,x_j,v_j)$ for $j=1,2$ are such quadruples, 
then the pre-\H-charts $(M_j,\pi_j,\Gamma_j,S)$ 
are compatible, 
hence 
there exists an equivalence of pseudogroups $\Phi$ from $\Gamma_1$ to $\Gamma_2$ 
as in Definition~\ref{compatib}. 
With this notation, one says that the quadruples $(M_j,\pi_j,x_j,v_j)$ for $j=1,2$ are \emph{pre-\H-equivalent} if 
there exists a 
diffeomorphism $h_{21}\in\Phi$ with 
$h_{21}(x_1)=x_2$, 
whose  tangent map $T_{x_1}(h_{21})\colon T_{x_1}M_1\to T_{x_2}M_2$ satisfies the condition $(T_{x_1}(h_{21}))(v_1)=v_2$. 

If $(M_3,\pi_3,x_3,v_3)$ is another quadruple as above 
with respect to a pre-\H-chart $(M_3,\pi_3,x_3,v_3)$ 
with $\pi_3(x_3)=s$, which is equivalent to the quadruple $(M_2,\pi_2,x_2,v_2)$, 
then there exist 
an equivalence of pseudogroups $\Phi'$ from $\Gamma_2$ to $\Gamma_3$ 
	as in Definition~\ref{compatib} and 
a 
diffeomorphism  $h_{32}\in\Phi'$ 
with $h_{32}(x_2)=x_3$ and $(T_{x_2}(h_{32}))(v_2)=v_3$. 
Then $\Phi'\circ\Phi:=\{h\circ h'\mid h\in \Phi,h'\in\Phi'\}$ 
is an equivalence of pseudogroups from $\Gamma_1$ to $\Gamma_3$ 
as in Definition~\ref{compatib} and 
we obtain the 
diffeomorphism $h_{32}\circ h_{21}\in\Phi'\circ\Phi$ 
with $(h_{32}\circ h_{21})(x_1)=x_3$ 
and $(T_{x_1}(h_{32}\circ h_{21}))(v_1)=v_3$. 
This shows that the quadruple $(M_3,\pi_3,x_3,v_3)$ is equivalent to the quadruple $(M_1,\pi_1,x_1,v_1)$. 

Thus the equivalence of quadruples is transitive, and then it makes sense to consider the equivalence classes of quadruples $(M,\pi,x,v)$ as above 
with $\pi(x)=s$. 
Such an equivalence class is called a \emph{tangent vector} at the point $s$ of the pre-\H-manifold~$S$. 
For any fixed pre-\H-chart 
$(M,\pi,\Gamma,S)$ 
if $s\in S$, $x_0\in \pi^{-1}(s)$ and $v\in T_{x_0}M$, 
then  $(M,\pi,x_0,v)$  is a quadruple as above and its equivalence class is denoted by $(T_{x_0}\pi)(v)\in T_sS$. 
It is easily seen that for any fixed $x_0\in \pi^{-1}(s)$, every tangent vector at $s\in S$ is equal to $(T_{x_0}\pi)(v)$ for a suitable vector $v\in T_{x_0}M$. 
This shows that all the tangent vectors at $s\in S$ constitute a set (unlike any of the above equivalence classes of quadruples), to be denoted $T_sS$ and called the \emph{tangent space of the pre-\H-manifold at $s\in S$}, 
and it is easily checked that the mapping 
\begin{equation}
\label{tgH_def_eq1}
T_{x_0}\pi\colon T_{x_0}M\to T_sS,\quad v\mapsto (T_{x_0}\pi)(v).
\end{equation}
is surjective for any pre-\H-chart 
$(M,\pi,\Gamma,S)$  
with $s\in\pi(M)$ and $x_0\in\pi^{-1}(s)$. 
The disjoint union 
$$TS:=\bigsqcup\limits_{s\in S}T_sS$$ 
is called the \emph{tangent space of the pre-\H-manifold~$S$}. 
\end{definition}

We now prove that the tangent space of every pre-\H-manifold has the natural structure of a pre-\H-manifold. 

\begin{proposition}
\label{preHtg}
If $(M,\pi,\Gamma,S)$ 
is a pre-\H-chart, then the following assertions hold. 
\begin{enumerate}[{\rm 1.}]
	\item\label{preHtg_item1} 
	For every $x\in M$, the mapping $T_x\pi\colon T_xM\to T_{\pi(x)}S$ is surjective. 
	\item\label{preHtg_item2}  If 
	$(M,\pi,\Gamma,S)$
	is an \H-chart then the mapping $T_x\pi\colon T_xM\to T_{\pi(x)}S$ is bijective 
	for arbitrary $x\in M$. 
	For every $s\in \pi(M)$, the space $T_sS$ has the natural structure of a Banach space for which the mapping $T_x\pi\colon T_xM\to T_sS$ 
	is a Banach space isomorphism for any $x\in\pi^{-1}(s)$. 
	\item\label{preHtg_item3}  
	Let us denote by $T\Gamma$ the diffeomorphism pseudogroup on the n.n.H. manifold $TM$ 
	generated by the set of diffeomorphisms $\{Th\mid h\in\Gamma\}$ in the sense of \cite[Def. 1.1.3]{Wal}.
	Also define  the mapping $T\pi\colon TM\to TS$ given by $T\pi\vert_{T_xM}:=T_x\pi$ for every $x\in M$. 
   Then the quadruple $(TM,T\pi,T\Gamma,TS)$ is again a pre-\H-chart, which is a pre-\H-atlas if and only if 
   $(M,\pi,\Gamma,S)$ 
   is a pre-\H-atlas. 
	If moreover 
	$(M,\pi,\Gamma,S)$  is an \H-atlas, 
	then $(TM,T\pi,T\Gamma,TS)$ 
	is in turn an \H-atlas. 
	\item
	\label{preHtg_item4} 
	If $S$ is a pre-\H-manifold then 
	the canonical projection $p_S\colon TS\to S$ is smooth map in the sense of pre-\H-manifolds  and we have the commutative diagram 
	\begin{equation}
	\label{Tpi}
	\xymatrix{
		TM 
		\ar[r]^{T\pi} 
		\ar[d]_{p_M}  
		&TS \ar[d]^{p_S} 
		\\
		M  \ar[r]^{\pi} 
		&S
	}
	\end{equation}
	Moreover, $\pi$ is both a submersion and an immersion of pre-\H-manifolds.
\end{enumerate}
\end{proposition}

\begin{proof}
\ref{preHtg_item1}. 
Straightforward.  

\ref{preHtg_item2}. 
Taking into account Assertion~\ref{preHtg_item1}, it suffices to show that 
$T_x\pi\colon T_xM\to T_{\pi(x)}S$ si injective. 
To this end let $v_1,v_2\in T_xM$ with $(T_x\pi)(v_1)=(T_x\pi)(v_2)$. 
Then, by Definition~\ref{tgH_def}, there exists 
a 
diffeomorphism $h\in\Gamma$ 
with $h(x)=x$ and $(T_xh)(v_1)=v_2$. 
Since $(M,\pi,\Gamma,S)$ 
is an \H-chart, the condition $h(x)=x$ implies $h=\id_M$ on a suitable open neighbourhood of $x\in M$, hence $T_xh=\id_{T_xM}$, and then $v_1=v_2$. 

Since $T_xM$ is a Banach space, we then obtain a Banach space structure on $T_{\pi(x)}S$ by transport of structure via the bijective map $T_x\pi\colon T_xM\to T_{\pi(x)}S$. 

\ref{preHtg_item3}. 
It is clear that the mapping $T\pi$ is surjective if and only if $\pi$ is surjective. 
We now prove that if $(M,\pi,\Gamma,S)$ 
is a pre-\H-chart then 
$(TM,T\pi,T\Gamma,TS)$ 
is also a pre-\H-chart. 
To this end we use the notation of Lemma~\ref{ParticularOpen} 
and in addition, for any $x\in M$ it is convenient to denote any vector in $u\in T_xM$ as a pair $(x,u)$.

One can now define an equivalence relation $\widehat{R}$  on $TM$ as follows: 
$$(x,u )\widehat{R} (y,v)\iff  
\text{there exists } h\in\Gamma 
\text{ with }h(x)=y\text{ and }(T_xh)u=v.$$
Let us denote by $\widehat{\pi}\colon TM\to TM/\widehat{R}$ the corresponding quotient map. 
It is clear that we have $(x,u )\widehat{R} (y,v)$ if and only if the quadruples $(M,\pi,x,u)$ and $(M,\pi,y,v)$ are equivalent in the sense of Definition~\ref{tgH_def}, hence 
one has a natural injective map 
$\iota \colon TM/\widehat{R}\to TS$ (which is surjective if $\pi$ is surjective) for which the diagram 
$$\xymatrix{& TM \ar[dl]_{\widehat{\pi}}\ar[dr]^{T\pi}& \\
TM/\widehat{R} \ar[rr]^{\iota}& & TS}$$
is commutative. 
Therefore it suffices to show that 
$(TM,\widehat{\pi},T\Gamma,TS)$ 
is a pre-\H-chart. 

To this end we note that for arbitrary $(x,u),(y,v)\in TM$ with 
$(x,u )\widehat{R} (y,v)$ we have $\pi(x)=\pi(y)$ 
and then the 
diffeomorphism $h\colon W\to W'$ with 
$h\in\Gamma$ and 
$h(x)=y$ gives rise to the diffeomorphism 
$Th\colon TW\to TW'$ 
between the open neighbourhoods $TW$ and $TW'$ of the points $(x,u)\in TM$ and $(y,v)\in TM$, respectively. 
We have $Th\in T\Gamma$.
Moreover, since $\pi\circ h=\pi\vert_W$, 
we obtain $\widehat{\pi}\circ Th=\widehat{\pi}\vert_{TW}$. 
This shows that 
$(TM,\widehat{\pi},T\Gamma,TS)$ 
is a pre-\H-chart.

It remains to prove that 
if $(M,\pi,\Gamma,S)$  is an \H-atlas, 
then $(TM,T\pi,T\Gamma,TS)$
is an \H-atlas. 
To this end we 
note that 
the equivalence relation defined by the action of 
the diffeomorphism pseudogroup $T\Gamma$ 
on $TM$ (cf. Lemma~\ref{pseudogroup}) coincides with the above equivalence relation $\widehat{R}$. 
Moreover, the action of  
$T\Gamma$ 
on $TM$ is free. 
In fact,  
if we have $(x,u)\in TM$  and $h\in \Gamma(\pi)$ with $Th(x,u)=(x,u)$, 
then $h(x)=x$, hence the germ of $h$ at $x\in M$ is the identity (since 
$(M,\pi,\Gamma,S)$ 
is an \H-atlas hence the action of $\Gamma$ on $M$ is free). 
Therefore the germ of $Th$ at $(x,u)\in TM$ coincides with the germ of the identity map of~$TM$. 
It now follows by Proposition~\ref{actionG} that the quotient map 
$\widehat{\pi}\colon TM\to TM/\widehat{R}$ 
defines 
an \H-atlas 
$(TM,\widehat{\pi},T\Gamma,TS)$.  
Consequently, by the above commutative diagram, 
$(TM,T\pi,T\Gamma,TS)$ 
is an \H-atlas as well. 

\ref{preHtg_item4}. 
To check that the diagram \eqref{Tpi} is commutative, let $x_0\in M$ and $v\in T_{x_0}M$. 
We have $(T\pi)(v)=(T_{x_0}\pi)(v)$ ($\in T_{\pi(x_0)}S$).  
Therefore $p_S((T\pi)(v))=\pi(x_0)=\pi(p_M(v))$, 
and thus $p_S\circ T\pi=\pi\circ p_M$. 

To check that the projection $p_S\colon TS\to S$ is surjective, 
we need the fact that $S$ is a pre-\H-manifold, hence it has a pre-\H-atlas 
$(M',\pi',\Gamma',S)$. 
Using the commutative diagram \eqref{Tpi} for this pre-\H-atlas, 
we see that the mapping $p_{M'}\colon TM'\to M'$ is a lift of $p_S\colon TS\to S$ 
with respect to the pre-\H-charts 
$(M',\pi',\Gamma',S)$ and $(TM',T\pi',T\Gamma',S)$
(cf. \ref{preHtg_item3}.\ above). 
Since $p_{M'}\colon TM'\to M'$ is a submersion of Banach manifolds, it then follows that $p_S\colon TS\to S$ is a submersion of pre-\H-manifolds. 

Finally, using the trivial commutative diagram 
\begin{equation*}
\xymatrix{
	M 
	\ar[r]^{\id_M} 
	\ar[d]_{\id_M}  
	&M \ar[d]^{\pi} 
	\\
	M  \ar[r]^{\pi} 
	&S
}
\end{equation*}
where the vertical arrow $\id_M\colon M\to M$ is regarded as a pre-\H-chart of the Banach manifold~$M$, we see that the horizontal arrow $\id_M$ is a lift of $\pi$, hence $\pi$ is both a submersion and an immersion of pre-\H-manifolds. 
\end{proof}

\begin{definition}
\label{fields}
\normalfont 
Let $S$ be an \H-manifold, hence $TS$ is again an \H-manifold by Proposition~\ref{preHtg}. 
A (smooth) \emph{tangent vector field} on $S$ is a smooth mapping $V\colon S\to TS$ satisfying $V(s)\in T_s S$ for every $s\in S$. 
We denote by $\Xc(S)$ the set of all tangent vector fields on $S$, 
and it is clear that $\Xc(S)$ is a real vector space with pointwise addition and scalar multiplication. 
\end{definition}

\begin{definition}
	\label{fields_invar}
\normalfont
If $M$ is a Banach manifold, and $\Gamma$ is a 
diffeomorphism pseudogroup on~$M$, 
then we denote by $\Xc^\Gamma(M)$ the set of all 
smooth vector fields  $X\colon M\to TM$ 
which are \emph{$\Gamma$-invariant}, that is, 
satisfy the condition $(T_xh)(X(x))=X(h(x))$ for every diffeomorphism $h\in\Gamma$ and every point $x\in U$, where $h\colon U\to U'$. 

Since the diffeomorphisms preserve the Lie bracket of vector fields, 
it follows that $\Xc^\Gamma(M)$ is a subalgebra of the Lie algebra $\Xc(M)$ of all (smooth, global) vector fields on~$M$.  
\end{definition}

\begin{remark}
\label{fields_bracket}
\normalfont 
If $(M,\pi,\Gamma,S)$ is an \H-atlas 
(cf. Definition~\ref{pre_H_def})
then it is easily checked, using Proposition~\ref{preHtg}, that the mapping 
$$\pi_*\colon \Xc^\Gamma(M)\to \Xc(S),\quad (\pi_*(X))(\pi(x)):=(T_x\pi)(X(x))\in T_{\pi(x)}S,$$
is a well-defined isomorphism of real vector spaces. 
We just note that, for any $X\in \Xc^\Gamma(M)$, the mapping $\pi_*(X)\colon S\to TS$ is smooth since it has the smooth lift $X\colon M\to TM$ with respect to the \H-atlases $\pi\colon M\to S$ and $T\pi\colon TM\to TS$. 

Since $\Xc^\Gamma(M)$ is a Lie algebra of vector fields on~$M$, 
we endow $\Xc(S)$ with the unique Lie bracket for which the mapping $\pi_*$ is a Lie algebra isomorphism. 
\end{remark}

If $f\colon S_1\to S_2$ is a pre-\H-smooth map between two pre-\H-manifolds $S_1$ and $S_2$,  
the map $Tf\colon TS_1\to TS_2$ is well defined. 
Note that if $(M_i,\pi_i, \Gamma_i,S_i)$ is a pre-\H-atlas for $S_i$, for $i=1,2$, 
then for any $x_i\in M_i$ we have a neighbourhood $U_i$ of $x_i$ and smooth  map $\hat{f}\colon U_1\to U_2$ which is a local lift of $f$, that is, $\pi_2\circ \hat{f}=f\circ \pi_1\vert_{U_1}$. 
Then moreover $T\hat{f}\colon TU_1\to TU_2$ is a local lift of $Tf$. 

 We end this section by the following result  
 on pre-\H-diffeomorphisms.

\begin{proposition}
\label{EtalMapB-Manifold}  
If $S_1$ and $S_2$ are pre-\H-manifolds 
and $f\colon S_1\to S_2$  is a pre-\H-smooth bijective map for which $T_sf\colon T_sS_1\to T_{f(s)}S_2$ is bijective for every $s\in S_1$, 
then the inverse mapping  $f^{-1}\colon S_2\to S_1$ is pre-\H-smooth. 
Equivalently, the mapping  $f\colon S_1\to S_2$ is a pre-\H-diffeomorphism. 
\end{proposition}
 
\begin{proof}  
Let $s_1\in S_1$ and $s_2\in S_2$ with $f(s_1)=s_2$ and 
$(M_j,\pi_j, \Gamma_j,S_j)$ 
be a pre-\H-atlas for $j=1,2$. 
Select any $x_j\in \pi_j^{-1}(s_j)$. 
We must prove that there exists an open neighborhood $U_j$ of $x_j$ and a smooth map $\beta\colon U_2\to U_1$ such that $\pi_1\circ \beta=f^{-1}\circ {\pi_2}\vert_{U_2}$.  

Since $f$ is smooth, there exists a smooth map $\hat{f}\colon V_1\to V_2$ with $x_j\in V_j\subseteq M_j$ for $j=1,2$, $\hat{f}(x_1)=x_2$, and $f\circ {\pi_1}\vert _{V_1}=\pi_2\circ \hat{f}$. 
The differential $T_{s_1}f\colon T_{s_1}S_1\to T_{s_2}S_2$ is invertible, therefore the differential $T_{x_1}\hat{f}\colon T_{x_1} M_1\to T_{x_2}M_2$ is invertible. 
Then, since both $M_1$ and $M_2$ are Banach manifolds, it follows by the local inversion theorem that there exists open neighborhoods $V'_1$ of $x_1$  and $V'_2=\hat{f}(V'_1)$  of $x_2=\hat{f}(x_1)$ such that  that $\hat{f}\vert_{V'_1}\colon V'_1\to V'_2$ is a diffeomorphism.   
If we define $U_1:=V_1\cap V'_1$ and $U_2:=\hat{f}(W_1)$ it follows that $U_1\subseteq M_1$ and $U_2\subseteq M_2$ are open subsets with $x_j\in U_j$ for $j=1,2$ and $\hat{f}\vert_{U_1}\colon U_1\to W_2$ is a diffeomorphism. 
On the other hand, since $f\circ {\pi_1}\vert _{V_1}=\pi_2\circ \hat{f}$ 
and $U_j\subseteq V_j$ for $j=1,2$, we obtain $f\circ {\pi_1}\vert _{U_1}=\pi_2\circ \hat{f}\vert_{U_1}$, 
which further implies $ {\pi_1}\circ(\hat{f}\vert_{U_1})^{-1}=f^{-1}\circ \pi_2\vert _{W_2}$. 
Since the mapping $\beta:=(\hat{f}\vert_{W_1})^{-1}\colon U_2\to U_1$ is smooth 
and $U_j$ is an open neighbourhoods of the point $x_j\in \pi_j^{-1}(s_j)$ for $j=1,2$, 
and we have $ {\pi_1}\circ\beta=f^{-1}\circ \pi_2\vert _{W_2}$
it follows that the mapping $f^{-1}\colon S_2\to S_1$ is pre-\H-smooth.  
\end{proof}

\section{\H-manifold structures on leaf spaces of  regular foliations
}\label{Sect3}
This section contains a study of the holonomy for regular foliations in the general framework of Banach manifolds, 
culminating with our main construction of the natural structure of an \H-manifold on the leaf space of a foliation without holonomy  (Theorem~\ref{QuotientBanachManifold}). 
As an illustration of that general result, we prove that the homogeneous space defined as the quotient of  Banach-Lie subgroup by a split immersed subgroup has the natural structure of a \H-manifold (Proposition~\ref{G/K}). 
These results play a key role in the Lie theory of \H-groups developed in Section~\ref{Sect4}. 

\subsection{Holonomy transport of a regular foliation}\label{holonomy}
Consider a regular foliation~$\Fc$ in a Banach manifold $M$ (modeled on the Banach space $\mathbb{M}$)  and let  $E=T\Fc\subseteq TM$ be the corresponding subbundle tangent to $\Fc$. 
(See \cite{Lan}, and also \cite{Pel} for a more general framework.)  
We denote by $\EE$   the typical fiber of $E$ and fix a closed linear space   $\FF$  such that $\mathbb{M}=\EE\oplus \FF$. 
For every open subset $V\subseteq M$ we denote by  $\Fc_V$ the induced foliation  on $V$.  
According to the proof of Frobenius Theorem, for each $x\in M$, there exist an open subset $V\subseteq M$ with $x\in V$ and a chart $\varphi\colon V\to \EE\times \FF$ such that
 if $p_2\colon \EE\times \FF\to\FF$ is the Cartesian projection onto the second factor, we have
\begin{enumerate}[{\rm(1)}]
	\item\label{fol_item1} $\varphi(V)=\bar{U}\times\bar{T}$ and $\varphi\colon V\to \bar{U}\times\bar{T}$ is a diffeomorphism, where  $\bar{U}\subseteq \EE$ and $\bar{T}\subseteq\FF$ are simply connected open subsets; 
	\item\label{fol_item2} $p_2\circ \varphi:V\to \bar{T}$ is a submersion;
	\item\label{fol_item3} $\Fc_V$ is the foliation on $V$ defined by the fibers of the submersion $p_2 \circ\varphi$.
\end{enumerate}
We call $(V,\varphi) $ a \emph{foliated chart} and each fiber $P_x=(p_2\circ\varphi)^{-1}(x)$ for $x\in\bar{T}$  is called a \emph{plaque} of $\Fc$ in $V$.
Moreover if $(V_i,\varphi_i)$ for $i=1,2$, are two foliated charts with $V_1\cap V_2\not=\emptyset$, then 
\begin{enumerate}[{\rm(1)}]
	\setcounter{enumi}{3}
	\item\label{fol_item4} every $x\in V_1\cap V_2$ has a neighbourhood of type
	$\varphi_1^{-1}(\bar{U}_1'\times \bar{T}'_1)$ with $\bar{U}_1'\times \bar{T}_1'\subseteq \bar{U}_1\times \bar{T}_1$ such that 
	$$(\varphi_2\circ\varphi_1^{-1})\vert_{\bar{U}_1'\times \bar{T}_1'} :  \bar{U}_1'\times \bar{T}_1'\to \bar{U}_2\times \bar{T}_2$$
	is a local diffeomorphism of the form
	\begin{equation}\label{changechart}
	(\bar{x},\bar{t})\mapsto (f(\bar{x},\bar{t}),g(\bar{t}))\in \bar{U}_2\times \bar{T}_2
	\end{equation}
	In particular, $g\colon \bar{T}_1\to\bar{T}_2$ is a local diffeomorphism around $(p_2\circ\varphi)(x)\in\bar{T}_1$.
\end{enumerate}

\begin{definition} \label{transversal}
\normalfont
A Banach submanifold $\Tc$  of $M$ is a called a (local) \emph{transversal} of the foliation $\Fc$, if there exists   a foliated chart $(V,\varphi)$ such that $\Tc=\varphi^{-1}(\bar{x})\times \bar{T}$  for some $\bar{x}\in \bar{U}$ (according to the previous notations).
\end{definition}

As in finite dimensions, a regular foliation $\Fc$ in a Banach manifold $M$ can be alternatively defined by a family of local charts $\{(V_\alpha,\varphi_\alpha)\}_{\alpha\in A}$ with the above properties \eqref{fol_item1}--\eqref{fol_item3}
for each $\alpha\in A$, such that  $\{V_\alpha\}_{\alpha\in A}$ is a covering of $M$ and, for each $V_\alpha\cap V_\beta\not=\emptyset$, the map  $\varphi_\beta\circ\varphi_\alpha^{-1}$ satisfies the condition~\eqref{fol_item4}. 

As above, let $\Fc$ be a regular foliation $\Fc$ associated to a subbundle $T\Fc=E\subseteq TM$ whose typical fiber is $\EE$   and fix a direct sum decomposition $\mathbb{M}=\EE\oplus\FF$ as above. 
Consider the open unit balls $B_1\subseteq \EE$ and $B_2\subseteq\FF$ 
with respect to some norms that define the topology of on $\EE$ and $\FF$, respectively.
For every $x\in M$ there exists a foliated chart $(V,\varphi)$ with 
$x\in V$,  $\varphi(x)=(0,0)\in \EE\times \FF$, and $\varphi(V)=B_1\times B_2$. 

If the Banach manifold $M$ is $C^0$-paracompact, as in finite dimensions 
(cf. \cite[Lemma 1.3.3]{Wal})  we we can construct an atlas $\Vc=\{(V_\alpha,\varphi_\alpha)\}_{\alpha\in A}$  with the following properties: 
\begin{enumerate}[{\rm(i)}]
	\item\label{nice_item1} the covering $\{V_\alpha\}_{\alpha\in A}$ is locally finite;
	\item\label{nice_item2} we have $\varphi_\alpha(V_\alpha)=B_1\times B_2$;
	\item\label{nice_item3}  if $V_\alpha\cap V_\beta\not=\emptyset$ then the closure of $V_\alpha\cup V_\eta$ is contained in a foliated chart.
\end{enumerate}
In general, any atlas $ \mathcal{V}$ of a Banach manifold with the above properties will be  called a \emph{nice atlas}. 
A foliation of a (not necessarily $C^0$-paracompact) Banach manifold for which a nice atlas exists is called a \emph{nice foliation}.

\begin{remark}\label{IntersectionPlaques} 
\normalfont 
If $V_\alpha$ and $V_\beta$ are chart domains of a nice atlas, and if $V_\alpha\cap V_\beta\not=\emptyset$, 
then from property~\eqref{nice_item3} 
each plaque of $V_\alpha$ meets at most one 
plaque of $V_\beta$.
\end{remark}

\begin{convention}\label{Nicefoliation}
	From now on we consider only nice foliations and so every regular foliation will be assumed to be a nice foliation.
\end{convention}

\begin{definition}\label{M/F}
\normalfont 
A  (nice)  atlas $\mathcal{V}$ of a Banach manifold $M$ 
defines a structure of \emph{foliated manifold} on $M$  
and a topology called the \emph{foliated topology}.\footnote{This terminology makes sense  for  any atlas of foliated charts even if it is not a nice atlas.}

Each connected component of the foliated topology is a called a \emph{leaf}.  
This partition into leaves of $M$  gives rise to an equivalence relation  on $M$ 
and we denote by $M/\Fc$ the quotient space provided with the  quotient topology of the original topology of the Banach manifold~$M$.
\end{definition}

\begin{remark}\label{LeafTopology} 
\normalfont 
Let $L$ be a leaf of a regular foliation $\Fc$ and consider 
any atlas $\Vc=\{(V_\alpha,\varphi_\alpha)\}_{\alpha\in A}$ 
consisting of foliated charts.   
We denote $\Pc(\alpha)_L$  
the set of plaques of  $V_\alpha$ which meet $L$ 
(i.e., the plaques which are contained in $L$)
and if $P\in\Pc(\alpha)_L$  
we define $\varphi_\alpha^P:=\varphi_\alpha\vert_P$. 
Then the family  $\{(P,\varphi_\alpha^P)\}_{\Pc(\alpha)_L, \alpha\in A}$ 
 is an atlas for the structure of Banach manifold of~$L$.
\end{remark}

 Consider a nice atlas $\Vc=\{(V_\alpha,\varphi_\alpha)\}_{\alpha\in A}$ of the Banach manifold $M$ and for every $\alpha\in A$ let $x_\alpha\in V_\alpha$ with  $\varphi_\alpha(x_\alpha)=(0,0)\in\EE\times \FF$.  
 We also define $$\Tc_\alpha:=\varphi_\alpha^{-1}(\{0\}\times B_2)\subseteq V_\alpha 
 \text{ and  }\bar{T}_\alpha:= \{0\}\times B_2,$$ 
 and $\pi_\alpha:=p_2\circ \varphi_\alpha\colon V_\alpha\to\FF$. 
If $V_\alpha\cap V_\beta\not=\emptyset$, let $g_{\alpha\beta}$ be the local diffeomorphism of $\FF$ associated to transition map $\varphi_\beta\circ\varphi_\alpha^{-1}$ as in \eqref{changechart}.  
Then there exist open subsets $\Sc_\alpha\subseteq\Tc_\alpha$ and $\Sc_\beta\subseteq\Tc_\beta$ with $x_\alpha\in\Sc_\alpha$ and $x_\beta\in\Sc_\beta$ for which the restricted mapping 
$$h_{\alpha\beta}:=g_{\alpha\beta}\vert_{\Sc_\alpha}
\colon\Sc_\alpha\to\Sc_\beta$$ is a diffeomorphism with $h_{\alpha\beta}(x_\alpha)=x_\beta$.

 \begin{remark}\label{plaqueqq}
\normalfont  
In the above notation, we note the following facts. 
\begin{enumerate}[1.]
\item
For every $x\in M$, there exists a  foliated  chart $(V,\varphi )$ with 
$\varphi(V)=B_1\times B_2\subseteq\EE\times \FF$, $x\in V$ and $\varphi(x)=(0,0)$. 
Then $\mathcal{P}_x:=\varphi^{-1}(B_1\times\{0\})$ is a plaque through $x$ 
and $\Tc_x:=\varphi^{-1}(\{0\}\times B_2)$  a local transversal  through $x$.
\item
If $\mathcal{P}_\alpha$ is the plaque through $x_\alpha$ and $z\in \mathcal{P}_\alpha$  
then $\Tc'_z:=\varphi^{-1}(\{\varphi(z)\}\times \bar{T}_\alpha)$ is a transversal through $z$. 
Therefore, for any $x\in V_\alpha$ we can find a transversal  $\Tc_x$ 
with $x\in\Tc_x$ and 
 $\varphi_\alpha(\Tc_x)=\bar{T}_\alpha$. 
\end{enumerate}
 \end{remark}

The disjoint union 
$$\Tc:=\bigsqcup_{\alpha\in A}\Tc_\alpha$$ 
will be called a \emph{global transversal} to $\Fc$. 
Now let 
$$\bar{T}:=\bigsqcup_{\alpha\in A}\bar{T}_\alpha.$$  
Since $\bar{T}_\alpha\subseteq\FF$ is an open subset for every $\alpha\in A$, we can provide $\bar{T}$ with a (Hausdorff) Banach manifold structure.  
Note that the mapping 
$$\varphi_\Tc:=\bigsqcup\limits_{\alpha\in A}(\varphi_\alpha\vert_{\Tc_\alpha})\colon\Tc\to\bar{T}$$ 
is bijective hence we can provide $\Tc$ with a (Hausdorff) Banach manifold structure for which the above bijective map $\varphi_\Tc$ is a diffeomorphism.

\emph{We will make no difference if there is no ambiguity in the context considered  
in which a global transversal is used then we shall identify $\Tc$ with $\bar{T}$}.

Now let $L$ be a leaf of the foliation $\Fc$ and  $c\colon [0,1]\to L$ be a piecewise smooth curve. 
We denote $y_0:=c(0) $ and $y_1:=c(1)$.  
According to Remark \ref{plaqueqq}, 
let $\Tc_{x_0}$ and $\Tc_{x_1}$ be 
two local transversals to $\Fc$ in $x_0$ and $x_1$ respectively.

 We can cover the compact subset $c([0,1])\subseteq M$ by a  minimal finite  subfamily 
 of foliated charts $(V_i,\varphi_i)$, $i=0,\dots,n$  of the nice atlas $\{(V_\alpha,\varphi_\alpha)\}_{\alpha \in A}$ such that $V_{i-1}\cap V_{i}\not=\emptyset$ for $i=1,\dots, n$.  
  Such a covering of $c([0,1]$ will be called a \emph{nice covering of $c$}. 
We can choose  a sequence $t_0=0\leq t_1\leq\cdots\leq t_i\leq\cdots\leq t_n\leq t_{n+1}=1$ 
with $x_i:=c(t_i)\in V_{i-1}\cap V_{i}$ for $i=1,\dots, n$, 
where $x_0=y_0=c(t_0)$ and $x_{n+1}=y_1=c(t_{n+1})$. 
For each $i=0,\dots,n+1$, according to Remark \ref{plaqueqq} 
we have  a local transversal $\Tc_i$ through $x_i$ such that $\Tc_0=\Tc_{x_0}$ and $\Tc_{n+1}=\Tc_{x_1}$ are 
the initial fixed local transversals.
As above, to the transition map $\varphi_{i+1}\circ \varphi_{i}^{-1}$  
there corresponds a local diffeomorphism $g_{i, i+1}$ of $\FF$ 
which induces a diffeomorphism $h_i\colon \Sc_{i}\to\Sc_{i+1}$ 
for suitable open subsets $\Sc_i\subseteq\Tc_i$ with $x_i\in\Sc_i$ and $h_i(x_i)=x_{i+1}$ for $i=0,\dots,n$. 
Then the composition 
$$h_c:=h_n\circ\cdots\circ h_0\colon\Sc_0\to\Sc_{n+1}$$ 
is a local diffeomorphism from $\Tc_0$  
into $\Tc_{n+1}$  with 
$h(y_0)=y_1$.

Just as in the case of finite-dimensional manifolds (see e.g., \cite{Wal})
the above map~$h_c$ is called the \emph{\it holonomy transport along the curve~$c$}.
Note that according to Remark \ref{plaqueqq} and the definition of~$\Tc$, $h_c$ 
belongs to the pseudogroup $\Gamma(\Tc)$ of all local diffeomorphisms of~$\Tc$. 
In order to emphasize the role of the nice atlas $\Vc$ in the above construction, we denote by 
$\Gamma(\Vc,\Tc)$ 
the pseudogroup of $\Vc$ generated in the sense of \cite[Def. 1.1.3]{Wal} by the holonomy transport  along curves in the leaves of $\Fc$. 

\begin{remark}\label{chgtparameter}
	\normalfont 
With the above notation we record the following facts: 
	\begin{enumerate}
\item[1.] Let $c\colon [0,1]\to L$ be a piecewise smooth piecewise and we define the curve  $c^-\colon  [0,1]\to L$,  
$c^-(s)=c(1-s)$. 
Then the holonomy transport with respect to  these curves satisfies 
$h_{c^-}=(h_{c})^{-1}$.
 \item[2.] Since the holonomy transport depends only of the family of plaques which cover $c([0,1])$, 
the corresponding local diffeomorphism does not depends of the parametrization of~$c$. 
\end{enumerate}
\end{remark}
 
 More generally, we have the following result in the above setting of regular foliations on Banach manifolds, which is similar to the finite-dimensional case, cf. e.g., \cite[page 11]{Wal}.
 
 \begin{proposition}\label{dhcDependHomotopy} 
The holonomy transport $h_c$ along a curve $c\colon [0,1]\to L$ in a leaf $L$ of the regular foliation $\Fc$ 
depends only of the class of $C^0$-homotopy of $c$ in $L$ with fixed endpoints. 
 \end{proposition}
 
 \begin{proof} 
Denote $x_0:=c(0),x_1:=c(1)\in L$ and consider another piecewise smooth curve $c'\colon[0,1]\to L$ 
such that $x_0=c'(0)$ and $x_1=c'(1)$ which is $C^0$-homotopic to $c$ with fixed endpoints. 
There exists a continuous map $\bar{c}\colon[0,1]\times [0,1]\to L$ such that 
\begin{align*}
\bar{c}(s,0) & =x_1\text{ and }\bar{c}(s,1)=x_1\text{ for all }s\in [0,1], \\
\bar{c}(0,t) &=c(t)\text{ for all }t\in [0,1],\\
\bar{c}(1,t) &=c'(t)\text{ for all }t\in [0,1]. 
\end{align*}
Since $K:=\bar{c}([0,1]\times [0,1])$ is a compact subset of $L$, 
we can cover $K$ by a finite number of chart domains $V_1,\dots, V_n$.   
From this covering  we have also a finite number of plaques $P_1,\dots, P_m$ which cover $K$.  
Since $c(\cdot)=c(0,\cdot)$ and $c'(\cdot)=c(1,\cdot)$, we obtain 
$$c([0,1])\subseteq\bar{c}([0,1]\times[0,1])=K\subseteq P_1\cup\cdots\cup P_m$$
and it then follows  by the construction of $h_c$ and $h_{c'}$, using the same family of plaques $P_1,\dots, P_m$, that $h_c=h_{c'}$. 
  \end{proof}

\begin{definition}\label{withoutholonomy_def} 
\normalfont 
A regular foliation $\Fc$ of a Banach manifold $M$ is called \emph{without holonomy} 
if for every leaf $L$, every point $x\in L$, and every smooth piecewise  closed curve  $c\colon[0,1]\to L(x)$    
with $c(0)=c(1)=x$, the germ of the holonomy transport $h_c$ at $x$ is the identity.
 \end{definition}

\begin{corollary}\label{simplyconnected}  
Let  $\Fc$ be a regular foliation $\Fc$ of a Banach manifold $M$ such that 
each leaf is simply connected. 
Then $\Fc$ is without holonomy.
\end{corollary}

\begin{proof}
Use Proposition \ref{dhcDependHomotopy}. 
\end{proof}

 Given a nice atlas  $\Vc=\{(V_\alpha,\varphi_\alpha)\}_{\alpha\in A}$ and  global transversal $\Tc=\bigsqcup\limits_{\alpha\in A}\Tc_\alpha$ as above, we have the natural mapping 
 $$\tau\colon\Tc\to M, \quad \tau(x)=x.$$ 
 Since the tangent space of $\Tc_\alpha$ is split in $TM$ by construction, this implies that $\tau$ is a split immersion. 
 Let  $c\colon[0,1]\to L$ be a piecewise curve in a leaf $L$ which joins $x_{\bar{t}_0}\in \Tc$ to $x_{\bar{t}_1}\in \Tc$  with its corresponding  holonomy transport $h_c$, which is 
 a local diffeomorphism of $\Tc$.   
 
 Recalling the holonomy pseudogroup $\Gamma(\Vc,\Tc)$ on $\Tc$ introduced before Remark~\ref{chgtparameter}, we now consider equivalence relation ${R}_\Tc$ defined by $\bar{t} \backsim \bar{t}'$ if and only if there exists $h\in \Gamma(\Vc,\Tc)$ such that $h(\bar{t})=\bar{t}'$. 
 
 \begin{proposition}\label{PFequivTG}  
The equivalence relation $R_\Tc$ is the equivalence relation on $\Tc$ induced by  $\Fc$ and the above mapping $\tau\colon \Tc\to M$ is a complemented immersion which gives rise to a  
bijective mapping $\bar{\tau}\colon \Tc/\Gamma(\mathcal{V},\Tc)\to M/\Fc$.
\end{proposition}

\begin{proof}
 Denote by $[\bar{t}]\in\Tc/\Gamma(\Vc,\Tc)$ the equivalence class 
 (i.e., the orbit) of $\bar{t}\in \Tc$ under the action of the pseudogroup $\Gamma(\mathcal{V},\Tc)$. 
 This means that  
 $\bar{t}'\in [\bar{t}]$ if and only if there exists $h\in \Gamma(\mathcal{V},\Tc)$ with $h(\bar{t})=\bar{t}'$. 
 But by the definition of $\Gamma(\mathcal{V},\Tc)$ and
  Remark~\ref{chgtparameter}, there exists piecewise smooth curves $c_1,\dots,c_k$ into a leaf $L$ such that $h=h_{c_1}\circ\cdots\circ h_{c_k}$. 
 It follows  that $\tau([\bar{t}])\subset L(\bar{t})$, 
 where $L(\bar{t})$ denotes the leaf through $\bar{t}\in\Tc\subseteq M$. 
 This shows that there exists a mapping $\bar{\tau}$ that makes the following diagram commutative, 
 \begin{equation}
 \label{PFequivTG_proof_eq1} 
 \xymatrix{ \Tc \ar[r]^{\tau} \ar[d]_{\pi_\Tc} & M \ar[d]^{\pi_\Fc}\\
 \Tc/\Gamma(\Vc,\Tc) \ar[r]^{\bar{\tau}} & M/\Fc}
 \end{equation}
 where the vertical arrows are the quotient mappings. 
 
We will prove that the mapping $\bar{\tau}$ is bijective, by constructing its inverse mapping. 
To this end, we first define the mapping 
$\sigma\colon M\to \Tc/\Gamma(\Vc,\Tc)$ 
in the following way: 
For arbitrary $x\in M$  there exists $\alpha\in A$ with $x\in V_\alpha$. 
Then 
there exists a unique plaque $\Pc_x$  in $V_\alpha$ with $x\in\Pc_x$  and $\Pc_x$ meets the local transversal~$\Tc_\alpha$ in an unique point $\bar{x}_\alpha$.
If now $x\in V_\beta$ for some other index $\beta\in A$, then, 
in the same way, we obtain a unique point $\bar{x}_\beta\in \Tc_\beta\cap\Pc_x$. 
But by the construction of $\Gamma(\Vc,\Tc)$,  the points $\bar{x}_\alpha$ and $\bar{x}_\beta$ belong to the same orbit of $\Gamma(\mathcal{V},\Tc)$. 
That is, using the notation introduced above, we have $[\bar{x}_\alpha]=[\bar{x}_\beta]\in\Tc/\Gamma(\Vc,\Tc)$. 
This discussion shows that we obtain a well-defined mapping 
$\sigma\colon M\to \Tc/\Gamma(\Vc,\Tc)$ given by 
$$\sigma\colon M\to \Tc/\Gamma(\Vc,\Tc),\quad 
\sigma(x):=[\bar{x}_\alpha]\in\Tc/\Gamma(\Vc,\Tc)
\text{ if } x\in V_\alpha\text{ and }\Tc_\alpha\cap \Pc_x=\{\bar{x}_\alpha\}.$$
We now claim that for every leaf $L\in M/\Fc$, if $x, y\in L$, then
$\sigma(x)=\sigma(y)\in\Tc/\Gamma(\Vc,\Tc)$. 
In fact, there exists a piecewise smooth curve $\gamma \colon [0,1]\to L$ with $\gamma(0)=x$ and $\gamma(1)=y$.
As we have seen above, $\gamma([0,1])$ can be covered by a finite number of plaques $\{\Pc_i\}_{i=0,\dots,n}$ such that $\Pc_{i-1}\cap \Pc_{i}\not=\emptyset$ 
so if $x\in V_\alpha$ and $y\in V_\beta$ then $\sigma(x)=[\bar{x}_\alpha]$ 
and $\sigma(y)=[\bar{x}'_\beta]$, where  $\{\bar{x}_\alpha\}= \mathcal{P}_0\cap \Tc_\alpha$ and $\{\bar{x}'_\beta\}=\mathcal{P}_n\cap \Tc_\beta$. 
By the construction of $\Gamma(\mathcal{V},\Tc)$, the points  $\bar{x}_\alpha$ and $\bar{x}'_\beta$ belong to the same orbit of $\Gamma(\mathcal{V},\Tc)$, 
that is, $[\bar{x}_\alpha]=[\bar{x}'_\beta]\in\Tc/\Gamma(\Vc,\Tc)$. 
Equivalently, $\sigma(x)=\sigma(y)\in\Tc/\Gamma(\Vc,\Tc)$, as claimed above. 

This implies that there exists the well-defined  mapping 
$$\bar{\sigma}\colon M/\Fc\to\Tc/\Gamma(\mathcal{V},\Tc),\quad 
\bar{\sigma}(L):=\sigma(x)\text{ if }L\in M/\Fc\text{ and }x\in L,$$
 which is surjective by its construction.  
We also have $\bar{\tau}\circ \bar{\sigma}(L)=L$  for every $L\in M/\Fc$, which implies that both mappings $\bar{\sigma}$ and $\bar{\tau}$ are bijective and inverse to each other. 
Moreover, this implies  that  $R_{\Tc}$ is the equivalence relation which the pull-back by $\tau$ of the equivalence relation defined by $\Fc$. 
This completes the proof. 
\end{proof}

\begin{corollary}
	\label{PFequivTG_cor}
	If $\Fc$ be a nice foliation~of the Banach manifold~$M$ with a global transversal~$\Tc$ associated to a nice atlas~$\Vc$, 
	and consider 
	the quotient maps 
	$$\pi_\Tc\colon\Tc\to\Tc/\Gamma(\Vc.\Tc)\text{ and }\pi_\Fc\colon M\to M/\Fc.$$ 
	Then the following assertions hold: 
	\begin{enumerate}[{\rm(i)}]
		\item\label{PFequivTG_cor_item1} 
		The quadruple  $(\Tc,\pi_\Tc,\Gamma(\Vc.\Tc),\Tc/\Gamma(\Vc.\Tc))$ is a    pre-\H-atlas.
		\item\label{PFequivTG_cor_item2}  
		The quadruple  $(\Tc,\bar{\tau}\circ \pi_\Tc,\Gamma(\Vc.\Tc),M/\Fc)$ is a    pre-\H-atlas.
		\item\label{PFequivTG_cor_item3}  The quotient map  
		$\pi_\Fc\colon M\to M/\Fc$ 
		is a submersion of pre-\H-manifolds,  
		 one has the pre-\H-diffeomorphism $\bar{\tau}\colon \Tc/\Gamma(\mathcal{V},\Tc)\to M/\Fc$, and the diagram~\eqref{PFequivTG_proof_eq1} is commutative. 
		\item\label{PFequivTG_cor_item4}  
		If moreover the foliation $\Fc$ is without holonomy, then the above pre-\H-atlases are actually \H-atlases. 
	\end{enumerate}
	The above assertions hold for every regular foliation on a $C^0$-paracompact Banach manifold. 
\end{corollary}

\begin{proof}
\eqref{PFequivTG_cor_item1} This follows by Proposition~\ref{actionG}.

\eqref{PFequivTG_cor_item2}
	The mapping $\bar{\tau}\colon \Tc/\Gamma(\mathcal{V},\Tc)\to M/\Fc$ 
	is bijective by Proposition~\ref{PFequivTG}, 
	hence the assertion follows by Assertion~\eqref{PFequivTG_cor_item1}.

\eqref{PFequivTG_cor_item3}
Let $x_0\in M$ arbitrary and a foliated chart $\varphi\colon V\to B_1\times B_2$ with $x_0\in V$. 
Here $B_2\subseteq \Tc$, hence the Cartesian projection $p_2\colon V\to B_1$ 
fits in the commutative diagram 
$$\xymatrix{
V \ar[r]^{p_2} \ar@{^{(}->}[d]& \Tc \ar[d]^{\bar{\tau}\circ \pi_\Tc}\\
M \ar[r]^{\pi_\Fc} & M/\Fc
}
$$
which shows that the submersion $p_2$ is a local lift of the quotient map $\pi_\Fc$. 
Therefore $\pi_\Fc\colon M\to M/\Fc$ 
is a submersion of pre-\H-manifolds.

To prove that the mapping $\bar{\tau}$ is a pre-\H-diffeomorphism, 
we first recall that it is bijective. 
Moreover, we have the commutative diagram 
$$\xymatrix{ \Tc \ar[r]^{\id_\Tc} \ar[d]_{\pi_\Tc} & \Tc \ar[d]^{\bar{\tau}\circ\pi_\Tc}\\
	\Tc/\Gamma(\Vc,\Tc) \ar[r]^{\bar{\tau}} & M/\Fc}
$$
which shows that the diffeomorphism $\id_\Tc\colon\Tc\to\Tc$ is a lift of $\bar{\tau}$, 
hence $\bar{\tau}$ is an \'etale map of pre-\H-manifolds. 
Since we have seen that $\bar{\tau}$  is bijective, it follows that it is a pre-\H-diffeomorphism. 

\eqref{PFequivTG_cor_item4}  We have only to prove that for any $x\in \Tc$, 
		every element of the holonomy pseudogroup $\Gamma(\Vc,\Tc)$ which fixes the point~$x$ has its germ at~$x$ equal to the identity. 
		But this is precisely the definition of a regular foliation without holonomy. 

The final assertion follows from the existence of nice atlases for regular foliations on on $C^0$-paracompact Banach manifolds, noted after Definition~\ref{transversal}. 
\end{proof}

Let  $\pi \colon E\to M$ be a locally trivial bundle with typical fiber  $\EE$. 
A \emph{fibering atlas} is a collection $\Uc=\{(U_\alpha,\phi_\alpha)\}_{\alpha\in A}$ such  that 
$\{U_\alpha\}_{\alpha\in A}$ is a covering of $M$ and such that $\Phi_\alpha\colon \pi^{-1}(U_\alpha)\to U_\alpha\times \EE$ is a diffeomorphism  with the following commutative diagram:
$$
\xymatrix{
 \pi^{-1}(U_\alpha)  
\ar[r]^{\Phi_\alpha} 
\ar[d]_{\pi} 
& U_\alpha\times \mathbb{E} 
\ar[d] 
\\ 
U \ar[r]^{\id} 
& U_\alpha
}
$$
and the transition map  $\Phi_\beta\circ \Phi_\alpha^{-1}(x,u)=(x, g_{\alpha\beta}(x)(u))$ for $x\in U_\alpha\cap U_\beta$ and $g_{\alpha\beta}\colon U_\alpha\cap U_\beta\to\rm{Diff}(\mathbb{E})$ is a smooth field of diffeomorphisms of $\EE$. 
If we have such an atlas with constant fields  $g_{\alpha\beta}$,  for all $\alpha,\beta\in A$ with $U_\alpha\cap U_\beta\ne\emptyset$, then the fiber bundle $\pi \colon E\to M$ is called \emph{flat}. 
This condition is equivalent to the existence of a regular foliation $\Fc$ such that  for each $e\in E$ the tangent space $T_eL$ of the leaf through $e$ is the complemented space to the vertical space $T_eE_x$ of the fiber $E_x:=\pi^{-1}(x)$. 
In this case, each leaf $L$ is a (split) immersed submanifold of $E$ and the restricted map of $\pi\vert_L\colon L\to M$ is a local diffeomorphism.  
Moreover each fiber $E_x$ is a transversal of $\Fc$ in the previous sense and so the holonomy pseudogroup $\Gamma(\Uc,\Fc)$ of $\Fc$  is generated by $\{g_{\alpha\beta}\mid \alpha,\beta\in A\}$ and is a sub-pseudogroup of the pseudogroup of all local diffeomorphisms of $E$. 
In these conditions the quadruple $(E,\pi, M, \Fc)$ is called a \emph{foliated fiber bundle}.

\subsection{\H-manifold  structure on leaf spaces}
 
 We begin this section by an example of regular foliation whose leaf space has a structure of \H-manifold:
 
 \begin{proposition}\label{withoutholonomy_prop} 
Let $\Fc$ be a nice foliation of a Banach manifold $M$. 
If $\Fc$ is without holonomy, 
then the quotient space  $M/\Fc$  is an \H-manifold.
\end{proposition}

\begin{proof}
This follows by Corollary~\ref{PFequivTG_cor}\eqref{PFequivTG_cor_item4}. 
  \end{proof}

\begin{corollary}\label{simplyconnectedleaf} 
	Let $\Fc$ be a nice foliation of a Banach manifold $M$ whose all leafs are simply connected.  
	Then the quotient space  $M/\Fc$  is an \H-manifold.
\end{corollary}

\begin{proof}
	Since the leaves are simply connected, it follows by Corollary~\ref{simplyconnected} that the foliation~$\Fc$ is without holonomy, 
	hence Proposition~\ref{withoutholonomy_prop} applies. 
\end{proof}

The following lemma is needed in the proof of Theorem~\ref{QuotientBanachManifold} below.
 
 \begin{lemma}
 \label{double}
 Let $\Xc$ and $\Yc$ be real Banach spaces and  $P_j\colon\Xc\to\Yc$ be surjective continuous linear operators with split kernels for $j=1,2$, for which the linear operator 
 $P\colon \Xc\to\Yc\times\Yc$, $x\mapsto (P_1x,P_2x)$, 
 is injective and its range is a split closed subspace of $\Yc\times\Yc$. 
 Then $P_1(\Ker P_2)$  is a split closed linear subspace of $\Yc$. 
 \end{lemma}
 
 \begin{proof}
 	First note that 
 \begin{align*}
 P(\Ker P_2)
 &=\{(P_1x,P_2x)\mid x\in\Ker P_2\}=\{(P_1x,0)\mid x\in\Ker P_2\} \\
 &=P_1(\Ker P_2)\times\{0\}.
\end{align*}
 By hypothesis, $\Ker P_2$ is a split closed linear subspace of $\Xc$ while 
  $P\colon \Xc\to \Ran P$ is an isomorphism of Banach spaces 
  (where $\Ran P$ is a closed subspace of $\Yc\times\Yc$), 
  hence $P(\Ker P_2)$ is a split closed linear subspace of $\Ran P$. 
  Thus, by the above equalities, $P_1(\Ker P_2)\times\{0\}$ is a split closed linear subspace of $\Ran P$. 
  On the other hand,  $\Ran P$ is a closed subspace of $\Yc\times\Yc$, 
  hence it easily follows that $P_1(\Ker P_2)\times\{0\}$ is a split closed linear subspace of $\Yc\times\Yc$. 
  Therefore, there exists a closed linear subspace $\Vc\subseteq\Yc\times\Yc$ 
  for which we have the direct sum decomposition 
  $$(P_1(\Ker P_2)\times\{0\})\dotplus\Vc=\Yc\times\Yc.$$
  Since $P_1(\Ker P_2)\times\{0\}\subseteq\Yc\times\{0\}$, we then obtain 
  $$(P_1(\Ker P_2)\times\{0\})\dotplus(\Vc\cap(\Yc\times\{0\}))=\Yc\times\{0\}.$$
  (Compare the proof of \cite[Lemma 2.6]{BGJP}.)
  Thus $(P_1(\Ker P_2)\times\{0\}$ is a split closed linear subspace of $\Yc\times\{0\}$. 
  This directly implies that $P_1(\Ker P_2)$  is a split closed linear subspace of $\Yc$. 
 \end{proof}
 
  Now we give a general situation of a regular foliation whose leaf space has an \H-structure. 
This example is 
in some sense a variant 
of Proposition~\ref{characterization}. 
  (Compare \cite[Th. I.2.1]{Pla}.)

\begin{theorem}\label{QuotientBanachManifold} 
Let $\Rc\subseteq M\times M$ be an equivalence relation 
on a $C^0$-paracompact Banach manifold $M$ 
with its corresponding quotient map $\pi\colon M\to M/\Rc$, 
the Cartesian projections $p_1,p_2\colon M\times M\to M$ 
and the restricted Cartesian projections $p_j^\Rc:=p_j\vert_\Rc\colon\Rc\to M$ for $j=1,2$. 
We assume that 
there exists on 
$\Rc$ 
a Banach manifold structure
 with the following properties:
\begin{enumerate}[{\rm(i)}]
\item 
The 
restricted Cartesian projections $p_1^\Rc,p_2^\Rc\colon \Rc\to M$ are surjective submersions with connected fibers.
\item
The inclusion map $\iota:\Rc\hookrightarrow M\times M$ is a split immersion. 
\end{enumerate}
Then the partition of $M$ into $\Rc$-equivalence classes is a nice foliation~$\Fc$ on $M$. 

Moreover, fixing any nice atlas  $\Vc=\{(V_\alpha,\varphi_\alpha)\}_{\alpha\in A}$ on $M$ with respect to the foliation $\Fc$, 
we select a family of local transversals $\{\Tc_\alpha\}_{\alpha\in A}$
and define the disjoint union $\Tc:=\bigsqcup\limits_{\alpha\in A}\Tc_\alpha$  
with its correponding split immersion $\tau\colon \Tc\to M$, 
where $\tau\vert_{\Tc_\alpha}\colon\Tc_\alpha\hookrightarrow M$ is the inclusion map for every $\alpha\in A$.    
We also define  
$\Rc_\Tc:=(\tau\times\tau)^{-1}(\Rc)\subseteq\Tc\times\Tc$  
and the Cartesian projection $q_1,q_2\colon\Tc\times\Tc\to \Tc$ 
with their restrictions $q_j^{\Rc_\Tc}:=q_j\vert_{\Rc_\Tc}\colon\Rc_\Tc\to\Tc$ for $j=1,2$.
Then we have the following properties:
\begin{enumerate}[{\rm(a)}]
\item
\label{QuotientBanachManifold_item_a}
The mapping $\pi\circ \tau\colon \Tc\to M/\Rc$ is surjective. 
\item
\label{QuotientBanachManifold_item_b}
 The set $\Rc_\Tc$ has the structure of a Banach manifold for which  the inclusion map $\Rc_\Tc\hookrightarrow\Tc\times\Tc$ is a split immersion,  
 the restricted Cartesian projections $q_1^{\Rc_\Tc},q_2^{\Rc_\Tc}\colon \Rc_\Tc\to\Tc$ are \'etale maps, 
 and the quotient map $\pi_\Tc\colon\Tc\to \Tc/ R_\Tc$ is an \H-atlas. 
\item
\label{QuotientBanachManifold_item_c}
The map $\tau$ 
gives rise to a pre-\H-diffeomorphism $\bar{\tau}\colon \Tc/\Rc_\Tc\to M/\Rc$  
and the triple $(\Tc,\pi\circ \tau, M/\Rc)$ is an \H-atlas on $M/\Rc$.
\item
\label{QuotientBanachManifold_item_d}
The quotient map $\pi\colon M\to M/\Rc$ is a submersion of \H-manifolds 
and for every $x\in M$ we have 
\begin{equation}
\label{QuotientBanachManifold_eq1}
\Ker(T_x\pi)=T_x(L(x))=T_x(\pi(x))
\end{equation} 
where $L(x)\hookrightarrow M$ is the leaf of the foliation $\Fc$ with $x\in L(x)$, 
while $\pi(x)\hookrightarrow M$ is the $\Rc$-equivalence class of~$x$.
\end{enumerate}
\end{theorem}

\begin{proof} 
Recall the (pair) Banach-Lie groupoid $M\times M\tto M$:
\begin{itemize}
 \item the source map is the Cartesian projection $p_1$ and the target is $p_2$;
\item the product of $(x,x')$ and $(y,y')$ is defined if $x'=y$ and its value is $(x,y')$; 
\item the inverse of $(x,y)$ is $(y,x)$;
\item the identity $\mathbf{1}_x$ is $(x,x)$.
\end{itemize}
According to the properties of $\Rc$, this groupoid structure induces  the structure of a Banach-Lie groupoid $\Rc\tto M$ in the sense of \cite[Def. 3.2]{BGJP}. 
In this way, we also identify $M$ with the diagonal  $\Delta M\subseteq\Rc\subseteq M\times M$,  
and the Lie algebroid associated to $\Rc\tto M$ is the bundle 
$$E:=(\Ker (Tp_2)\cap T\Rc)\vert_{\Delta M}=(\Ker (Tp_2^\Rc))\vert_{\Delta M}\to\Delta M\simeq M$$  
whose anchor  $\rho \colon E\to TM$  is the restriction of 
$Tp_1^\Rc\colon T\Rc\to TM$ to $E$. 
By the direct sum decomposition of vector bundles 
$$T(M\times M)=\Ker (Tp_1)\oplus \Ker (Tp_2)\to M\times M$$ 
we  obtain
$\Ker (Tp_1^\Rc)\cap \Ker (Tp_2^\Rc)=\{0\}$. 
As  $p_1^\Rc\colon\Rc\to M$  is a submersion,  
its tangent map $Tp_1^\Rc$ gives an isomorphism of vector bundles 
$T\Rc/\Ker (Tp_1^\Rc)\mathop{\longrightarrow}\limits^\sim TM$. 

On the other hand $p_2^\Rc\colon\Rc\to M$  is a submersion, 
hence the subbundle $\Ker (Tp_2^\Rc)$ is complemented in $T\Rc$. 
By restricting these bundles from $\Rc$ to  the diagonal $\Delta M\subseteq\Rc$, 
we obtain that  
the subbundle~$E$ is complemented in $T\Rc\vert_{\Delta M}$. 
Then, for arbitrary $x_0\in M$, 
	an application of Lemma~\ref{double} for $\Xc:=T_{(x_0,x_0)}\Rc$, $\Yc:=T_{x_0}M$, and  $P_j:=T_{(x_0,x_0)}p_j^\Rc$ for $j=1,2$, shows that 
$\rho(E_{(x_0,x_0)})=P_1(\Ker P_2)\subseteq T_{x_0}M$ is a split closed linear subspace. 
Thus $\rho(E)$ is a complemented subbundle of $TM$. 
which is also involutive, 
hence, by the Frobenius Theorem, we obtain a regular foliation $\Fc$ on~$M$.

It now follows by \cite[Th. 4.24]{BGJP}
that the leaves of $\Fc$ are the connected components of the equivalence classes for $\Rc$. 
Moreover, the hypothesis ensures that the fibers of 
	the restricted Cartesian projections $p_1^\Rc,p_2^\Rc\colon \Rc\to M$ are connected, 
hence the leaves of the foliation $\Fc$ are exactly the equivalence classes for $\Rc$. 

\eqref{QuotientBanachManifold_item_a}
It suffices to recall from Proposition \ref{PFequivTG} 
that the mapping $\tau\colon\Tc\to M$ induces a bijection 
$$\bar{\tau}\colon \Tc/R_\Tc\to M/\Fc=M/\Rc $$ 
where $M/\Fc=M/\Rc$ since the leaves of the foliation $\Fc$ are the equivalence classes for $\Rc$, as noted above. 

\eqref{QuotientBanachManifold_item_b} 
We first endow the set $\Rc_\Tc$ with the structure of a Banach manifold for which  the inclusion map $\Rc_\Tc\hookrightarrow\Tc\times\Tc$ is a split immersion. 
To this end we note that for any point   $(x,y)\in \Rc_\Tc$ 
we have
\begin{equation}
\label{QuotientBanachManifold_proof_eq0}
T_{\tau(x)}M
= T_{\tau(x)}(L(\tau(x)))\oplus (T_x\tau)(T_x\Tc)
\end{equation}
hence
\begin{align}
T_{(\tau(x),\tau(y))}(M\times M)
=&T_{\tau(x)}M\oplus T_{\tau(y)} M \nonumber \\
=& T_{\tau(x)}(L(\tau(x)))\oplus (T_x\tau)(T_x\Tc) \nonumber \\
& \oplus T_{\tau(y)}(L(\tau(y)))\oplus (T_y\tau)(T_y\Tc) \nonumber\\
=&T_{\tau(x)}(L(\tau(x)))\oplus T_{\tau(y)}(L(\tau(y))) \nonumber \\
& \oplus (T_{(x,y)}(\tau\times\tau))(T_{(x,y)}(\Tc\times\Tc)) \nonumber\\
\label{QuotientBanachManifold_proof_eq1}
=&T_{(\tau(x),\tau(y))}\Rc+ (T_{(x,y)}(\tau\times\tau))(T_{(x,y)}(\Tc\times\Tc)).  
\end{align}
The last of the above equalities follows by the fact that $p_1^\Rc\colon\Rc\to M$ is a submersion by hypothesis, hence the leaves $L(\tau(x))$ and $L(y)$ 
can be regarded as embedded submanifolds of $\Rc$ via 
$$L(\tau(x))\simeq \{\tau(x)\}\times L(\tau(x))=(p_1^\Rc)^{-1}(\tau(x))\subseteq \Rc$$
and similarly $L(\tau(x))\simeq(p_1^\Rc)^{-1}(y)\subseteq \Rc$, 
and thus 
$$T_{\tau(x)}(L(\tau(x)))\simeq T_{(\tau(x),\tau(y))}((p_1^\Rc)^{-1}(\tau(x)))\subseteq T_{(\tau(x),\tau(y))}\Rc$$
and similarly  $ T_{\tau(y)}(L(\tau(y)))\simeq T_{(\tau(x),\tau(y))}((p_2^\Rc)^{-1}(\tau(y)))\subseteq T_{(\tau(x),\tau(y))}\Rc$. 
It follows by \eqref{QuotientBanachManifold_proof_eq1} that the mapping $\tau\times\tau\colon\Tc\times\Tc\to M\times M$ is transversal to the split immersed submanifold $\iota \colon \Rc\hookrightarrow M\times M$ 
in the sense of \cite[5.11.6]{Bou3}. 
Then, by \cite[5.11.7]{Bou3}, we obtain that $\Rc_\Tc\subseteq \Tc\times\Tc$
is a split immersed submanifold. 
It remains to prove that the restricted Cartesian projections $q_1^{\Rc_\Tc},q_2^{\Rc_\Tc}\colon \Rc_\Tc\to\Tc$ are \'etale maps, 
since it will then directly follow by Proposition~\ref{characterization} that
the quotient map $\pi_\Tc\colon\Tc\to \Tc/ \Rc_\Tc$ is an \H-atlas.

In order to prove that $q_1^{\Rc_\Tc},q_2^{\Rc_\Tc}\colon \Rc_\Tc\to\Tc$ are \'etale maps, we first note that 
the diagonal embedding $\delta\colon \Tc\to\Rc_\Tc$, $t\mapsto (t,t)$, 
satisfies $q_j^{\Rc_\Tc}\circ\delta=\id_\Tc$ for $j=1,2$. 
Considering the tangent map of both sides of these equalities 
we directly obtain that the tangent maps $T_{(x,y)}(q_1^{\Rc_\Tc})\colon T_{(x,y)}\Rc_\Tc\to T_x\Tc$ 
and 
$T_{(x,y)}(q_2^{\Rc_\Tc})\colon T_{(x,y)}\Rc_\Tc\to T_y\Tc$
are surjective. 
It remains to check that these tangent maps are injective as well. 
To this end we use the commutative diagrams
\begin{equation*}
\xymatrix{\Rc_\Tc \ar[r]^{\qquad (\tau\times\tau)\vert_{\Rc_\Tc}\qquad } \ar[d]_{q_j^{\Rc_\Tc}} & \Rc \ar[d]^{p_j^\Rc}\\
\Tc \ar[r]^{\tau} & M }
\end{equation*}
for $j=1,2$. 
In order to show that $T_{(x,y)}(q_2^{\Rc_\Tc})\colon T_{(x,y)}\Rc_\Tc\to T_y\Tc$ is injective, we prove that if $v\in T_{(x,y)}\Rc_\Tc$ and $T_{(x,y)}(q_2^{\Rc_\Tc})v=0$, 
then $v=0$. 
In fact, we obtain by the above commutative diagram for $j=2$, 
\begin{align*}
(T_{(\tau(x),\tau(y))}(p_2^\Rc))(T_{(x,y)}(\tau\times\tau)v)
&=T_{(x,y)}(p_2^\Rc\circ (\tau\times\tau))v \\
&=T_{(x,y)}(\tau\circ q_2^{\Rc_\Tc})v \\
& =(T_y\tau)(T_{(x,y)}(q_2^{\Rc_\Tc})v) \\
&=0
\end{align*}
hence $T_{(x,y)}(\tau\times\tau)v\in\Ker T_{(\tau(x),\tau(y))}(p_2^\Rc)$. 
Since $p_2^\Rc\colon\Rc\to M$ is a submersion by hypothesis, 
we have 
\begin{align*}\Ker T_{(\tau(x),\tau(y))}(p_2^\Rc)
&=
T_{(\tau(x),\tau(y))}(p_2^\Rc)^{-1}(\tau(y)) \\
&=T_{(\tau(x),\tau(y))}(L(\tau(y))\times\{\tau(y)\}) \\
&= T_{\tau(x)} (L(\tau(y)))\times\{0\}
\subseteq T_{(\tau(x),\tau(y))}(M\times M)
\end{align*}
where we recall that $L(\tau(y))$ is the leaf through the point $\tau(y)\in M$, 
and we have $\tau(x)\in L(\tau(y))$ by the assumption $(x,y)\in\Rc_\Tc$. 
Since $\Tc$ is transversal to the leaves by construction, 
we have $(\Ran T_{\tau(x)}\tau)\cap T_{\tau(x)}L(\tau(x))=\{0\}$ 
and $(\Ran T_{\tau(y)}\tau)\cap T_{\tau(y)}L(\tau(y))=\{0\}$, 
where of course $L(\tau(x)=L(\tau(y))$ since $\tau(x)\in L(\tau(y))$. 
Consequently 
$$\Ran (T_{(x,y)}(\tau\times\tau))\cap 
(T_{\tau(x)}L(\tau(x))\times T_{\tau(y)}L(\tau(x)))=\{0\}.$$
Recalling from above that 
$$T_{(x,y)}(\tau\times\tau)v\in \Ker T_{(\tau(x),\tau(y))}(p_2^\Rc)
=T_{\tau(x)} (L(\tau(y)))\times\{0\}$$
we obtain $T_{(x,y)}(\tau\times\tau)v=0$. 
On the other hand $\tau\colon\Tc\to M$ is an immersion, 
hence $\tau\times\tau\colon\Tc\times\Tc\to M\times M$ is an immersion as well, 
and it then follows that $v=0$. 
This completes the proof of the fact that $T_{(x,y)}(q_2^{\Rc_\Tc})\colon T_{(x,y)}\Rc_\Tc\to T_y\Tc$ is injective. 
One can similarly prove that $T_{(x,y)}(q_1^{\Rc_\Tc})\colon T_{(x,y)}\Rc_\Tc\to T_x\Tc$ is injective, 
and this concludes the proof of Assertion~\eqref{QuotientBanachManifold_item_b}. 

\eqref{QuotientBanachManifold_item_c}
This is a consequence of Proposition~\ref{PFequivTG} 
and particularly Corollary~\ref{PFequivTG_cor}.

\eqref{QuotientBanachManifold_item_d} 
See Corollary~\ref{PFequivTG_cor}\eqref{PFequivTG_cor_item3}. 
Alternatively, for  arbitrary  $x\in M$,  if $x\in V_\alpha$, 
then we have a submersion $\pi_\alpha: U_\alpha\to \Tc_\alpha$\footnote{See the notation before Remark \ref{plaqueqq}.} 
so that $\pi_\Tc\circ \pi_\alpha=\pi $ on $U_\alpha$ which implies by definition that $\pi$ is a \H-submersion. 

Moreover, if $\Pc_x$ is a plaque through $x$, then $\Pc_x$ is an open subset of the leaf $L(x)$, while $L(x)$ is an open subset of the $\Rc$-equivalence class $\pi(x)$, hence $T_x(\pi(x))=T_x(L(x))=T_x(\Pc_x)=\Ker(T_x\pi)$, 
and this completes the proof. 
\end{proof}

\begin{remark}
\label{isomorphism}
\normalfont 
Let us resume the notation of the proof of Theorem~\ref{QuotientBanachManifold}
By \cite[5.11.7]{Bou3}, we obtain not only that $\Rc_\Tc\subseteq \Tc\times\Tc$
is a split immersed submanifold 
but also that its tangent space at $(x,y)\in\Rc_\Tc$ is 
\begin{equation}
\label{QuotientBanachManifold_proof_eq2}
T_{(x,y)}\mathcal{R}_\Tc
=(T_{(x,y)}(\tau\times\tau))^{-1}(T_{(\tau(x),\tau(y))}\Rc)
\end{equation}
and 
there exists a closed linear subspace 
$K_{(x,y)}\subseteq T_{(x,y)}(\Tc\times\Tc)$ 
satisfying 
\begin{equation}
\label{QuotientBanachManifold_proof_eq3}
T_{(x,y)}(\Tc\times\Tc)=T_{(x,y)}\Rc_\Tc\oplus K_{(x,y)}. 
\end{equation}
We claim that, 
if 
$H_{(\tau(x),\tau(y))}\subseteq T_{(\tau(x),\tau(y))}(M\times M)$  is a 
closed linear subspace satisfying 
\begin{equation}
\label{QuotientBanachManifold_proof_eq4}
T_{(\tau(x),\tau(y))}(M\times M)
=T_{(\tau(x),\tau(y))}\Rc\oplus H_{(\tau(x),\tau(y))}. 
\end{equation} 
then the Banach spaces $H_{(\tau(x),\tau(y))}$ and  $K_{(x,y)}$ are isomorphic. 

Indeed  we have as in  the proof of \cite[Th. I.2.1]{Pla}, 
\allowdisplaybreaks
\begin{align*}
T_{(\tau(x),\tau(y))}\Rc\oplus H_{(\tau(x),\tau(y))} 
\mathop{=}\limits^{\eqref{QuotientBanachManifold_proof_eq4}}
&T_{(\tau(x),\tau(y))}(M\times M)  \\
\mathop{=}\limits^{\eqref{QuotientBanachManifold_proof_eq1}}
&T_{(\tau(x),\tau(y))}\Rc+ (T_{(x,y)}(\tau\times\tau))(T_{(x,y)}(\Tc\times\Tc))\\
\mathop{=}\limits^{\eqref{QuotientBanachManifold_proof_eq3}}
&T_{(\tau(x),\tau(y))}\Rc \\
&+(T_{(x,y)}(\tau\times\tau))(T_{(\tau(x),\tau(y))}\Rc_\Tc\oplus K_{(x,y)}) \\
\mathop{=}\limits^{\eqref{QuotientBanachManifold_proof_eq2}}
&
T_{(\tau(x),\tau(y))}\Rc\oplus 
(T_{(x,y)}(\tau\times\tau))(K_{(x,y)})
\end{align*}
which leads to Banach space isomorphisms 
$$H_{(\tau(x),\tau(y))}\simeq T_{(x,y)}(M\times M)/T_{(\tau(x),\tau(y))}\Rc
\simeq (T_{(x,y)}(\tau\times\tau))(K_{(x,y)})
\simeq K_{(x,y)}.$$ 
Thus, $H_{(x,y)}$ and $K_{(x,y)}$ are isomorphic Banach spaces. 
\end{remark}

 Theorem \ref{QuotientBanachManifold} has an essential application in the Banach-Lie group framework 
as explained in the following.
 
\begin{definition}[{cf. \cite[Ch. III, \S 6, Def. 1]{Bou2}}]
	\label{weaksubgroup} 
	\normalfont 
	Let $G$ be a Banach-Lie group.  
A \emph{split immersed subgroup} of $G$ is a subgroup $K\subseteq G$ endowed with the structure of a Banach-Lie group for which the inclusion map $\iota\colon K\hookrightarrow G$ is a split immersion. 

There are two important classes of split immersed subgroups: 
\begin{enumerate}
	\item If  the Banach-Lie group $K$ is connected, then it is called an \emph{integral subgroup} of $G$. 
	\item If  the Banach-Lie group $K$ is an initial submanifold, then it is called an \emph{initial subgroup} of $G$. 
\end{enumerate}
\end{definition}

In the setting of Definition~\ref{weaksubgroup}, we denote by 
$\Lie(G)$ the Lie algebra of $G$, regarded as the tangent space $T_\1 G$ 
at the unit element $\1\in G$, endowed with the Lie bracket. 
We also denote by $\Lie(\iota):=T_\1\iota\colon \Lie(K)\to \Lie(G)$ the tangent mapping of $\iota$ at the unit element $\1\in H$, 
which is a topological isomorphism of the Lie algebra $\Lie(K)$ of $K$ onto a a split closed subalgebra of the Lie algebra $\Lie(G)$ of $G$.

We will need the following parameterization of the integral subgroups of a Banach-Lie group. 
 
 \begin{lemma}
 \label{Lie2}
 If $G$ is a Banach-Lie group, then the following assertions hold.  
 	\begin{enumerate}[{\rm(i)}]
 		\item\label{Lie2_item1}
 		The mapping  $H\mapsto \Lie(H)$ is a   bijection from the set of 
 		integral subgroups of $G$
 		onto the set of  split closed subalgebras  of the Banach-Lie algebra $\Lie(G)$. 
 		\item\label{Lie2_item2} 
 		If moreover $G$ is connected, then an integral subgroup $H\subseteq G$ is a normal subgroup if and only if $\Lie(K)$ is an ideal of $\Lie(G)$. 
 		\item\label{Lie2_item3} 
 		An integral subgroup $K$ is separable if and only if its Lie algebra $\Lie(K)$ is a separable Banach-Lie algebra and, if this is the case, then $K$ is an initial subgroup of $G$. 
 	\end{enumerate}
 \end{lemma}

\begin{proof}
\eqref{Lie2_item1} This  follows by \cite[Ch. III, \S 6, Th. 2(i)]{Bou2}. 

\eqref{Lie2_item2} 
See \cite[Satz 12.6]{Mai62}. 

\eqref{Lie2_item3} 
It is well known that a connected Banach-Lie group is second countable if and only if its Lie algebra is separable. 
(See for instance \cite[Lemma II.2.2]{Pla}.) 
Furthermore, if any integral subgroup $K$ is a leaf of the regular foliation of $G$ whose leaves are the left cosets $gK$ for arbitrary $g\in G$. 
Therefore, if $K$ is second countable, then it is an initial submanifold of $G$, cf. for instance \cite[9.2.8]{Bou3} or \cite[Prop. 2.2]{Pen70}.
\end{proof}

 With these definitions and results we have:
  
 \begin{proposition}
 	\label{G/K} 
 	Let $G$ be a   Banach-Lie group with a split immersed subgroup $K\subseteq G$. 
Then $G/K$ has the structure of a \H-manifold for which the quotient map 
$q\colon G\to G/K$ is a submersion of \H-manifolds 
with $\Ker(T_\1 q)=T_\1 K$. 

If $K$ is moreover an initial subgroup of $G$, then  
$G/K$ has the structure of a \Q-manifold for which the quotient map 
$q\colon G\to G/K$ is a submersion of \Q-manifolds. 
 \end{proposition}

 \begin{proof} (Compare the proof of \cite[Th. II.4.2]{Pla}.) 
We consider the diffeomorphism 
$$\Psi\colon G\times G\to G\times G, \quad (g,h)\mapsto (g,gk)$$ 
and 
the set 
 $$\Rc:=\{(g,g')\in G\times G\mid \pi(g)=\pi(g')\}.$$  
Since $K$ is a split immersed submanifold of $G$, 
it follows that $G\times K$ is a split immersed submanifold of
 $G\times G$, 
 and the first Cartesian projection $G\times K\to G$ is a submersion 
 with connected fibers, since $K$ is connected. 
 Since $\Psi\colon G\times G\to G\times G$ is a diffeomorphism that intertwines the Cartesian projections on the first factor, 
 it then follows that $\Psi(G\times K)$ is a split immersed submanifold of
 $G\times G$ and the first Cartesian projection $\Psi(G\times K)\to G$ is a submersion. 
 It is easily seen that 
 $\Rc=\Psi(G\times K)$, 
 hence $\Rc$ has the structure of a split immersed submanifold of
 $G\times G$ for which the first Cartesian projection $\Rc\to G$ is a submersion 
 with connected fibers.  
 
 We now note that $\Rc$ is invariant to the flip diffeomorphism 
 $$\Phi\colon G\times G\to G\times G, \colon (g,g')\mapsto (g',g)$$ 
 and $(\Psi^{-1}\circ \Phi \circ\Psi)(g,k)=\Psi^{-1}(gk,g)=(gk,k^{-1})$, 
 hence $\Psi^{-1}\circ \Phi \circ\Psi\colon G\times K\to G\times K$ 
 is smooth. 
 Since the mapping $\Phi\vert_{G\times K}\colon G\times K\to\Rc$ is a diffeomorphism by the definition of the Banach manifold structure of $\Rc$, it then follows that $\Phi\vert_\Rc\colon\Rc\to\Rc$ is smooth. 
 This smooth mapping $\Phi\vert_\Rc\colon\Rc\to\Rc$ intertwines the Cartesian projections from $\Rc$ onto $G$ and, since we already know that the first Cartesian projection is a submersion, it follows that 
  it then follows that the second Cartesian projection $\Rc\to G$ is a submersion  as well, with connected fibers. 

Consequently, it follows by Theorem \ref{QuotientBanachManifold}, 
that $G/K$ has the structure of an \H-manifold and   $q\colon G\to G/K$ is a submersion of \H-manifolds. 
The equality $\Ker(T_\1 q)=T_\1 K$ follows by \eqref{QuotientBanachManifold_eq1}.

Finally, if $K$ is an initial subgroup of $G$, then $\Rc=\Psi(G\times K)$ is an initial submanifold of $G\times G$, 
hence the assertion follows by \cite[Th.]{Pl80b} or \cite[Th. I.2.1]{Pla}. 
   \end{proof}

\section{Group \H-manifolds}
\label{Sect4}

This section contains our main results on the Lie theory for the \H-groups, 
corresponding to Sophus Lie's fundamental theorems I and III. 
Specifically, we construct the Lie functor from the category of \H-groups to the category of real Banach-Lie algebras (Theorem~\ref{Lie_I}) 
and we prove that every real Banach-Lie algebra arises from an \H-group (Theorem~\ref{Lie_III}). 

\subsection{Group \H-manifolds and their Lie algebras}
By way of motivation for the notion we are about to introduce, 
recall that a (Banach-)Lie group $G$ is a  (Banach) manifold 
which has  also a group  structure whose structural mapping 
$G\times G\to G$, $(x,y)\mapsto xy^{-1}$, is smooth. 
(See for instance \cite[Ch. III, \S 1, Def. 1]{Bou2}.)  
In the same way, we make the following definition. 

\begin{definition}
\normalfont 
A \emph{group \H-manifold} (for short an \emph{\H-group})
is  a  \H-manifold $G$  with a group  structure whose structural mapping 
$G\times G\to G$, $(x,y)\mapsto xy^{-1}$, is smooth in the sense of \H-manifolds, where the direct product $G\times G$ is a \H-manifold by Corollary~\ref{characterization_cor1} and Example~\ref{pre-H_prod}. 
\end{definition} 

Special case of this notion are the \Q-groups introduced in  \cite{Ba73} in finite dimensions and \cite{Pla} in the Banach context. 
(See also \cite{BPZ19}.)
Every Banach-Lie group is of course an \H-group.  
More generally we have:

\begin{proposition}\label{TG} 
Let $G$ be an \H-group with its group multiplication, inversion, and unit element denoted by $\bm\colon G\times G\to G$, $\iota\colon  G\to G$, 
and $\1\in G$, respectively. 
Then the tangent \H-manifold $TG$ is an \H-group 
whose group multiplication, inversion, and unit element are $T\bm\colon TG\times TG\to TG$,  
$T\iota\colon TG\to TG$, and $0\in T_\1 G$, respectively. 
\end{proposition}

 \begin{proof} 
 We first recall from Proposition~\ref{preHtg} that the tangent space of any \H-manifold is again an \H-manifold, 
 hence the correspondence $M\mapsto TM$ is a functor on the category of \H-manifolds. 
 At first note that by construction of the tangent bundle of a \H-manifold $S$, 
for each $(s,s')\in S\times S$ we have $T_{(s,s')}(S\times S)=T_sS\times T_{s'}S$. 
(See Definition~\ref{tgH_def}.)
This implies that the above functor commutes with the Cartesian products of finitely many objects. 
Now the assertion follows along the lines of the proof in the classical case of Banach-Lie groups. 
See for instance the proof of \cite[Prop. 2.4]{Be06}.
\end{proof}

Just as in the classical case of the Banach-Lie groups, we consider the left translation 
 ${L}_g\colon G\to G$, ${L}_g(h)=gh$, for all $g,h\in G$. 
Note that $L_g$ is a diffeomorphism whose inverse is $L_{g^{-1}}$. 
For any $u\in \gg:=T_\1 G$ we define the vector field $\Lc_u\colon G\to TG$,  $\mathcal{L}_u(g)=(T_\1({L}_g))(u)\in T_g G$. 
Conversely, a  (global) vector field $X\colon G\to TG$ is called 
\emph{left invariant} if  $T_h(L_g)(X(h))=X(gh)$ for all $g,h\in G$. 
We denote by $\Xc^\Lie(G)$ the set of all left-invariant vector fields
 on~$G$. 

In order to define the Lie algebra of $G$ we need to prove that the Lie bracket of left invariant vector fields is again a left invariant vector field, which is essentially the content of the following proposition.

\begin{proposition}
\label{left_bracket}
If $G$ is an \H-group, 
then the set of all left-invariant vector fields~$\Xc^\Lie(G)$ is a subalgebra of the Lie algebra $\Xc(G)$. 
Moreover, the mapping 
$\Lc\colon T_\1 G\to \Xc^\Lie(G)$, $u\mapsto\Lc_u$, is a linear isomorphism. 
\end{proposition}

\begin{proof}
Let $\pi\colon M\to G$ be an \H-atlas 
and denote by $\Gamma_\Lie$ the pseudogroup of local diffeomorphisms of $M$ generated by the smooth (local) lifts of the left-translation diffeomorphisms $L_g\colon G\to G$ for all $g\in G$. 
Recalling the Lie algebra isomorphism 
$$\pi_*\colon \Xc^{\Gamma(\pi)}(M)\to\Xc(G)$$ 
from Remark~\ref{fields_bracket}, it is easily seen that 
for any $X\in \Xc^{\Gamma(\pi)}(M)$, its corresponding 
vector field $\pi_*(X)\in \Xc(G)$ is left-invariant if and only if $X$ is invariant with respect to the pseudogroup $\Gamma_\Lie$. 
That is, we have the linear isomorphism 
$$\pi_*\vert_{\Xc^{\Gamma(\pi)}(M)\cap \Xc^{\Gamma_\Lie}(M)}\colon \Xc^{\Gamma(\pi)}(M)\cap \Xc^{\Gamma_\Lie}(M)\to\Xc^\Lie(G).$$
Here $\Xc^{\Gamma(\pi)}(M)\cap \Xc^{\Gamma_\Lie}(M)$ is a subalgebra of the Lie algebra $\Xc(M)$ (see the end of Definition~\ref{fields_invar}) 
hence, using the above Lie algebra isomorphism $\pi_*$, 
it follows that $\Xc^\Lie(G)$ is a subalgebra of the Lie algebra $\Xc(G)$. 

The proof of the fact that $\Lc\colon T_\1 G\to \Xc^\Lie(G)$ is a linear isomorphism follows by the same lines as in the case of Banach-Lie groups. 
(See for instance \cite[Lemma 2.17]{Be06}.) 
The smoothness of $\Lc_u\colon G\to G$ follows by the formula 
$(\Lc_u)(g)=(T_\1(L_g))(u)=(T_{(g,\1)}\bm)(0,u)$ for all $g\in G$ and $u\in T_\1 G$, where $0\in T_g G$, 
using the fact that the multiplication mapping $\bm\colon G\times G\to G$ is smooth, hence its tangent mapping $T\bm\colon TG\times TG\to TG$ is again smooth. 
(See also Proposition~\ref{TG}.)
\end{proof}

\begin{definition}
\label{Lie_alg_def}
\normalfont
If $G$ is an \H-group, then its \emph{Lie algebra} is the Banach space $\Lie(G):=T_\1 G$ endowed with the unique Lie bracket for which the linear isomorphism $\Lc\colon T_\1 G\to \Xc^\Lie(G)$ from Proposition~\ref{left_bracket} is a Lie algebra isomorphism. 

If $K$ is another \H-group and $\varphi\colon K\to G$ is a smooth group homomorphism, then we denote $\Lie(\varphi):=T_\1\varphi\colon \Lie(K)\to\Lie(G)$. 

We denote by $\BHLG$ th category whose objects are the \H-groups and whose morphisms are the smooth group homomorphisms, and by $\BLA$ the category whose objects are the real Banach-Lie algebras, and whose morphisms are continuous Lie algebra homomorphisms. 
\end{definition}

\begin{theorem}
\label{Lie_I}
The correspondence $\Lie\colon\BHLG\to\BLA$ is a well-defined functor. 
\end{theorem}

\begin{proof}
We must prove that 
if $\varphi\colon K\to G$ is a smooth homomorphism of \H-groups, then $\Lie(\varphi)\colon \Lie(K)\to\Lie(G)$ is a Lie algebra morphism. 
To this end we consider the linear isomorphisms 
$\Lc^G\colon T_\1 G\to \Xc^\Lie(G)$, $u\mapsto\Lc^G_u$, 
and 
$\Lc^K\colon T_\1 K\to \Xc^\Lie(K)$, $u\mapsto\Lc^K_u$, 
given by Proposition~\ref{left_bracket} for the \H-groups $G$ and $K$, respectively. 
For all $g,h\in G$ and $k,r\in K$ we also define $L^G_g(h)=gh$ and $L^K_k(r)=kr$, hence $\varphi\circ L^K_k=L^G_{\varphi(k)}\circ \varphi$ 
since $\varphi\colon K\to G$ is a group homomorphism. 

For every $u\in T_\1K$ we then obtain
\begin{align*}
(T_k\varphi)((\Lc^K_u)(k))
&=(T_k\varphi)((T_\1(L^K_k))(u)) \\
&=T_\1(\varphi\circ L^K_k)(u) \\
&=T_\1(L^G_{\varphi(k)}\circ \varphi)(u) \\
&=T_\1(L^G_{\varphi(k)})((T_\1\varphi)u) \\
&=\Lc^G_{(T_\1\varphi)u}(\varphi(k))
\end{align*}
for arbitrary $k\in K$. 
If  $\pi\colon M\to G$ and $\tau\colon N\to K$ are \H-atlases 
and $\widehat{\varphi}\colon U\to V$ is a local lift of $\varphi$, 
then the diagram 
$$\xymatrix{
	TW \ar@{^{(}->}[d] \ar[rrr]^{T\widehat{\varphi}} & & & TV \ar@{^{(}->}[d] \\
	TN \ar[r]^{T\tau} & TK \ar[r]^{T\varphi}& TG & TM \ar[l]_{T\pi} \\
N \ar[u]^{\tau_*^{-1}(\Lc^K_u)} \ar[r]^{\tau} & K \ar[u]^{\Lc^K_u}\ar[r]^{\varphi} & G  \ar[u]_{\Lc^G_{(T_\1\varphi)u}} & M \ar[u]_{\pi_*^{-1}(\Lc^G_{(T_\1\varphi)u})}\ar[l]_{\pi} \\
W \ar@{^{(}->}[u] \ar[rrr]^{\widehat{\varphi}}  & & & V \ar@{^{(}->}[u]
}$$
is commutative, where $\tau_*^{-1}(\Lc^K_u)\in\Xc^{\Gamma(\tau)}(N)$ and 
$\pi_*^{-1}(\Lc^G_{(T_\1\varphi)u})\in\Xc^{\Gamma(\pi)}(M)$, 
cf. Remark~\ref{fields_bracket}.

We have $T\pi\circ T\widehat{\varphi}=T\varphi\circ T\tau\vert_{TW}$ 
hence 
\allowdisplaybreaks
\begin{align*}
T\pi\circ T\widehat{\varphi}\circ \tau_*^{-1}(\Lc^K_u)
& =T\varphi\circ T\tau\circ \tau_*^{-1}(\Lc^K_u)\vert_{W} \\
& =T\varphi\circ\Lc^K_u\circ\tau\vert_{W} \\
&=\Lc^G_{(T_\1\varphi)u} \circ\varphi \circ\tau\vert_{W} \\
&=\Lc^G_{(T_\1\varphi)u} \circ\pi \circ\widehat{\varphi} \\
&=T\pi\circ \pi_*^{-1}(\Lc^G_{(T_\1\varphi)u})\circ\widehat{\varphi}.
\end{align*}
Since the tangent map of the \H-atlas $\pi\colon M\to G$ is  fiberwise injective on $TM$  by Proposition~\ref{preHtg}, 
we then obtain for all $u\in T_\1 K$, 
\begin{equation}
\label{related}
T\widehat{\varphi}\circ \tau_*^{-1}(\Lc^K_u)
=\pi_*^{-1}(\Lc^G_{(T_\1\varphi)u})\circ\widehat{\varphi} 
\end{equation}
that is, the vector fields $\tau_*^{-1}(\Lc^K_u)\vert_W\colon W\to TW$ 
and $\pi_*^{-1}(\Lc^G_{(T_\1\varphi)u})\vert_V\colon V\to TV$ are $\widehat{\varphi}$-related. 
It is well-known that the property of being $\widehat{\varphi}$-related is compatible with the Lie bracket of vector fields, 
hence we obtain for all $u_1,u_2\in T_\1 K$, 
\begin{equation*}
T\widehat{\varphi}\circ [\tau_*^{-1}(\Lc^K_{u_1}),\tau_*^{-1}(\Lc^K_{u_2})]
=[\pi_*^{-1}(\Lc^G_{(T_\1\varphi)u_1}),\pi_*^{-1}(\Lc^G_{(T_\1\varphi)u_2})]\circ\widehat{\varphi}
\end{equation*}
Using the fact that the mappings 
$\pi_*\colon\Xc^{\Gamma(\pi)}(M)\to \Xc(G)$, 
and 
$\Lc^G\colon \Lie(G)\to \Xc^\Lie(G)$  
are Lie algebra isomorphisms 
(by Remark~\ref{fields_bracket} and Definition~\ref{Lie_alg_def}), 
we further obtain 
\begin{equation}
\label{related1}
T\widehat{\varphi}\circ [\tau_*^{-1}(\Lc^K_{u_1}),\tau_*^{-1}(\Lc^K_{u_2})]
=\pi_*^{-1}(\Lc^G_{[(T_\1\varphi)u_1,(T_\1\varphi)u_2]})\circ\widehat{\varphi}.
\end{equation}
On the other hand, using \eqref{related} for $u:=[u_1,u_2]\in\Lie(K)$, 
we have 
\begin{equation}
\label{related2}
T\widehat{\varphi}\circ [\tau_*^{-1}(\Lc^K_{u_1}),\tau_*^{-1}(\Lc^K_{u_2})]
=
T\widehat{\varphi}\circ \tau_*^{-1}(\Lc^K_{[u_1,u_2]})
=\pi_*^{-1}(\Lc^G_{(T_\1\varphi)[u_1,u_2]})\circ\widehat{\varphi} 
\end{equation}
since 
$\tau_*\colon\Xc^{\Gamma(\tau)}(N)\to \Xc(K)$
and
$\Lc^K\colon \Lie(K)\to \Xc^\Lie(K)$ 
are Lie algebra isomorphisms 
(by Remark~\ref{fields_bracket} and Definition~\ref{Lie_alg_def} again), 
while the second of the above equalities 
is an application of \eqref{related} for $u:=[u_1,u_2]\in\Lie(K)$. 

It now follows by \eqref{related1}--\eqref{related2} that 
\begin{equation*}
\pi_*^{-1}(\Lc^G_{[(T_\1\varphi)u_1,(T_\1\varphi)u_2]})\circ\widehat{\varphi}
=
\pi_*^{-1}(\Lc^G_{(T_\1\varphi)[u_1,u_2]})\circ\widehat{\varphi}. 
\end{equation*}
Since $\widehat{\varphi}$ is an arbitrary lift of $\varphi$ we then obtain, using again that 
$\pi_*\colon\Xc^{\Gamma(\pi)}(M)\to \Xc(G)$, 
and 
$\Lc^G\colon \Lie(G)\to \Xc^\Lie(G)$  
are Lie algebra isomorphisms, 
$$[(T_\1\varphi)u_1,(T_\1\varphi)u_2]=(T_\1\varphi)[u_1,u_2]$$
for arbitrary $u_1,u_2\in\Lie(K)$. 
Consequently $T_\1\varphi\colon\Lie(K)\to\Lie(G)$ is a Lie algebra homomorphism, and this completes the proof. 
\end{proof}

\begin{corollary}\label{quotient_group} 
	If $\iota\colon K\hookrightarrow G$ is the inclusion map of a  normal split  immersed   Lie subgroup of a  Banach-Lie group $G$, then $G/K$ has the structure of a \H-group for which the 
	quotient homomorphism $q\colon G\to G/K$ is a submersion and 
	we have the exact sequence of Banach-Lie algebras 
	\begin{equation}
	\label{quotient_group_eq1} 
	\xymatrix{0 \ar[r] & \Lie(K) \ar[r]^{\Lie(\iota)} & \Lie(G)\ar[r]^{\Lie(q)} & \Lie(G/K) \ar[r] & 0.
	}\end{equation}
If moreover $K$ is an initial subgroup, then $G/K$ is a \Q-group. 
\end{corollary}

\begin{proof} 
	Since $K$ is a normal subgroup, $G/K$ has a group structure and 
	$q\colon G\to G/K$ is a surjective group morphism.  
	From Proposition~\ref{G/K} it follows that $G/K$ is a \H-manifold and the mapping~$q$ is a submersion, and $G/K$ is a \Q-group if moreover $K$ is an initial subgroup of $G$.

	The structural map $\mu\colon G\times G\to G$, $(g, h)\mapsto gh^{-1}$ is smooth by hypothesis, 
	while the structural map $\widetilde{\mu}\colon (G/H)\times (G/H)\to G/H$, $(s,t)\mapsto st^{-1}$,   satisfies 
	$$q\circ\mu=\widetilde{\mu}\circ(q\times q).$$
	Here $q\times q\colon G\times G\to (G/K)\times(G/K)$ is a submersion, 
	hence it follows by Proposition~\ref{H-submersion} that 
	$\widetilde{\mu}\colon (G/K)\times (G/K)\to G/K$ is a smooth mapping, 
	that is, $G/K$ is an \H-group.
	
It remains to prove~\eqref{quotient_group_eq1}. 
The mappings $\Lie(\iota)$ and $\Lie(q)$ are Banach-Lie algebra homomorphisms by Theorem~\ref{Lie_I}. 
Moreover, $\Lie(\iota)=T_\1 \iota\colon T_\1K\to T_\1G$ is injective since $\iota\colon K\to G$ is an immersion, 
while $\Lie(q)=T_\1 q\colon T_\1G\to T_\1(G/K)$ is surjective since $q\colon G\to G/K$ is an submersion.  
Finally, we have $\Ker\Lie(q)=\Ran\Lie(\iota)$ by the equality $\Ker(T_\1 q)=T_\1 K$ in Proposition~\ref{G/K}. 
\end{proof}

We are now in a position to prove the following infinite-dimensional version of Sophus Lie's Third Fundamental Theorem.

\begin{theorem}
\label{Lie_III}
Every real Banach-Lie algebra is isomorphic to the Lie algebra of some  \H-group.  
\end{theorem}

\begin{proof}
We denote by $I=[0,1]$ the (compact) unit interval in the real line. 
For an arbitrary  real Banach-Lie algebra $\gg$ we define the path algebra 
$$\Lambda\gg:=\{\gamma\in\Cc(I,\gg)\mid \gamma(0)=0\}$$
and we endow $\Lambda\gg$ with its natural structure of a Banach-Lie algebra with pointwise defined operations and topology of uniform conergence on~$I$. 
We also define 
the evaluation mapping 
$$\ev_1\colon\Lambda\gg\to\gg,\quad \ev_1(\gamma):=\gamma(1)$$
and the loop algebra
$$\Omega\gg:=\Ker(\ev_1)=\{\gamma\in\Lambda\gg\mid \gamma(1)=0\}.$$
It is clear that $\ev_1$ is a Banach-Lie algebra homomorphism and $\Omega\gg$ is a closed ideal of $\Lambda\gg$. 
Moreover, if we define $\eta(v)\in\Lambda\gg$ by $(\eta(v))(t):=tv$ for all $t\in I$ and $v\in\gg$, then $\eta\colon\gg\to\Lambda\gg$ is a bounded linear map 
satisfying $\ev_1\circ\eta=\id_\gg$.  
Therefore we have the exact sequence of Banach-Lie algebras
\begin{equation}
\label{H-integr_proof_eq1} 
\xymatrix{
	0 \ar[r] & \Omega\gg \ar@{^{(}->}[r] & \Lambda\gg \ar[r]^{\ev_1} & \gg \ar[r] & 0
}
\end{equation}
for which $\eta\colon\gg\to\Lambda\gg$ is a linear cross-section. 
In particular, we obtain the decomposition into a direct sum of closed linear subspaces $\Lambda\gg=\Omega\gg\dotplus\Ran\eta$, 
and thus $\Omega\gg$ is a split closed ideal of~$\Lambda\gg$. 

On the other hand, there exists a connected and simply connected Banach-Lie group $\widetilde{G}$ whose Lie algebra is isomorphic to the path algebra~$\Lambda\gg$, by the main result of \cite{Swi}. 
We may assume $\Lie(\widetilde{G})=\Lambda\gg$ without loss of generality. 
Let $K\subseteq \widetilde{G}$ be the integral subgroup with $\Lie(K)=\Omega\gg$. 
Since $\Omega\gg$ is an ideal of $\widetilde{G}$, it follows that $K$ is a normal subgroup of $\widetilde{G}$ by Lemma~\ref{Lie2}\eqref{Lie2_item2}. 
Now Corollary~\ref{quotient_group} implies that $\widetilde{G}/K$  has the structure of an \H-group for which the 
quotient homomorphism $q\colon \widetilde{G}\to \widetilde{G}/K$ is a submersion and 
we have the exact sequence of Banach-Lie algebras 
\begin{equation*}
\xymatrix{0 \ar[r] & \Lie(K) \ar[r]^{\Lie(\iota)} & \Lie(\widetilde{G})\ar[r]^{\Lie(q)} & \Lie(\widetilde{G}/K) \ar[r] & 0.
}\end{equation*} 
Here $\Lie(\widetilde{G})=\Lambda\gg$ and $\Lie(K)=\Omega\gg$ hence, 
taking into account the exact sequence \eqref{H-integr_proof_eq1}, 
we obtain an isomorphism of Banach-Lie algebras $\Lie(\widetilde{G}/K)\simeq\gg$, and this completes the proof. 
\end{proof}

\begin{remark}
\label{critique}
\normalfont 
In connection with 
Theorem~\ref{Lie_III}, we note the following facts. 
\begin{enumerate}
	\item The construction used in the proof of Theorem~\ref{Lie_III} goes back to \cite{Pla} and \cite[Th. II.4.2]{Pl80b} in the special case of the separable Banach-Lie algebras.
	Specifically, it was proved that if  the real Banach-Lie algebra $\gg$ is separable, then the quotient group $\widetilde{G}/K$ is a \Q-group and a Banach space isomorphism $T_\1(\widetilde{G}/K)\simeq \gg$ was also constructed at the end of the proof of \cite[Th. II.4.2]{Pl80b}, 
	without to check that the Lie brackets of $\Lie(\widetilde{G}/K)$ and $\gg$ are compatible with that isomorphism. 
	In the proof of our Theorem~\ref{Lie_III}, it is exactly this last aspect which needs Theorem~\ref{Lie_I}. 
	\item Let us assume that a Banach-Lie algebra $\gg$ is \emph{enlargible}, 
	that is, we have $\gg=\Lie(G)$, where $G$ is a connected and simply connected Banach-Lie group. 
	(See for instance \cite[Ch. 3]{Be06}.)
	Then, we claim that, with the notation from the proof of Theorem~\ref{Lie_III}, 
	the \H-group $\widetilde{G}/K$ is actually a Banach-Lie group and 
	there exists an isomorphism of Banach-Lie groups $\widetilde{G}/K\simeq G$. 
	
	In fact, let us define 
	$\Lambda G:=\{\gamma\in\Cc(I,G)\mid \gamma(0)=\1\}$
	and we endow $\Lambda G$ with its natural structure of a topological group with pointwise defined operations and topology of uniform conergence on~$I$. 
	Then $\Lambda G$ is actually a contractible (hence connected and simply connected) Banach-Lie group and the evaluation map 
	$\ev_1\colon\Lambda G\to G$, $\gamma\mapsto\gamma(1)$ 
	is a surjective, submersive homomorphism of Banach-Lie groups 
	whose kernel $\Omega G:=\{\gamma\in\Lambda G\mid\gamma(1)=\1\}$ is a (split, topologically embedded) normal Banach-Lie subgroup of $\Lambda G$, which gives an isomorphism of Banach-Lie groups $\Lambda G/\Omega G\simeq G$, cf. \cite[Def. III.1]{GN03}. 
	Since $G$ is simply connected, it then follows that the group $\Omega G$ is connected, hence $\Omega G=K$, using the notation from the proof of Theorem~\ref{Lie_III}. 
	Consequently, we have the isomorphism of Banach-Lie groups $\widetilde{G}/K\simeq G$, as claimed. 
	\item There exists still another functorial construction that presents a given Banach-Lie algebra as the quotient of an enlargible Banach-Lie algebra, namely the construction of free Banach-Lie algebras. 
	Specifically, see \cite[Th. 3.3]{Pes93a}, \cite[Th. 6]{Pes93b}, and \cite[Th. B]{Pes95}. 
	However, in that construction, it is less clear that the kernel of the corresponding quotient map has a direct complement in the free Banach-Lie algebra under consideration. 
	Such a direct complement is essential for the techniques used in the present paper, related to the transverse structure of certain regular foliations. 
	It is for this reason that, in the above proof of Theorem~\ref{Lie_III}, we used instead the construction from \cite{Swi}, in which one can easily construct the linear cross-section $\eta$ of the exact sequence~\eqref{H-integr_proof_eq1}. 
\end{enumerate}
\end{remark} 

We conclude this section by briefly recalling a classical example of non-enlargible Banach-Lie algebra which is not separable. 
Hence the integrability results of \cite{Pla} and \cite{Pl80b} are not applicable to this example, but Theorem~\ref{Lie_III} is.

\begin{example}
\normalfont
Let $\Hc$ be a separable infinite-dimensional complex Hilbert space, 
and denote by $\Bc(\Hc)$ the associative Banach$*$- algebra of all bounded linear operators on~$\Hc$. 
It is well known that $\Bc(\Hc)$, endowed with its operator norm topology, 
is not separable. 
(See for instance \cite[Solution 99]{Hal82}.) 
Then the set of all skew-adjoint operators $\ug(\Hc):=\{a\in\Bc(\Hc)\mid a^*=-a\}$ is a real Banach-Lie algebra with the Lie bracket given by the operator commutator $[a,b]=ab-ba$ for all $a,b\in\ug(\Hc)$, 
and we have the direct sum decomposition $\Bc(\Hc)=\ug(\Hc)+\ie\ug(\Hc)$, 
which directly implies that $\ug(\Hc)$  is not separable. 

For an arbitrary real number $\theta\in\RR$ 
and we now define 
$$\ag_\theta:=\{(\ie t\cdot \id_\Hc,\ie t\theta\cdot \id_\Hc)\mid t\in\RR\}\subseteq \ug(\Hc)\times \ug(\Hc).$$
It is clear that $\ag_\theta$ is a (central) closed ideal of 
the Banach-Lie algebra $\ug(\Hc)\times \ug(\Hc)$, hence the quotient 
$$\gg_\theta:=(\ug(\Hc)\times \ug(\Hc))/\ag_\theta$$
has the natural structure of a real Banach-Lie algebra. 
Since $\ug(\Hc)\times \ug(\Hc)$ is not separable and $\dim\ag_\theta=1$, 
it is not difficult to check that $\gg_\theta$ is not separable. 
Moreover, $\gg_\theta$ is not enlargible if $\theta\in\RR\setminus\QQ$, 
as shown in \cite[\S 4]{DL66}. 
See also \cite[Ex. 3.35]{Be06} for additional details. 

It follows by Theorem~\ref{Lie_III} that for every $\theta\in\RR\setminus\QQ$ there exists an \H-group $G_\theta$ whose Lie algebra is isomorphic 
to~$\gg_\theta$, but we do not know if $\gg_\theta$ can be realized as the Lie algebra of a \Q-group. 
\end{example}

\section{Examples of \Q-groups which are not Banach-Lie groups}
\label{Sect5}

In order to show the wide variety of examples of \H-groups that exist beyond the realm of Banach-Lie groups, we develop in this section some methods of constructing quotients of Banach-Lie groups which are \Q-groups, hence in particular \H-groups. 
Our main result in this connection says in particular that, assuming the continuum hypothesis, if the center of the Lie algebra of a connected separable Banach-Lie group has the dimension $>1$, then that group has a quotient which is a \Q-group but not a Banach-Lie group (Theorem~\ref{pseudo_discr_cor2}). 
 
A central role is held by the following notion of pseudo-discrete subgroup of a Lie group, which is slightly more general than the one defined in \cite[page 254]{Ba73}. 
In the cas of normal subgroups, the following definition is equivalent to \cite[Def. 3.4]{BPZ19}.

\begin{definition}\label{pseudo-discr_def}
	\normalfont
	A \emph{pseudo-discrete subgroup} of a Banach-Lie group $G$ is any subgroup $H\subseteq G$ with the property that if $\gamma\in\Ci(D_\gamma,G)$ and $\gamma(D_\gamma)\subseteq H$, 
	where $D_\gamma\subseteq \RR$ is an open interval, then $\gamma$ is constant. 
\end{definition}

The importance of pseudo-discrete subgroups comes from the following fact, 
which shows in particular examples of the quotients Banach-Lie groups which are \Q-groups but not Banach-Lie groups are obtained as soon as we have 
examples of normal subgroups which are simultaneously connected and pseudo-discrete. 

\begin{proposition}
If $G$ is a Banach-Lie group with a normal subgroup $N\subseteq G$, 
then the quotient map $\pi\colon G\to S:=G/N$ is a \Q-atlas if and only if the subgroup $N$ is pseudo-discrete. 
If this is the case, then $G/N$ is a Banach-Lie group if and only if $N$ is discrete subgroup of $G$. 
\end{proposition}

\begin{proof}
See \cite[Prop. 3.5]{BPZ19}.
\end{proof}

\begin{proposition}\label{pseudo-discr_prop}
	Let $G$ be a connected Banach-Lie group with its center $Z\subseteq G$. 
	Then for every pseudo-discrete normal subgroup $H\subseteq G$ one has $H\subseteq Z$.  
\end{proposition}

\begin{proof}
	Since the Banach-Lie group $G$ is connected, it follows that it is generated by the image of its exponential map $\exp_G\colon \gg\to G$. 
	Therefore it suffices to prove that for every $h\in H$ and $X\in\gg$ one has $(\exp X)h(\exp X)^{-1}=h$. 
	To this end we define 
	$$\gamma_{X,h}\colon \RR\to G, \quad \gamma_{X,h}(t):=(\exp tX)h(\exp t X)^{-1}h^{-1}.$$ 
	Then $\gamma_{X,h}\in\Ci(\RR,G)$. 
	
	On the other hand, since $H$ is a normal subgroup of $G$ and $h\in H$, one has $\gamma_{X,h}(\RR)\subseteq H$. 
	Therefore, by the hypothesis that $H$ is pseudo-discrete, one obtains that $\gamma_{X,h}$ is constant. 
	In particular, $\gamma_{X,h}(1)=\gamma_{X,h}(0)$, that is, $(\exp X)h(\exp X)^{-1}=h$, which completes the proof. 
\end{proof}

\begin{remark}
	\normalfont
	Proposition~\ref{pseudo-discr_prop} reduces the study of pseudo-discrete normal subgroups of connected Banach-Lie groups to the study of 
	pseudo-discrete subgroups of \emph{abelian} connected Banach-Lie groups. 
	This idea is used in the proof of Theorem~\ref{pseudo_discr_cor2}. 
\end{remark}

\begin{remark}
	\normalfont
	If $G$ is a finite-dimensional Lie group, then a subgroup $H\subseteq G$ is pseudo-discrete if and only if the arc-connected component of $\1\in H$ is $\{\1\}$. 
	This follows by a theorem of Yamabe as discussed for instance \cite[Rem. II.6.5(b)]{Ne06}. 
\end{remark}

\begin{remark}
	\normalfont 
	Let $\Xc$ be a Banach space, regarded as an additive Banach-Lie group $(\Xc,+)$. 
	Here are some examples of pseudo-discrete subgroups of $\Xc$ under various special assumptions. 
	\begin{itemize}
		\item If $\Xc=L^p_{\RR}(I)$ with $1\le p<\infty$ for the interval $I=[0,1]$, then K.H.~Hofmann's example $H:=L^p_{\RR}(I)$ is a  closed and path connected subgroup of $\Xc$ which is also pseudo-discrete. See \cite[Ex. 3.6]{BPZ19}.
		\item Every 0-dimensional subgroup of $\Xc$ is clearly pseudo-discrete. 
		(A topological space is 0-dimensional if its topology has a basis consisting of subsets that are simultaneously closed and open.)
		There exist 0-dimensional, closed, nondiscrete subgroups of $\Xc$ if and only if $\Xc$ contains a subspace isomorphic to the Banach space $c_0$ of sequences that are convergent to zero, by \cite[Th. 4.1]{ADG94}. 
	\end{itemize}
\end{remark}

We now prepare to prove that there exist \emph{connected} pseudo-discrete normal nontrivial subgroups of any separable Banach-Lie group for which the dimension of the center of its Lie algebra is~$>1$. 
(See Theorem~\ref{pseudo_discr_cor2}.)

\begin{definition}
	\normalfont 
	A topological space $T$ is said to satisfy the \emph{countable separation condition} if for every subset $A\subseteq T$  with $\card\, A\le\aleph_0$, its complement $T\setminus A$ is connected. 
\end{definition}

\begin{lemma}\label{pseudo_discr_L4}
	If $G$ is a connected separable abelian Banach-Lie group with $\dim G>1$ then $G$ satisfies the countable separation condition. 
\end{lemma}

\begin{proof}
	Let $q\colon\widetilde{G}\to G$ be the universal covering of $G$. 
	Then $\widetilde{G}$ is a simply connected, separable, abelian Banach-Lie group, hence there exists a separable Banach space~$\Xc$ with $\widetilde{G}=(\Xc,+)$. 
	On the other hand, $\Gamma:=\Ker q$ is a discrete subgroup of  $\widetilde{G}$ and 
	one has the isomorphism of Banach-Lie groups $\widetilde{G}/\Gamma\to G$, $x+\Gamma\mapsto q(x)$. 
	Since $\Gamma$ is a discrete subgroup of $(\Xc,+)$ and $\Xc$ is separable, it is straightforward to show that $\Gamma$ is at most countable. 
	
	Thus, we may assume 
	$G=(\Xc/\Gamma,+)$, where $\Xc$ is a separable real Banach space and $\Gamma\subseteq\Xc$ is a discrete subgroup with 
	$\card\,\Gamma\le\aleph_0$,    
	while $q\colon\Xc\to\Xc/\Gamma$ is the corresponding quotient map. 
	
	We must prove that for any subset $A\subseteq G$  with $\card\, A\le\aleph_0$, its complement $G\setminus A$ is connected. 
	The set $q^{-1}(A)=A+\Gamma$ satisfies $\card\,(q^{-1}(A))\le\aleph_0$ and 
	$q(\Xc\setminus q^{-1}(A))=(\Xc/\Gamma)\setminus A$, 
	hence it suffices to prove that $\Xc\setminus q^{-1}(A)$ is connected. 
	
	Therefore we may assume $\Gamma=\{0\}$ and $A\subseteq\Xc$ with $\card\, A\le\aleph_0$. 
	In this case we can actually prove that the complement $\Xc\setminus A$ is pathwise connected. 
	In fact, for any linearly independent vectors $x_0,x_1\in\Xc\setminus A$ with $x_0\ne x_1$ let $\Vc:=\RR x_0\dotplus\RR x_1$. 
	For every $s\in[0,1]$ let $\gamma_s\colon [0,1]\to \Vc$ with $$\gamma_s(0)=x_0,\ \gamma_s(1)=x_1,\ \gamma_s(1/2)=s(x_0+x_1),$$
	and 
	$$\gamma_s(t)=
	\begin{cases}
	(1-2t) x_0+2t\gamma_s(1/2)&\text{ if }0\le t<1/2,\\
	2(1-t)\gamma_s(1/2)+(2t-1)x_1&\text{ if }1/2<t\le 1.
	\end{cases}$$
	Then one has $\gamma_{s_1}([0,1])\cap\gamma_[s_2]([0,1])=\emptyset$ if $s_1,s_2\in[0,1]$ with $s_1\ne s_2$. 
	Since $\card\,A\le\aleph_0$, it then follows that there exists $s\in[0,1]$ with $\gamma_s([0,1])\cap A=\emptyset$, that is, $\gamma_s$ is a continuous path that connects $x_0$ to $x_1$ \emph{inside} $\Xc\setminus A$. 
	
	If however  $x_0,x_1\in\Xc\setminus A$ are linearly dependent then, using the hypothesis $\dim\Xc>1$, we may select $x_2\in\Xc\setminus A$ 
	which is linearly independent on $x_0$ and $x_1$. 
	Then, as above, we construct continuous paths in $\Xc\setminus A$ connecting $x_0$ to $x_2$ and $x_1$ to $x_2$, respectively. 
	The concatenation of these two paths is a continuous path in $\Xc\setminus A$ connecting $x_0$ to $x_1$, and we are done. 
\end{proof}

\begin{definition}
	\normalfont 
	Let $T$ be a topological space. A subset $P\subseteq T$ is called \emph{perfect} if it is closed and has no isolated points. 
	A \emph{Bernstein subset} is any subset $B\subseteq T$ satisfying $P\cap B\ne\emptyset$ and $P\cap(T\setminus B)\ne\emptyset$ for every perfect subset $P\subseteq T$ with $P\ne\emptyset$. 
	A \emph{strongly Bernstein subset} is any subset $B\subseteq T$ satisfying $F\cap B\ne\emptyset$ and $F\cap(T\setminus B)\ne\emptyset$ for every closed subset $F\subseteq T$ with $\card\,F>\aleph_0$. 
	
	If moreover $T$ is a topological group, then a \emph{Bernstein subgroup} of $T$ is any subgroup $B\subseteq T$ which is in addition a Bernstein subset 
	of~$T$. 
	One similarly defines the strongly Bernstein subgroups. 
\end{definition}

\begin{remark}\label{B-vs-sB}
	\normalfont
	It easily follows by Baire's category theorem that every complete metric space without isolated points is uncountable. 
	In particular, if $T$ is a complete metric space, then every perfect subset $P\subseteq T$ satisfies $\card\, P>\aleph_0$. 
	(See also \cite[Cor. 6.1.4]{En89} for connected metric spaces.)
	Hence every Bernstein subset of a complete metric space is a strongly Bernstein subset. 
	
	In every Polish space (i.e., separable complete metric space) every closed subset is the union of a perfect set and a countable set by the Cantor-Bendixson theorem \cite[Ch. I, Th. 6.4]{Ke95}, 
	and it then easily follows that in this framework the notions of Bernstein subset and strongly Bernstein subset coincide. 
	
	We also mention that the Bernstein subsets of any uncountable Polish space (i.e., uncountable separable complete metric space) are exactly the absolutely nonmeasurable subsets by \cite[Th. 3, page 206]{Kh09}, that is, the subsets that do not belong to the domain of any nonzero $\sigma$-finite continuous  Borel measure on that space. 
\end{remark}

\begin{lemma}\label{dense-connect}
	Assume that $T$ is a topological space satisfying the condition that every point of $T$ has a neighbourhood basis consisting of open neighbourhoods homeomorphic to open subsets of suitable topological vector spaces different from~$\{0\}$. 
	Then every Bernstein subset of $T$ is dense in $T$. 
	
	If moreover $T$ is a complete metric space that satisfies the countable separation condition
	and $H$ is a strongly Bernstein subset of $T$, 
	then $H$ is connected. 
\end{lemma}

\begin{proof}
	Let $B\subseteq T$ be a Bernstein subset. 
	We first prove that $H$ is dense in $T$. 
	In fact, arbitrary $x\in T$ has by hypothesis a neighbourhood basis consisting of open subsets $U_x\subseteq T$  for which there exist a
	topological vector space $\Xc$ and an open subset $U_0\subseteq \Xc$ with $x\in U_x$, $0\in U_0$, as well as a homeomorphism $\chi\colon U_0\to U_x$ with $\chi(0)=x$. 
	Since $U_0$ is an open neighbourhood of $0\in\Xc$ and $\Xc\ne\{0\}$, we may select $v\in U_0\setminus\{0\}$. 
	Again, since $U_0$ is a neighbourhood of $0\in\Xc$, there exists $\varepsilon>0$ such that $sv\in U_0$ for all $s\in[-\varepsilon, \varepsilon]$. 
	Then $\{\chi(sv)\mid s\in[-\varepsilon, \varepsilon] \}$ is clearly a perfect subset of $T$, and moreover this perfect set is contained in $\chi(U_0)=U_x$. 
	On the other hand, since $H$ is a Bernstein subset, its intersection with every perfect subset is nonempty, hence we obtain $B\cap U_x\ne\emptyset$. 
	Since these open sets $U_x$ constitute a neighbourhood basis of the arbitrary point $x\in T$, it follows that $H$ is dense in $T$. 
	
	Now assume that $T$ is a complete metric space satisfying the countable separation condition
	and $H$ is a strongly Bernstein subset of $T$. 
	In particular, $H$ is a Bernstein subset by Remark~\ref{B-vs-sB},
	hence is dense in $T$ by what we already proved. 
	
	To show that $H$ is connected, assume that one has open subsets $V,W\subseteq T$ with $B\subseteq V\cup W$ and $(B\cap V)\cap(B\cap W)=\emptyset$. 
	Then $B\cap(V\cap W)=\emptyset$.
	Since $H$ is dense and $V\cap W$ is open, we then obtain $V\cap W=\emptyset$. 
	Denoting $F:=T\setminus(V\cup W)$, it follows that $F\subseteq T$ is a closed subset satisfying $T\setminus F=V\cup W$ with  $V\cap W=\emptyset$. 
	Since $T$ satisfies the countable separation condition, we then obtain either $V=\emptyset$, or $W=\emptyset$, or $\card\, F>\aleph_0$. 
	If $\card\, F>\aleph_0$ then, using the hypothesis that $H$ is a strongly Bernstein subset, we obtain $B\cap F\ne\emptyset$, which is a contradiction with $B\subseteq V\cup W=T\setminus F$.
	Consequently either $V=\emptyset$, or $W=\emptyset$, and this completes the proof of the fact that $H$ is connected. 
\end{proof}

\begin{remark}
	\normalfont 
	The proof of Lemma~\ref{dense-connect} actually shows a stronger  connectedness property of strongly Bernstein subsets $H$ of any complete metric space $T$ satisfying the countable separation condition: 
	If $V,W\subseteq T$ are open subsets with $B\subseteq V\cup W$ 
	and $V\ne\emptyset\ne W$, then 
	$B\cap(V\cap W)\ne\emptyset$. 
\end{remark}

\begin{lemma}\label{push}
	Let $q\colon T\to S$ be a continuous surjective map between topological spaces. 
	If $B\subseteq T$ is a strongly Bernstein subset then $q(B)\subseteq S$ is a strongly Bernstein subset. 
	
	If moreover $T$ and $S$ are Hausdorff topological spaces and $q\colon T\to S$ is a continuous open surjective map, then for every Bernstein subset $B\subseteq T$ its image $q(B)\subseteq S$ is Bernstein subset.  
\end{lemma}

\begin{proof}
	For any closed subset $F\subseteq S$ with $\card\,F>\aleph_0$, 
	the subset $q^{-1}(F)\subseteq T$ is also closed since $q$ is continuous, and moreover, since $q$ is surjective, $\card\, q^{-1}(F)\ge \card\, F>\aleph_0$. 
	Since $B\subseteq T$ is a strongly Bernstein subset, it then follows that $q^{-1}(F) \cap B \ne\emptyset\ne q^{-1}(F)\cap(S\setminus B)$. 
	Therefore, using again that $q$ is surjective, 
	\begin{align*}
	q^{-1}(F\cap q(B))
	&=q^{-1}(F) \cap q^{-1}(q(B)) \\
	&=q^{-1}(F)\cap B \\
	&\ne\emptyset,
	\end{align*} 
	which implies $F\cap q(B)\ne\emptyset$. 
	Similarly, 
	\begin{align*}
	q^{-1}((S\setminus F)\cap q(B))
	& =q^{-1}(S\setminus F)\cap q^{-1}(q(B)) \\
	&=(T\setminus q^{-1}(F))\cap B \\
	&\ne\emptyset,
	\end{align*} 
	which implies $(S\setminus F)\cap q(B)\ne\emptyset$. 
	Hence $q(B)\subseteq S$ is a strongly Bernstein subset. 
	
	Now let us assume that $q\colon T\to S$ is a continuous open surjective homomorphism of topological groups. 
	It suffices to prove that for every perfect subset $P\subseteq S$ its preimage $q^{-1}(P)\subseteq T$ is a perfect subset, and then the proof follows the same lines as above. 
	Since $q$ is continuous, the set $q^{-1}(P)$ is closed, so it suffices to show that every point $x\in q^{-1}(P)$ belongs to the closure of $q^{-1}(P)\setminus \{x\}$. 
	In fact, one has $q(x)\in q(q^{-1}(P))=P$ (since $q$ is surjective). 
	Since $P$ is a perfect set, there exists a generalized sequence $\{y_j\}_{j\in J}$ in $P\setminus \{q(x)\}$ with $\lim\limits_{j\in J}y_j=q(x)$. 
	The mapping $q$ is assumed to be continuous, open, and surjective, hence it has the limit covering property. 
	That is, there exist a generalized subsequence $\{z_i\}_{i\in I}$ of $\{y_j\}_{j\in J}$ and a generalized subsequence $\{x_i\}_{i\in I}$ in $T$ with $q(x_i)=z_i$ for every $i\in I$ and moreover $\lim\limits_{i\in I}x_i=x$. 
	We note that for every $i\in I$ one has 
	$$x_i\in q^{-1}(z_i)\subseteq q^{-1}(P\setminus\{q(x)\})=q^{-1}(P)\setminus q^{-1}(\{q(x)\})\subseteq q^{-1}(P)\setminus\{x\}$$
	hence $x$ indeed belongs to the closure of $q^{-1}(P)\setminus \{x\}$, 
	and we are done. 
\end{proof}

\begin{remark}\label{transfinite_R}
	\normalfont
	In the proof of Lemma~\ref{transfinite} below we use a few basic notions of cardinals and ordinals. 
	See for instance \cite[App. A]{Ke95}. 
	We denote by $\omega_1$ the first uncountable ordinal and 
	$\Omega$ be the set of all countable ordinals~$\le \omega_1$. 
	Hence $\Omega$ is a well-ordered set for which $\omega_1$ is the largest element, and for every $\alpha\in\Omega\setminus\{\omega_1\}$ the set $ \{\beta\in\Omega\mid \beta<\alpha\}\equiv\alpha$ is at most countable. 
	Here we identify every $\alpha \in\Omega$ with the set of its predecessors, and in this way we can speak about the cardinality of every ordinal. 
	
	We identify the finite ordinals with the natural numbers $\{1,2,\dots,n\}\equiv n$ for $n=1,2,\dots$, and thus $\{1,2,\dots\}=\NN$ is the first infinite ordinal, with its cardinality denoted by~$\aleph_0$. 
	The cardinality of $\Omega\equiv\omega_1$ is denoted by~$\aleph_1$, 
	hence this is the first ordinal $>\aleph_0$. 
	Since one can prove that $\aleph_0< 2^{\aleph_0}$, it follows that  
	$\aleph_1\le 2^{\aleph_0}$. 
	
	We recall that the ``continuum hypothesis'' is the assumption $\aleph_1=2^{\aleph_0}$, that is, the cardinality of $\Omega$ is equal to the cardinality of the set of all subsets of $\NN$. 
	We also recall that $\card\,\RR=2^{\aleph_0}$, which follows by writing the real numbers in base~2, and then one has $\card\,\RR=\aleph_1$ if the continuum hypothesis is assumed. 
\end{remark}

\begin{lemma}\label{count}
	If $X$ is a separable metric space with its set of all closed subsets denoted by $\Fc$, then $\card\,\Fc\le 2^{\aleph_0}$. 
	
	If moreover the metric of $X$ is complete and there exists a continuous injective map $\gamma\colon[0,1]\to X$, 
	then 
	$$\card\,\Pc= \card\,\Cc=\card\,\Fc=2^{\aleph_0},$$
	where  $\Pc$ is the set of all perfect subsets of $\Xc$, and $\Cc:=\{F\in\Fc\mid \card\, F>\aleph_0\}$. 
\end{lemma}

\begin{proof}
	Since $X$ is a separable metric space, its topology has a basis $\Bc$ with $\card\,\Bc=\aleph_0$. 
	Every open subset of $X$ is the equal to $\bigcup\limits_{D\in\Sc}D$ for some subset $\Sc\subseteq\Bc$, hence it follows that the cardinality of the set of all open subsets of $X$ is $\le$ the cardinality of all subsets of $\Bc$, that is $ 2^{\aleph_0}$. 
	Since the complement correspondence 
	$$D\mapsto X\setminus D$$ 
	is a bijection between the open subsets and the closed subsets of $X$, it follows  that  $\card\,\Fc\le  2^{\aleph_0}$. 
	
	Now assume the metric of $X$ is complete and there exists a continuous injective map $\gamma\colon[0,1]\to X$. 
	One has by Remark~\ref{B-vs-sB}
	$$\Pc\subseteq\Cc\subseteq\Fc.$$
	All the sets $\gamma([0,r])\subseteq X$ for arbitrary $r\in(0,1)$ are compact, connected, non-singletons, and distinct since $\gamma$ is injective. 
	In particular $\{\gamma([0,r])\}_{0<t<1}$ is a family of distinct elements of~$\Pc$. 
	Using $\card\,(0,1)=2^{\aleph_0}$, we then obtain $\card\,\Pc\ge 2^{\aleph_0}$, 
	and then the assertion follows. 
\end{proof}

\begin{lemma}\label{transfinite}
	Assume the continuum hypothesis. 
	If $\Xc$ is a separable real Banach space, 
	then there exists a $\QQ$-linear subspaces $\Zc,\Wc\subseteq\Xc$ whose underlying additive groups $(\Zc,+)$ and $(\Wc,+)$ are strongly Bernstein subgroups of $(\Xc,+)$ and moreover one has the direct sum decomposition $\Xc=\Zc\dotplus\Wc$.  
\end{lemma}

\begin{proof}
	Let $\Fc$ be the set of all closed subsets of $\Xc$ and denote  $$\Cc:=\{F\in\Fc\mid \card\, F>\aleph_0\}.$$ 
	Since $\Xc$ is a separable complete metric space, it follows by Lemma~\ref{count} along with the the continuum hypothesis that  
	$ \card\,\Cc=\aleph_1$. 
	
	Now resume the notation of Remark~\ref{transfinite_R} and denote  $\Omega_0:=\Omega\setminus\{\omega_1\}$. 
	It is known that $\card\,\Omega_0=\aleph_1$, hence $\card\,\Omega_0=\card\,\Cc$,  
	and then the set $\Cc$ can be bijectively indexed as $\Cc=\{F_\alpha\}_{\alpha\in\Omega_0}$. 
	
	We now construct the required $\Zc$ by transfinite induction. 
	Since $\card\, F_1>\aleph_0=\card\,\QQ$, we may select 
	$x_1\in F_1\setminus\{0\}$ and then $y_1\in F_1\setminus \QQ x_1$, hence 
	$x_1,y_1\in F_1$ are linearly independent over~$\QQ$. 
	
	For the induction step,  
	let $\alpha\in\Omega_0$ arbitrary and assume that we have 
	$x_\beta,y_\beta\in F_\beta$ for every $\beta\in\Omega_0$ with $\beta<\alpha$, whose corresponding set 
	$$S_\alpha:=\{x_\beta\mid \beta<\alpha\}\cup\{y_\beta\mid \beta<\alpha\}$$
	is linearly independent over~$\QQ$. 
	Using the fact that $\{\beta\in\Omega_0\mid \beta<\alpha\}$ is countable (see  Remark~\ref{transfinite_R}), we see that the rational vector space $$\Vc_\alpha:=\spa_{\QQ} (S_\alpha)$$ 
	satisfies  
	$\card\,\Vc_\alpha\le\aleph_0<\card\,F_\alpha$, hence 
	we can select $x_\alpha\in F_\alpha\setminus \Vc_\alpha$. 
	Similarly, one can after that select $y_\alpha\in F_\alpha\setminus (\Vc_\alpha+\QQ x_\alpha)$. 
	Thus $x_\alpha,y_\alpha\in F_\alpha$ and the set $S_\alpha\cup\{x_\alpha,y_\alpha\}$ is linearly independent over~$\QQ$, 
	which completes the induction step. 
	
	We now define 
	$$\Zc:=\spa_{\QQ}(\{x_\alpha\mid\alpha\in\Omega_0\}).$$ 
	This is a $\QQ$-vector subspace of $\Xc$ and for every $F=F_\alpha\in\Cc$ one clearly has $x_\alpha\in F_\alpha\cap\Zc$ and $y_\alpha\in F_\alpha\setminus\Zc$, hence $F\cap \Zc\ne\emptyset$ and $F\cap (\Xc\setminus \Zc)\ne\emptyset$. 
	Therefore $\Zc$ is a strongly Bernstein subset of $\Xc$. 
	
	Moreover, let $\Ac$ denote the set of all $\QQ$-vector subspaces $\Yc\subseteq\Xc$ satisfying $\{y_\alpha\mid \alpha\in\Omega_0\}\subseteq\Yc$ and $\Zc\cap\Yc=\{0\}$. 
	It is easily checked that the set $\Ac$ is inductively ordered upwards with respect to the partial ordering defined by the inclusion, 
	hence there exists a maximal element $\Wc\in\Ac$ by Zorn's lemma. 
	Then one has $\Zc\cap\Wc=\{0\}$ since $\Wc\in\Ac$, and on the other hand $\Zc+\Wc=\Xc$ as a direct consequence of the maximality property of~$\Wc$. 
	Moreover, for every $F=F_\alpha\in\Cc$ one has $y_\alpha\in F_\alpha\cap\Wc$ and $x_\alpha\in F_\alpha\setminus\Wc$, hence $F\cap \Wc\ne\emptyset$ and $F\cap (\Xc\setminus \Wc)\ne\emptyset$. 
	Thus $\Wc$ is also a strongly Bernstein subset of $\Xc$,  
	and we are done. 
\end{proof}

\begin{remark}
	\normalfont
	The proof of Lemma~\ref{transfinite} follows the lines of \cite[Lemma 1, page 284; Ex. 4, page 294]{Kh09}.  
\end{remark}

\begin{lemma}\label{B_vs_pseudo_discr}
	Every (strongly) Bernstein subgroup of a Banach-Lie group is a pseudo-discrete subgroup. 
\end{lemma}

\begin{proof}
	Let $H\subseteq G$ be a (strongly) Bernstein subgroup of a Banach-Lie group. 
	To prove that $H$ is pseudo-discrete, let $\gamma\in\Ci(D_\gamma,G)$ arbitrary with $\gamma(D_\gamma)\subseteq H$, where $D_\gamma\subseteq\RR$ is an open interval. 
	For every compact interval $K\subseteq D_\gamma$, its image $\gamma(K)$ is a second countable, compact, connected subset of $G$ since $\gamma\colon D_\gamma\to G$ is smooth, hence continuous. 
	Since $\gamma(K)$ is compact and second countable, it is metrizable by \cite[Th. 4.2.8]{En89}, and in fact any metric defining its topology is complete since $\gamma(K)$ is compact. 
	
	Since $\gamma(K)$ is connected, it follows that it has no isolated points, and then either $\card\,\gamma(K)=1$ or $\card\,\gamma(K)>\aleph_0$ by Remark~\ref{B-vs-sB}. 
	Since $\gamma(K)\subseteq \gamma(D_\gamma)\subseteq H$ and $H$ is a (strongly) Bernstein subgroup, it then follows that $\card\,\gamma(K)=1$ 
	for every compact interval $K\subseteq D_\gamma$, hence the mapping $\gamma\colon D_\gamma\to G$ is constant. 
	This shows that $H$ is a pseudo-discrete subgroup of~$G$, and we are done. 
\end{proof}

\begin{remark}
	\normalfont
	The proof of Lemma~\ref{B_vs_pseudo_discr} actually shows that if $X$ is a Hausdorff topological space with a (strongly) Bernstein subset $B\subseteq X$ then the following assertion holds: 
	If $D\subseteq \RR$ is an interval and  $\gamma\colon D\to X$ is a continuous mapping with $\gamma(D)\subseteq B$, then $\gamma$ is constant. 
	
	In particular, if $X$ is a separable Banach space with $\dim X>1$, 
	then it follows by Lemmas \ref{pseudo_discr_L4}--\ref{dense-connect} 
	that \emph{every Bernstein subset $B\subseteq X$ is connected but is quite far from being pathwise connected}, since every continuous path contained in $H$ is actually constant. 
\end{remark}

\begin{proposition}\label{pseudo_discr_ex}
	Assume the continuum hypothesis. 
	If $G$ is a connected separable abelian Banach-Lie group and $\dim G>1$, 
	then $G$ contains dense connected pseudo-discrete subgroups $H_1$ and $H_2$ with $G=H_1+H_2$. 
\end{proposition}

\begin{proof}
	As in the proof of Lemma~\eqref{pseudo_discr_L4} we may assume 
	$G=(\Xc/\Gamma,+)$, where $\Xc$ is a separable real Banach space and $\Gamma\subseteq\Xc$ is a discrete subgroup with 
	$\card\,\Gamma\le\aleph_0$. 
	Let $q\colon\Xc\to\Xc/\Gamma$ be the corresponding quotient map. 
	
	It follows by Lemma~\ref{transfinite} that the abelian group $(\Xc,+)$ contains two strongly Bernstein subgroups $\Zc_1$ and $\Zc_2$ for which one has the direct sum decomposition $\Xc=\Zc_1\dotplus\Zc_2$. 
	Since the quotient map $q\colon\Xc\to\Xc/\Gamma$ is continuous, open, and surjective, it follows by Lemma~\ref{push} that the set 
	$$H_j:=q(\Zc_j)$$ 
	is a strongly Bernstein subgroup of $G=\Xc/\Gamma$ for $j=1,2$. 
	Moreover, one has 
	$$H_1+H_2=q(\Zc_1+\Zc_2)=q(\Xc)=G.$$ 
	Also, $H_j$ is a pseudo-discrete subgroup of $G$ for $j=1,2$ by Lemma~\ref{B_vs_pseudo_discr}. 
	Finally, since $\Zc_j$ is connected and dense in $\Xc$ by Lemma~\ref{dense-connect}, and the quotient mapping $q\colon \Xc\to G$ is continuous and surjective, it follows that $H_j$ is connected and dense in $G$ for $j=1,2$, and we are done. 
\end{proof}

\begin{corollary}
	\label{pseudo_discr_cor1}
	Assume the continuum hypothesis. 
	If $G$ is a separable abelian Banach-Lie group and $\dim G>1$, 
	then $G$ contains a connected pseudo-discrete nontrivial subgroup. 
\end{corollary}

\begin{proof}
	This follows by a direct application of Proposition~\ref{pseudo_discr_ex} for the connected $\1$-component of~$G$. 
\end{proof}

\begin{theorem}
	\label{pseudo_discr_cor2}
	Assume the continuum hypothesis. 
	Let $G$ be a connected separable Banach-Lie group and denote by $\zg$ the center of the Lie algebra of~$G$. 
	Then~$G$ contains a connected pseudo-discrete normal nontrivial subgroup if and only if $\dim\zg>1$. 
\end{theorem}

\begin{proof} 
	Let us denote by $Z$ the center of $G$. 
	Then $Z$ is a Banach-Lie group with respect to its topology inherited from~$G$, and its Lie algebra is~$\zg$.  
	(See for instance \cite[Cor. IV.3.9, Th. IV.3.3, Prop. IV.3.4]{Ne06}.) 
	In particular, it follows that the connected $\1$-component $Z_0$ of $Z$ is a 
	Banach-Lie group with respect to its topology inherited from~$G$, and its Lie algebra is~$\zg$.
	
	If there exists a connected pseudo-discrete normal subgroup $H\subseteq G$, 
	then $H\subseteq Z$ by Proposition~\ref{pseudo-discr_prop}. 
	Moreover, since $H$ is connected and $\1\in H\cap Z_0$, one actually has $H\subseteq Z_0$. 
	This implies that $\dim\zg>1$ since otherwise the connected Banach-Lie group $Z_0$ would be either $\{\1\}$ (if $\dim\zg=0$) 
	or isomorphic to $\RR$ or to $\RR/\ZZ$ and in these cases it is easily seen that every connected pseudo-discrete subgroup of $Z_0$ is trivial. 
	
	Conversely, if $\dim\zg>1$, then the separable abelian Banach-Lie group~$Z$ satisfies $\dim Z>1$, hence there exists a connected pseudo-discrete nontrivial subgroup $H\subseteq Z$ by Corollary~\ref{pseudo_discr_cor1}.
\end{proof}

\begin{remark}
	\normalfont 
	Existence of Bernstein subgroups of arbitrary connected finite-dimensional Lie groups was stated in \cite{Th87} but the corresponding proof contains some gaps. 
	Constructions of Bernstein subgroups of separable Banach spaces can be found in \cite[Ex. 6, page 91; Lemma 1, page 284; Ex. 4, page 294]{Kh09}.

	The method of proof of Lemma~\ref{pseudo_discr_L4} goes back to \cite[page 202]{Ha62}. 
	See also \cite[Counterex. 124]{StSe78} and \cite[Th., page 110]{HoYo61} for existence of connected Bernstein subsets of 2-dimensional real vector spaces. 
\end{remark}

\newpage
\appendix

\section{Fr\"olicher spaces versus \H-manifolds}
\label{AppA}

In this section, we first recall (mostly without proofs) the definition and basic properties of Fr\"olicher  structure on sets  according to \cite{Mag} and \cite{Lau}.   
We then compare the notion of tangent space in the sense of Fr\"olicher spaces with the tangent spaces of \H-manifolds. 

\subsection{Definition and basic properties of Fr\"olicher spaces}

\begin{definition}\label{frolicher}  
	\normalfont
	A \emph{Fr\"olicher space} is a triple $(S,\Cg,\Fg)$
	satisfying the following conditions: 
	\begin{enumerate}[{\rm(i)}] 
		\item\label{frolicher_item1} $\Cg$ is a set of curves $c:\RR\to S$.
		\item\label{frolicher_item2} $\Fg$ is a set of functions $F:S\to \RR$.
		\item\label{frolicher_item3} For any map $c\colon \RR\to S$ one has $c\in\Cg$ if and only if $F\circ c\in\Ci(\RR,\RR)$ for all $F\in \Fg$.
		\item\label{frolicher_item4} For any map $F\colon S\to\RR$ one has $F\in\Fg$ if and only if $F\circ c\in\Ci(\RR,\RR)$ for all $c\in \Cg$. 
	\end{enumerate} 
	If  $(S, \Cg, \Fg)$  and $(S', \Cg', \Fg')$  are two  Fr\"olicher spaces,
	a map $f\colon S\to S'$ is called  \emph{smooth} if and only if $F'\circ f\circ c\in\Ci(\RR,\RR)$ for all $F'\in \Fg'$ and $c\in \Cg$. 
	If $f$ is bijective, then $f$ is called a \emph{diffeomorphism} if $f$ and $f^{-1}$ are smooth maps.
\end{definition}

\begin{example}\label{F_ex1}
	\normalfont 
	Every set $\Fg_0$ of maps from $S$ into $\RR$ naturally gives rise to the  Fr\"olicher  space structure $(S,\Cg(S,\Fg_0),\Fg(S,\Fg_0))$, where
	\begin{align*}
	\Cg(S,\Fg_0)
	&=\{ c:\RR\to S\mid F\circ c\in\Ci(\RR,\RR) \text{ for all } F\in \Fg \}, \\
	\Fg(S,\Fg_0)
	&=\{ F:S\to\RR\mid F\circ c\in\Ci(\RR,\RR) \text{ for all }
	c\in \Cg(S,\Fg_0)\}. 
	\end{align*}
	Clearly $\Fg_0\subseteq \Fg(S,\Fg_0)$. 
	
	Similarly, every set $\Cg_0$ of maps from $\RR$ into $S$ gives rise to the  Fr\"olicher  space structure $(S,\Cg(S,\Cg_0),\Fg(S,\Cg_0))$, where
	\begin{align*}
	\Fg(S,\Cg_0)
	&=\{ F:S\to\RR\mid F\circ c \text{ is smooth  for all } c\in \Cg \},\\ 
	\Cg(S,\Cg_0)
	&=\{ c:\RR\to S\mid F\circ c \text{ is smooth  for all } F\in \Fg(S,\Cg_0) \}, 
	\end{align*}
	and one has $\Cg_0\subseteq \Cg(S,\Cg_0)$. 
\end{example} 

\begin{example}\label{F_ex2}
	\normalfont 
	Every Banach manifold  $M$ has the canonical structure of a Fr\"olicher  space 
	$(M, \Cg_M, \Fg_M)$ as a special case of Example~\ref{F_ex1} for $\Cg_0:=\Ci(\RR,M)$. 
	Here 
	\begin{align*}
	\Fg_M 
	& :=
	\Fg(M,\Cg_0)=\{F\colon M\to\RR\mid F\circ c\in\Ci(\RR,\RR)\text{ for all }c\in\Ci(\RR,M)\} \\
	& =\Ci(M,\RR)
	\end{align*}
	where the last equality follows for instance by \cite[Th. 12.8]{KrMi} since $M$ is a Banach manifold. 
	Moreover, 
	$$\Cg_M:=\Cg(M,\Cg_0)=\{c\colon \RR\to M\mid F\circ c\in\Ci(\RR,\RR)\text{ for all }F\in\Ci(M,\RR)\}.$$
	Clearly, $\Ci(\RR,M)\subseteq \Cg_M$, and one has $\Ci(\RR,M)=\Cg$ 
	if there exist ``sufficiently many'' smooth functions on~$M$, 
	for instance the Banach manifold $M$ is smoothly paracompact, 
	cf. \cite[Th. 2]{Fr81}.
	
	We note that, by the construction of Example~\ref{F_ex1} for 
	$\Fg_0:=\Ci(M,\RR)$ one obtains the same Fr\"olicher structure 
	$(M, \Cg_M, \Fg_M)$ as above. 
\end{example}

\begin{example}\label{F_ex3}
	\normalfont 
	Let $\pi: S \to S/\mathcal{R}$ be the quotient space of $S$ by an equivalence relation $\mathcal{R}$. 
	If $(S,\Cg,\Fg)$ is a Fr\"olicher  space then one naturally obtains the Fr\"olicher  space structure $(S/\mathcal{R},\Cg/\mathcal{R},\Fg/\mathcal{R})$, 
	where 
	\begin{align*}
	\Fg/\mathcal{R}&:=\Fg(S/\mathcal{R},\Cg_\pi),\\ 
	\Cg/\mathcal{R}&:=\Cg(S/\mathcal{R},\Cg_\pi), 
	\end{align*}
	are defined as in Example~\ref{F_ex1} with respect to the set of curves 
	$\Cg_\pi:=\{\pi\circ c\mid c\in\Cg\}$ in $S/\mathcal{R}$. 
	We note that 
	\begin{align*}
	\Fg(S/\mathcal{R},\Cg_\pi)
	& =\{ F\colon S/\mathcal{R}\to\RR\mid F\circ \pi\circ c \text{ is smooth  for all } c\in \Cg \} \\
	& =\{ F\colon S/\mathcal{R}\to\RR\mid F\circ \pi\in \Fg \}
	\end{align*}
	which in particular shows that  $\pi: S \to S/\mathcal{R}$ is smooth.
\end{example}

\begin{example}\label{F_ex4}
	\normalfont 
	For $i=1,2$,  let  $(S_i,\Cg_i,\Fg_i)$ be two Fr\"olicher spaces. 
	Then $S_1\times S_2$ is naturally endowed with the Fr\"olicher structure   $(S_1\times S_2,\Cg_1\times \Cg_2,\Fg_1 \times \Fg_2)$ 
	where  
	\begin{align*}
	\Fg_1 \times \Fg_2&:=\Fg(S_1\times S_2,(\Cg_1\times \Cg_2)_0),\\ 
	\Cg_1\times \Cg_2&:=\Cg(S_1\times S_2,(\Cg_1\times \Cg_2)_0), 
	\end{align*}
	are defined as in Example~\ref{F_ex1} with respect to the set of curves 
	$
	(\Cg_1\times \Cg_2)_0 :=\{(c_1,c_2)\colon \RR\to S_1\times S_2\mid c_j\in \Cg_j \text{ for }j=1,2 \}$. 
\end{example}

\begin{definition} 
	\normalfont 
	A \emph{Fr\"olicher group} is a group $G$ provided with a Fr\"olicher structure $(G,\Cg,\Fg)$ for which the map $G\times G\to  G$, $ (g_1,g_2)\mapsto g_1g_2^{-1}$, is smooth. 
\end{definition}

\subsection{Tangent bundle of a Fr\"olicher  space}

\begin{definition}
\label{tgF}
\normalfont
Let $(S, \Cg, \Fg)$ be a Fr\"olicher   space and $s\in S$. 
We set $\Cg_s:=\{c \in \Cg \mid  c(0)=s\}$.  
On $\Cg_s$ we define the following equivalence relation:
$$c_1\simeq c_2\iff 
(\forall F\in \Fg)\quad \frac{\de}{\de t}\Bigl\vert_{t=0}(F\circ c_1)(t)
=\frac{\de}{\de t}\Bigl\vert_{t=0}(F\circ c_2)(t).$$
An equivalence class in $\Cg_s$ will be denoted by $[c]$. 
The \emph{Fr\"olicher-tangent space}  of $S$ at $s$ is the quotient 
$$T^\Fr_s S:=\Cg_s/ \simeq\quad =\{[c]\mid c\in \Cg_s\}$$ 
and we also denote by $q^\Fr_s\colon \Cg_s\to T^\Fr_s S $, $c\mapsto[c]$, the corresponding quotient map. 
\end{definition}

Unlike the usual tangent space of Banach manifolds, in general $T^\Fr_sS$ {\bf is not a vector space}. 
(See \cite[Ex. 2.2]{Lau}.)
However, we have the following result. 

\begin{proposition}
	\label{FrolicherGr}
	If $G$ is a  Fr\"olicher  group, then for any $g\in G$ the tangent space $T^\Fr_gG$ is a vector space. In particular, we have
	\begin{eqnarray}\label{suminverseTG}
	[c_1]+[c_2]=[s\mapsto c_1(s)g^{-1}c_2(s)]\nonumber\\
	-[c]=[s\mapsto gc(s)^{-1}g].\hfill{}
	\end{eqnarray}
\end{proposition}

\begin{definition}\label{tg-F}
	\normalfont
The \emph{tangent bundle} $T^\Fr S$ of a  Fr\"olicher  space $S$ is 
$$T^\Fr S:=\bigsqcup_{x\in S}T^\Fr_xS.$$
If $\phi\colon S\to S'$ is a smooth map, its \emph{tangent map} is  $$T^\Fr\phi\colon T^\Fr S\to T^\Fr S',\quad 
T^\Fr\phi([c])=[\phi\circ c].$$ 
We set  $T^\Fr_x\phi=T^\Fr\phi\vert_{T^\Fr_xS}$.
If $T^\Fr_xS$ is a vector space so is $T^\Fr_x\phi(T^\Fr_xS)\subset T^\Fr_{\phi(x)}S'$.
\end{definition}

If $(S, \Cg, \Fg)$ is a  Fr\"olicher  space, then we can endow $T^\Fr S$ with a natural Fr\"olicher   structure 
$(T^\Fr S,\Cg_{T^\Fr S},\Fg_{T^\Fr S})$, 
where 
\begin{align*}
\Fg_{T^\Fr S}&:=\Fg(T^\Fr S,\Cg_{T^\Fr S}^0),\\ 
\Cg_{T^\Fr S}&:=\Cg(T^\Fr S,\Cg_{T^\Fr S}^0), 
\end{align*}
are defined as in Example~\ref{F_ex1} with respect to the set of curves 
defined by:
$$\Cg_{T^\Fr S}^0=\{c\colon\RR\to T^\Fr S\mid (\forall f\in \Fg)\quad 
Tf\circ c\in\Ci(\RR,\RR)\}.$$
A detailed discussion of the following properties can be found in \cite{Lau}.

\begin{theorem} \label{TS}
	The following assertions hold: 
	\begin{enumerate}[{\rm 1.}]
		\item The natural projection $p_S:T^\Fr S\to S$ is a smooth map.
		\item If $S$ and $S'$ are Fr\"olicher  spaces, then $T^\Fr (S\times S')$ is diffeomorphic to $T^\Fr S\times T^\Fr S'$.
		\item Let $\phi: S\to S'$ be a smooth map between  Fr\"olicher  spaces, then $T^\Fr \phi: T^\Fr S\to T^\Fr S'$ is a smooth map.
		\item Let $G$ be a Fr\"olicher  group. 
		Then $T^\Fr G$ has the natural structure of a Fr\"olicher  group. 
		The projection $p_G:T^\Fr G\to G$ is a group morphism. 
		If we set $\mathfrak{g}=T^\Fr _\1G$  (where $\1\in G$ is the unit element) then the map
		$\Phi:T^\Fr G\to G\ltimes \mathfrak{g}$ given by $\Phi([c]):=(c(0), c(0)^{-1}[c])$ is an isomorphism of Fr\"olicher groups\footnote{That is, $\Phi $ is a diffeomorphism of Fr\"olicher spaces and a group isomorphism.}
		when $G\ltimes \mathfrak{g}$ is provided with a structure of semidirect product of Fr\"olicher groups. 
	\end{enumerate}
\end{theorem}

In the case of Banach manifolds there exists a natural map from the usual tangent space into the tangent space in the sense of Fr\"olicher spaces, 
as explained in the following remark. 

\begin{remark}
	\label{classical-F}
	\normalfont 
	Let $M$ be a Banach manifold with its canonical structure of a Fr\"olicher space $(M,\Cg_M,\Fg_M)$ as in Example~\ref{F_ex2}, where $\Ci(M,\RR)\subseteq\Cg_M$. 
	For every $x\in M$ and $c_1,c_2\in \Ci(M,\RR)$ with $c_1(0)=c_2(0)=x$ and  $\dot{c}_1(0)=\dot{c}_2(0)\in T_xM$ we clearly have 
	$\frac{\de}{\de t}\Bigl\vert_{t=0}(f\circ c_1)(t)=\frac{\de}{\de t}\Bigl\vert_{t=0}(f\circ c_2)(t)$, 
	hence $c_1\simeq c_2$. 
	This shows that have a well-defined mapping 
	$$\Psi_x\colon T_xM\to T^\Fr_xM,\quad \dot{c}(0)\mapsto [c]$$
	where $c\in \Ci(M,\RR)\subseteq\Cg_M$ with $c(0)=x$. 
	We also define 
	$$\Psi\colon TM\to T^\Fr M,\quad \Psi\vert_{T_xM}:=\Psi_x.$$
	It is easily seen that the mapping $\Psi_x$ is surjective for every $x\in M$ if for instance $\Ci(M,\RR)=\Cg_M$. 
	On the other hand, the mapping $\Psi_x$ is injective if and only if the points of the Banach space $T_xM$ are separated by the continuous linear functionals $\{T_xf\colon T_xM\to\RR\mid f\in \Ci(M,\RR)\}$, that is, if for every $v\in T_xM$ we have $v=0$ if and only if $(T_xf)(v)=0$ for every $f\in\Ci(M,\RR)$. 
	If the manifold $M$ is finite-dimensional then both these conditions are satisfied and the above mapping $\Psi_x$ is bijective for every $x\in M$, 
	cf. \cite[Th. 2.12]{Lau}. 
	
	However, the above conditions are satisfied even for infinite-dimensional manifolds. 
	For instance, if $M$ is a Banach space, then its points are separated by the continuous linear functionals (by the Hahn-Banach theorem) hence the mapping $\Psi_x$ is injective for every $x\in M$. 
\end{remark}

\subsection{Relation between \H-structures and Fr\"olicher structures} 

The following result is a generalization of Example~\ref{F_ex2}. 

\begin{proposition}\label{HversusF}
	Every pre-\H-manifold $S$ has the canonical Fr\"olicher structure given by the triple $(S,\Fg_S,\Cg_S)$, 
	where $\Fg_S:=\Hom_{\H}(S,\RR)$ and 
	$$\Cg_S:=\{c\colon \RR\to S\mid F\circ c\in\Ci(\RR,\RR)\text{ for all }F\in\Hom_{\H}(S,\RR)\}$$
	and moreover $\Hom_{\H}(\RR,S)\subseteq \Cg_S$. 
\end{proposition}

\begin{proof}
	By Example~\ref{F_ex1} with $\Cg_0:=\Hom_{\H}(\RR,S)$, 
	we obtain the Fr\"olicher structure 
	$(S,\Cg(S,\Cg_0),\Fg(S,\Cg_0))$, 
	where 
	$$\Fg(S,\Cg_0)=\{F\colon S\to\RR\mid F\circ c\in \Ci(\RR,\RR)\text{ for all }c\in\Cg_0\}$$
	and 
	$$\Cg(S,\Cg_0)=\{c\colon \RR\to S\mid F\circ c\in \Ci(\RR,\RR)
	\text{ for all }F\in\Fg(S,\Cg_0)\}.$$ 
	Therefore 
	it suffices to prove the equality 
	\begin{equation}\label{HversusF_proof_eq1}
	\Fg(S,\Cg_0)=\Hom_{\H}(S,\RR).
	\end{equation}
	To this end let $F\colon S\to\RR$ arbitrary. 
	If $F\in\Hom_{\H}(S,\RR)$, then for every curve $c\in\Hom_{\H}(\RR,S)$ 
	one has $F\circ c\in \Ci(\RR,\RR)$, as noted in the last part of Remark~\ref{H-categ}. 
	
	Conversely, let us assume that for every $c\in\Hom_{\H}(\RR,S)$ one has 
	$F\circ c\in \Ci(\RR,\RR)$, and fix an arbitrary pre-\H-atlas $(M,\pi,S)$. 
	Since $\pi\in\Hom_{\H}(M,S)$, it follows that 
	for every smooth curve $d\in\Ci(\RR,M)$ one has $\pi\circ d\in \Hom_{\H}(\RR,S)$, hence our assumption implies 
	$F\circ \pi\circ d\in \Ci(\RR,\RR)$. 
	Here $F\circ \pi\colon M\to\RR$ and $M$ is a Banach manifold, 
	while  $d\in\Ci(\RR,M)$ is arbitrary, hence $F\circ\pi\in \Ci(M,\RR)$. 
	(See for instance \cite[Th. 12.8]{KrMi}.) 
	This further implies $F\in\Hom_{\H}(S,\RR)$, which completes the proof of~\eqref{HversusF_proof_eq1}, and we are done. 
\end{proof}

\begin{remark}\label{FsmoothHsmooth} 
	\normalfont
	If $(M,\pi,S)$ is a 
	pre-\H-atlas on $S$, then we obtain 
	a structure  of Fr\"olicher  space $(S,\Cg_M/\Rc_\pi,\Fg_M/\Rc_\pi)$ defined as in Example~\ref{F_ex3} with respect to the equivalence relation 
	$\Rc_\pi$ whose equivalence classes are the fibers of~$\pi$, 
	where $(M,\Cg_M,\Fg_M)$ is the canonical Fr\"olicher structure of the Banach manifold $M$ as in Example~\ref{F_ex2}. 
	In particular $\Fg_M=\Ci(M,\RR)$ and 
	$$\Fg_M/\Rc_\pi=\{F\colon S\to\RR\mid F\circ\pi\in\Ci(M,\RR)\}
	=\Hom_{\H}(S,\RR).$$	
	(See also Remark~\ref{H-categ} for the last equality.)
	Moreover 
	\begin{align*}
	\Cg_M/\Rc_\pi
	&=\{c\colon\RR\to S\mid F\circ c\in\Ci(\RR,\RR)\text{ for all }F\in \Fg_M/\Rc_\pi\} \\
	&=
	\{c\colon\RR\to S\mid F\circ c\in\Ci(\RR,\RR)\text{ for all }F\in \Hom_{\H}(S,\RR)\}.
	\end{align*}
	Thus the quotient Fr\"olicher structure $(S,\Cg_M/\Rc_\pi,\Fg_M/\Rc_\pi)$ is actually independent on the pre-\H-atlas $(M,\pi,S)$ and it coincides with the canonical Fr\"olicher structure $(S,\Cg_S,\Fg_S)$ given by Proposition~\ref{HversusF}. 
\end{remark}

\begin{remark}\label{FsmoothHsmooth1} 
	\normalfont
	If $S_1$ and $S_2$ are pre-\H-manifolds and 
	$f\in\Hom_{\H}(S_1,S_2)$, then $f$ is a smooth map with respect to the canonical Fr\"olicher structures of $S_1$ and $S_2$. 
	
	In fact, for arbitrary $c\in\Cg_{S_1}$ and $F\in\Fg_{S_2}$ 
	one has $F\in \Hom_{\H}(S_2,\RR)$ hence $F\circ f\in \Hom_{\H}(S_1,\RR)$ 
	since any composition of \H-smooth mapping is \H-smooth. 
	(See the end of Remark~\ref{H-categ}.) 
	Then, by the definition of $\Cg_{S_1}$ in Proposition~\ref{HversusF}, 
	we obtain $(F\circ f)\circ c\in\Ci(\RR,\RR)$, and this shows that 
	$f$ is a smooth map with respect to the canonical Fr\"olicher structures of $S_1$ and $S_2$. 
\end{remark}

\subsection{Tangent spaces of pre-\H-manifolds vs. tangent spaces of Fr\"olicher spaces}

\begin{definition}\label{preHtg_def}
	\normalfont
	Let $S$ be a pre-\H-manifold. 
	For any $s\in S$ we define 
	$$\Cg^\H_s(S):=\{\gamma\in\Hom_{\H}(J_\gamma,S)\mid 0\in J_\gamma\subseteq \RR\text{ open interval},\ \gamma(0)=s\}$$	
	Fix any pre-\H-atlas $(M,\pi,S)$.  
	For any $\gamma_1,\gamma_2\in\Cg^\H_s(S)$, we write $\gamma_1\mathop{\sim}\limits^\pi\gamma_2$ if there exist 
	an open interval $J\subseteq \RR$ with $0\in J\subseteq J_{\gamma_1}\cap J_{\gamma_2}$, and, for $j=1,2$ there exist $c_j\in\Ci(J,M)$ 
	with $\pi\circ c_j=\gamma_j\vert_J$ for $j=1,2$,
	and there exists a transverse diffeomorphism $h\colon U_1\to U_2$ with 
	$h(c_1(0))=c_2(0)$, 
	and $(T_{c_1(0)}h)(\dot{c}_1(0))=\dot{c}_2(0)$. 
	
	Then $\mathop{\sim}\limits^\pi$ is an equivalence relation on $\Cg^\H_s(S)$ 
	which does not depend on the choice of the pre-\H-atlas $(M,\pi,S)$ 
	by Proposition~\ref{preHtg_lemma}\eqref{preHtg_lemma_item2} below, so we denote it simply by $\sim$, and 
	we define the quotient space 
	$$T^\Q_sS:=\Cg^\H_s(S)/\sim$$
	with its corresponding quotient map $q_s\colon \Cg^\H_s(S)\to T^Q_sS$. 
\end{definition}

\begin{proposition}\label{preHtg_lemma}
	If  $S$ is a pre-\H-manifold, then the following assertions hold: 
	\begin{enumerate}[{\rm(i)}]
		\item\label{preHtg_lemma_item1} 
		Let  $(M,\pi,S)$ be a pre-\H-atlas. 
		For any open interval $J\subseteq \RR$ with $0\in J$,  
		if $c_j\in\Ci(J,M)$ for $j=1,2$ and $h\colon U_1\to U_2$  is a transverse diffeomorphism with 
		$c_j(0)\in U_j$ for $j=1,2$, $h(c_1(0))=c_2(0)$,   
		and $(T_{c_1(0)}h)(\dot{c}_1(0))=\dot{c}_2(0)$,  
		then $\pi\circ c_1 \mathop{\sim}\limits^\pi \pi\circ c_2$. 
		\item\label{preHtg_lemma_item2} 
		The binary relation $\mathop{\sim}\limits^\pi$ is an equivalence relation on $\Cg^\H_s(S)$ 
		that does not depend on the choice of the pre-\H-atlas $(M,\pi,S)$. 
		\item\label{preHtg_lemma_item3}  
		Select $s\in S$, fix any $x_0\in\pi^{-1}(s)$ and consider the quotient map
		$$q\colon \Cg^\H_s(S)\to T^\Q_sS=\Cg^\H_s(S)/\sim.$$ 
		For every $v\in T_{x_0}M$ fix any curve $c_v\in \Ci(I_v,M)$ with $0\in I_v$, $c_v(0)=x_0$, and $\dot{c}_v(0)=v$, 
		and denote by 
		$$(T^\Q_{x_0}\pi)(v):=q(\pi\circ c_v)\in T^\Q_sS$$ 
		the equivalence class of $\pi\circ c_v\in \Cg^\H_s(S)$. 
		Then the mapping 
		$$T^\Q_{x_0}\pi\colon T_{x_0}M\to T^\Q_sS,\quad v\mapsto (T^\Q_{x_0}\pi)(v)$$
		is surjective and is independent of the choice of the curves $c_v$ for $v\in T_{x_0}M$. 
		\item\label{preHtg_lemma_item4} 
		For every $s\in S$ we have the surjective well-defined mapping 
		$$\Theta_s\colon T^\Q_s S\to T_sS,\quad \Theta_s(q_s(\gamma)):=(T_{c(0)}\pi)(\dot{c}(0))\in T_sS$$ 
		for every $\gamma\in\Cg^\H_s(S)$, any pre-\H-atlas $(M,\pi,S)$, any open interval $J\subseteq J_\gamma$ with $0\in J$, and any $c\in\Ci(J,M)$ with $\pi\circ c=\gamma\vert_J$.  
		\item\label{preHtg_lemma_item5} 
		For every pre-\H-atlas $\pi\colon M\to S$ and every $s\in S$ there exists a unique 
		mapping $\Lambda_s\colon T^\Q_sS\to T^\Fr_sS$ for which the diagram
			$$\xymatrix{\Cg^\H_s(S) \ar[d]_{q_s} \ar@{^{(}->}[r]&  (\Cg_M/\Rc_\pi)_s\ar[d]^{q_s^\Fr}\\
			T^\Q_sS \ar@{.>}[r]^{\Lambda}& T^\Fr_sS}$$
		is commutative. 
		\footnote{Here $(\Cg_M/\Rc_\pi)_s:=\Cg_s$ as in Definition~\ref{tgF} for the Fr\"olicher space $(S,\Cg,\Fg):=(S,\Cg_M/\Rc_\pi,\Fg_M/\Rc_\pi)$ constructed in Remark~\ref{FsmoothHsmooth}.} 
		\item\label{preHtg_lemma_item6}  
		If $(M,\pi,S)$ is a \Q-atlas, then the above mappings $T^\Q_{x_0}\pi$ and $\Theta_s$ are bijective. 
	\end{enumerate}
\end{proposition}

\begin{proof}
	\eqref{preHtg_lemma_item1} 
	This follows directly by Definition~\ref{preHtg_def}. 
	
	\eqref{preHtg_lemma_item2} 
	It is clear that the binary relation $\mathop{\sim}\limits^\pi$ is symmetric and reflexive, and it is also transitive as a direct consequence of the fact that any composition of transversal diffeomorphisms is again a transversal diffeomorphism. 
	
	In order to show that $\mathop{\sim}\limits^\pi$ does not depend on 
	the choice of the pre-\H-atlas $(M,\pi,S)$, 
	we apply~\eqref{preHtg_lemma_item1} for the pre-\H-atlas $(M,\pi,S)$ obtained as the disjoint union of two pre-\H-atlases $(M_i,\pi_i,S)$ for $i=1,2$, 
	and we thus easily prove that the equivalence relations $\mathop{\sim}\limits^{\pi_1}$ and $\mathop{\sim}\limits^{\pi_2}$ on $\Cg^\H_s(S)$ coincide. 
	
	\eqref{preHtg_lemma_item3} 	
	To see that $T_{x_0}\pi$ does not depend on the choice of the curves $c_v$, 
	it suffices to note that if $I\subseteq \RR$ is an open interval with $0\in I$ and $c^1_v,c^2_v\in\Ci(I,M)$ with $c^1_v(0)=c^2_v(0)=x_0$ and $\dot{c}^1_v(0)=\dot{c}^2_v(0)=v$, then $\pi\circ c^1_v\sim\pi\circ c^2_v$ 
	since the condition of Definition~\ref{preHtg_def} is satisfied for $h=\id_U$ for any open subset $U\subseteq M$ with $x_0\in U$. 
	
	We next define 
	$\widetilde{T_{x_0}\pi}\colon T_{x_0} M\to \Cg^\H_s(S)$, 
	$v\mapsto \pi\circ c_v$, 
	so that the diagram 
	$$\xymatrix{
		& \Cg^\H_s(S) \ar[d]^{q} \\
		T_{x_0}M \ar[ur]^{\widetilde{T^\Q_{x_0}\pi}} \ar[r]_{T_{x_0}\pi} & T^\Q_sS 
	}$$
	is commutative, that is, $T^\Q_{x_0}\pi=q\circ \widetilde{T_{x_0}\pi}$. 
	
	Then, by Definition~\ref{preHtg_def}, in order to prove that the mapping $T^\Q_{x_0}\pi$ is surjective, 
	it suffices to show that for every $\gamma\in\Cg^\H_s(S)$ there exists $v\in T_{x_0}M$ with $\widetilde{T_{x_0}\pi}(v)\sim \gamma$, 
	that is, $\pi\circ c_v\sim \gamma$. 
	Since $\gamma\in\Cg^\H_s(S)\subseteq\Hom_{\H}(J_\gamma,S)$ and 
	$0\in J_\gamma$, 
	it follows that there exist an open interval $J_\gamma^0\subseteq \RR$ 
	and $c\in\Ci(J_\gamma^0,M)$ 
	with $0\in J_\gamma^0\subseteq J_\gamma$, $c(0)=x_0$, and $\pi\circ c=\gamma\vert_{J_\gamma^0}$. 
	We now show that for $v:=\dot{c}(0)\in T_{x_0}M$ we have $\pi\circ c_v\sim \gamma$. 
	In fact, since $c(0)=x_0=c_v(0)$ and $\dot{c}(0)=v=\dot{c}_v(0)$, 
	it directly follows that $\pi\circ c_v\sim \pi\circ c=\gamma\vert_{J_\gamma^0}$. 
	(This is a very special case of the general remark made in the proof of \eqref{preHtg_lemma_item1} above.) 
	
	\eqref{preHtg_lemma_item4} 
	In order to prove that the mapping $\Theta_s$ is well defined, 
	we must show that if $\gamma_1,\gamma_2\in\Cg^\H_s(S)$ and there exist 
	a pre-\H-atlas $(M,\pi,S)$, 
	an open interval $J\subseteq \RR$ with $0\in J\subseteq J_{\gamma_1}\cap J_{\gamma_2}$, and,  $c_j\in\Ci(J,M)$ 
	with $\pi\circ c_j=\gamma_j\vert_J$ for $j=1,2$,
	as well as a transverse diffeomorphism $h\colon U_1\to U_2$ satisfying 
	$h(c_1(0))=c_2(0)$, 
	and $(T_{c_1(0)}h)(\dot{c}_1(0))=\dot{c}_2(0)$, 
	then $(T_{c_1(0)}\pi)(\dot{c}_1(0))=(T_{c_2(0)}\pi)(\dot{c}_2(0))$ ($\in T_sS$). 
	In fact, since $h$ is a transverse diffeomorphism, we have $\pi\circ h=\pi\vert_{U_1}$ hence 
	\begin{align*}
	(T_{c_1(0)}\pi)(\dot{c}_1(0))
	&=(T_{c_1(0)}(\pi\circ h))(\dot{c}_1(0))  \\
	&=((T_{c_2(0)}\pi)\circ(T_{c_1(0)}h))(\dot{c}_1(0)) \\
	&=(T_{c_2(0)}\pi)((T_{c_1(0)}h)(\dot{c}_1(0))) \\
	&=(T_{c_2(0)}\pi)(\dot{c}_2(0))
	\end{align*}
	as required. 
	
	In order to show that mapping $\Theta_s\colon T^Q_s S\to T_sS$ is surjective we note that for arbitrary $x_0\in \pi^{-1}(s)$, $v\in T_{x_0}M$, and $c_v\in\Ci(I_v,M)$ with $0\in I_v$, $c_v(0)=x_0$, and  $\dot{c}_v(0)=v$ we have by the definition of $T^\Q_{x_0}\pi$ and $\Theta_s$
	$$\Theta_s((T^\Q_{x_0}\pi)(v))
	=\Theta_s(q_s(\pi\circ c_v))=(T_{x_0}\pi)(\dot{c}_v(0))
	=(T_{x_0}\pi)(v).$$
	This shows that the diagram 
	\begin{equation}
	\label{preHtg_lemma_proof_eq0.5}
	\xymatrix{
		& T^Q_s S\ar[d]^{\Theta_s} \\
		T_{x_0}M \ar[ur]^{T^\Q_{x_0}\pi} \ar[r]_{T_{x_0}\pi} & T_sS 
	}
    \end{equation}
	is commutative, that is, 
	$\Theta_s\circ T^\Q_{x_0}\pi=T_{x_0}\pi$.
	Since the mapping $T_{x_0}\pi\colon T_{x_0}M\to T_sS$ from \eqref{tgH_def_eq1} is surjective, it the follows that $\Theta_s \colon T^Q_s S\to T_sS$ is surjective, too.

	\eqref{preHtg_lemma_item5}
	In order to show that there exists a mapping $\Lambda_s$ as in the statement, 
	we must check that if $\gamma_1,\gamma_2\in\Cg^\H_s(S)$ with $\gamma_1\mathop{\sim}\limits^\pi\gamma_2$ in the sense of Definition~\ref{preHtg_def}, then $\gamma_1\simeq\gamma_2$ in the sense of Definition~\ref{tgF}.
	To this end we recall that if  $\gamma_1\mathop{\sim}\limits^\pi\gamma_2$ then  $\gamma_j\vert_J=\pi\circ c_j$ with $c_j\in\Ci(J,M)$ as in Definition~\ref{preHtg_def}. 
	If $F\in\Fg_M/\Rc_\pi$ (cf. Remark~\ref{FsmoothHsmooth}), then $F\circ\pi\in\Ci(M,\RR)$, 
	and $F\circ \gamma_j\vert_J=(F\circ\pi)\circ c_j\in\Ci(J,\RR)$. 
	Since $\pi\circ h=\pi$ on a neighbourhood of $c_1(0)\in M$, 
	we have  $(F\circ\pi)\circ c_1=(F\circ\pi)\circ (h\circ c_1)$ 
	hence 
	\begin{align*}
	\frac{\de}{\de t}\Big\vert_{t=0}(F\circ \gamma_1)
	&=\frac{\de}{\de t}\Big\vert_{t=0}((F\circ\pi)\circ h\circ c_1) \\
	&=(T_{c_2(0)}(F\circ\pi))((T_{c_1(0)}h)(\dot{c}_1(0))) \\
	&=(T_{c_2(0)}(F\circ\pi))(\dot{c}_2(0)) \\
	&=\frac{\de}{\de t}\Big\vert_{t=0}((F\circ\pi)\circ c_2) \\
	&=\frac{\de}{\de t}\Big\vert_{t=0}(F\circ \gamma_2).
	\end{align*}
	Since $F\in\Fg_M/\Rc_\pi$ is arbitrary, we thus obtain $\gamma_1\simeq\gamma_2$ in the sense of Definition~\ref{tgF}.
		
	\eqref{preHtg_lemma_item6}
	We know from \eqref{preHtg_lemma_item3}  that the mapping $T^\Q_{x_0}\pi$ is surjective and we now use the hypothesis that $(M,\pi,S)$ is a Q-atlas in order to show that $T^\Q_{x_0}\pi$ is injective as well. 
	To this end we claim that if $v_1,v_2\in T_{x_0}M$ satisfy  $(T^\Q_{x_0}\pi)(v_1)=(T^\Q_{x_0}\pi)(v_2)$, then $v_1=v_2$. 
	The assumption is equivalent to $\pi\circ c_{v_1}\sim \pi\circ c_{v_2}$, 
	and this implies (by Definition~\ref{preHtg_def} for $\gamma_j:=\pi\circ c_{v_j}\in\Hom_{\H}(I_{v_j},S)$) that there exist
	an open interval $J\subseteq \RR$ with $0\in J\subseteq I_{v_1}\cap I_{v_2}$, and, for $j=1,2$ there exist $c_j\in\Ci(J,M)$ 
	with $\pi\circ c_j=\pi\circ c_{v_j}\vert_J$ for $j=1,2$, 
	and there exists a transverse diffeomorphism $h\colon U_1\to U_2$ satisfying 
	\begin{equation}
	\label{preHtg_lemma_proof_eq1}
	h(c_1(0))=c_2(0) 
	\text{ and }
	(T_{x_1}h)(\dot{c}_1(0))=\dot{c}_2(0).
	\end{equation} 
	On the other hand, since $\pi\circ c_j=\pi\circ c_{v_j}\vert_J$ and $c_{v_j}(0)=x_0$, it follows by Lemma~\ref{Qdot} that there exists a transversal diffeomorphism $h_j\colon V_j\to U'_j$ with 
	\begin{equation}
	\label{preHtg_lemma_proof_eq2}
	h_j(x_0)=c_j(0)
	\text{ and }
	(T_{x_0}h_j)(v_j)=(T_{x_0}h_j)(\dot{c}_{v_j}(0))=\dot{c}_j(0).
	\end{equation} 
	Replacing $V_1$ by $h_1^{-1}(U'_1\cap U_1)$, then $U_1$ by $U'_1\cap U_1$, 
	then $U_2$ by $h(U'_1\cap U_1)$, and finally $U'_2$ by $h(U'_1\cap U_1)$ 
	and $V_2$ by $h_2^{-1}(h(U'_1\cap U_1))$, we may assume $U'_j=U_j$ for $j=1,2$. 
	Denoting $k:=h_2^{-1}\circ h\circ h_1$, we thus obtain the 
	diagram 
	$$\xymatrix{
		J \ar[r]^{c_1}& U_1 \ar[r]^{h} & U_2 & J \ar[l]_{c_2}\\
		J \ar[r]_{c_{v_1}} & V_1\ar[u]^{h_1} \ar@{.>}[r]^{k} 
		& V_2 \ar[u]_{h_2} & J \ar[l]^{c_{v_2}}
	}$$
	where $c_{v_j}(0)=x_0\in V_1\cap V_2$ and $c_j(0)=h_j(x_0)\in U_j$ for $j=1,2$. 
	Therefore, using \eqref{preHtg_lemma_proof_eq1}--\eqref{preHtg_lemma_proof_eq2}, we obtain 
	$v_2=(T_{x_0}k)v_1$. 
	Here $k\colon V_1\to V_2$ is a transversal diffeomorphism satisfying $k(x_0)=x_0$, hence $h=\id$ on some neighbourhood of $x_0\in V_1$, 
	and therefore $v_2=v_1$, as claimed. 
	This completes the proof of the fact that the mapping $T^\Q_{x_0}\pi$ is bijective. 
	
	Moreover, since the \Q-chart $\pi$ is in particular an \H-chart 
	(Example~\ref{Q_ex}), 
	the mapping $T_{x_0}\colon T_{x_0}M\to T_sS$ is bijective by Proposition~\ref{preHtg}\eqref{preHtg_item2}. 
	It then follows by the commutative diagram ~\eqref{preHtg_lemma_proof_eq0.5} that the mapping $\Theta_s$ is bijective as well, and this concludes the proof. 
\end{proof}

If  $(M,\pi, \Gamma, S)$ is an \H-atlas, then $S$ has also a canonical structure of Fr\"olicher space (see Remark~\ref{FsmoothHsmooth}), hence we can define the  tangent space $T^\Fr S$ in the sense of Fr\"olicher spaces. 
We have noted prior to Proposition~\ref{FrolicherGr} that \emph{the tangent space at a point of a Fr\"olicher space is not a vector space in general}. 
However, in the special case of the Fr\"olicher spaces underlying an \H-manifold $S$ we have the following result, 
in which we use the maps $\Psi\colon TM\to T^\Fr M$ and $T^\Fr\pi\colon T^\Fr M\to T^\Fr S$ from Definition~\ref{tg-F} and Remark~\ref{classical-F}, respectively.

\begin{proposition}\label{propTS} 
	Let $(M,\pi, \Gamma, S)$ be a pre-\H-atlas. 
	If we regard $M$ and $S$ as Fr\"olicher spaces,  
	then the  map $\pi\colon M\to S$ is smooth in the sense of Fr\"olicher spaces. 
\end{proposition}

\begin{proof}
	The Fr\"olicher structure of the Banach manifold $M$ is $(M,\Fg_M,\Cg_M)$ as in Example~\ref{F_ex2}, where $\Fg_M=\Ci(M,\RR)$ and $\Ci(M,\RR)\subseteq\Cg_M$. 
	On the other hand, the Fr\"olicher structure of $S$ is $(S,\Fg_S,\Cg_S)$ as in Proposition~\ref{HversusF}, where $\Fg_S=\Hom_\H(S,\RR)$ and $\Hom_\H(\RR,S)\subseteq\Cg_S$. 
	In order to check that the map 
	$\pi\colon M\to S$ is smooth in the sense of Fr\"olicher spaces, 
	we must show that for every $c\in \Cg_M$ and $F\in\Fg_S$ we have $F\circ \pi\circ c\in\Ci(\RR,\RR)$, 
	cf. Definition~\ref{frolicher}. 
	To this end we first note that the hypothesis $F\in\Fg_S=\Hom_\H(S,\RR)$ implies $F\circ\pi\in \Ci(M,\RR)$. 
	Then, since $c\in \Cg_M$, we obtain $(F\circ \pi)\circ c\in\Ci(\RR,\RR)$ 
	by the definition of $\Cg_M$ in in Example~\ref{F_ex2}. 
	Thus the map 
	$\pi\colon M\to S$ is smooth in the sense of Fr\"olicher spaces.  
\end{proof}

\subsection*{Acknowledgment} 
We wish to thank Professor Raymond Barre for pointing out several important references 
and the Referee for several remarks that improved our manuscript. 
We acknowledge financial support from the Centre Francophone en Math\'ematiques de Bucarest and the GDRI ECO-Math.

\end{document}